\documentclass[reqno]{amsart}
\usepackage{amsmath, amsfonts, amsthm, amssymb, setspace, textcomp, bbm, multirow}
\usepackage{geometry}
\newtheorem{theorem}{Theorem}[section]
\newtheorem{lemma}[theorem]{Lemma}

\theoremstyle{definition}

\theoremstyle{remark}

\numberwithin{equation}{section}

\newcommand{\Parans}[1]{\left(#1\right)}

\newcommand{\SBrackets}[1]{\left[#1\right]}

\newcommand{\PieceTwo}[4]
{
	\left\{
   	\begin{array}{ll}
      	#1 & #3 \\
       	#2 & #4
     	\end{array}
	\right.
}
\newcommand{\aqprod}[3]{\Parans{#1;#2}_{#3}}
\newcommand{\jacprod}[2]{\SBrackets{#1;#2}_{\infty}}

\author{CHRIS JENNINGS-SHAFFER}
\address{Department of Mathematics, University of Florida\\
Gainesville, Florida 32611, USA
\endgraf cjenningsshaffer@ufl.edu}

\keywords{Number theory, partitions, vector partitions, congruences, Bailey pairs}

\subjclass[2010]{Primary 11P81, 11P83}

\title{Two partition functions with congruences modulo 3, 5, 7, and 13}

\allowdisplaybreaks
\begin{document}

\allowdisplaybreaks

\begin{abstract}
We introduce two new integer partition functions, both of which
are the number of partition quadruples of $n$ with certain size restrictions.
We prove both functions satisfy Ramanujan-type
congruences modulo $3$, $5$, $7$, and $13$ by use of generalized Lambert series
identities and $q$-series techniques.
\end{abstract}

\maketitle

\section{Introduction}
\allowdisplaybreaks

We recall a partition of a positive integer $n$ is a non-increasing
sequence of positive integers that sum to $n$. For example, the 
partitions of $5$ are $5$, $4+1$, $3+2$, $3+1+1$, $2+2+1$, $2+1+1+1$,
$1+1+1+1+1$. We let $p(n)$ denote the number of partitions of $n$. The function
$p(n)$ satisfies the well known congruences of Ramanujan
$p(5n+4)\equiv 0\pmod{5}$, $p(7n+5)\equiv 0\pmod{7}$, and 
$p(11n+6)\equiv 0\pmod{11}$. In this article we will consider two partition 
quadruples of $n$. We say a quadruple $(\pi_1,\pi_2,\pi_3,\pi_4)$ of partitions
is a partition quadruple of $n$ is altogether the parts of $\pi_1$, $\pi_2$,
$\pi_3$, and $\pi_4$ sum to $n$.

For a partition $\pi$, we let $s(\pi)$ denote the smallest part of $\pi$
and $\ell(\pi)$ denote the largest part of $\pi$. We use the conventions
that the empty partition has smallest part $\infty$ and largest part $0$. We 
let $u(n)$ denote the number of partition quadruples 
$(\pi_1,\pi_2,\pi_3,\pi_4)$ of $n$ such that $\pi_1$ is non-empty,
$s(\pi_1)\le s(\pi_2)$, $s(\pi_1)\le s(\pi_3)$, $s(\pi_1)\le s(\pi_4)$,
and $\ell(\pi_4)\le 2s(\pi_1)$.   
We let $v(n)$ denote the number of partition quadruples 
$(\pi_1,\pi_2,\pi_3,\pi_4)$ of $n$ such that the smallest part of $\pi_1$ 
appears at least twice,
$s(\pi_1)\le s(\pi_2)$, $s(\pi_1)\le s(\pi_3)$, $s(\pi_1)\le s(\pi_4)$,
and $\ell(\pi_4)\le 2s(\pi_1)$.   

We use the standard product notation,
\begin{align*}
	\aqprod{z}{q}{n} 
		&= \prod_{j=0}^{n-1} (1-zq^j)
	,
	&\aqprod{z}{q}{\infty} 
		&= \prod_{j=0}^\infty (1-zq^j)
	,\\
	\aqprod{z_1,\dots,z_k}{q}{n} 
		&= \aqprod{z_1}{q}{n}\dots\aqprod{z_k}{q}{n}
	,
	&\aqprod{z_1,\dots,z_k}{q}{\infty} 
		&= \aqprod{z_1}{q}{\infty}\dots\aqprod{z_k}{q}{\infty},
	\\
	\jacprod{z}{q} 
		&= \aqprod{z,q/z}{q}{\infty}
	,	
	&\jacprod{z_1,\dots,z_k}{q} &= \jacprod{z_1}{q}\dots\jacprod{z_k}{q}
.
\end{align*}
By summing according to $n$ being the smallest part of a partition, one easily
deduces that a generating function for $p(n)$ is given by
\begin{align*}
	F(q)
	&=
	\sum_{n=0}^\infty p(n)q^n
	=
	1+\sum_{n=1}^\infty \frac{q^n}{\aqprod{q^n}{q}{\infty}}
.
\end{align*}
Similarly, by summing according to $n$ being the smallest part of the
partition $\pi_1$, we find that generating functions for $u(n)$
and $v(n)$ are given by
\begin{align*}
	U(q)
	&=
	\sum_{n=0}^\infty u(n)q^n
	=
	\sum_{n=1}^\infty 
	\frac{q^n}{\aqprod{q^n}{q}{\infty}\aqprod{q^n}{q}
		{\infty}\aqprod{q^n}{q}{\infty}\aqprod{q^n}{q}{n+1}}
	,\\
	V(q)
	&=
	\sum_{n=0}^\infty v(n)q^n
	=
	\sum_{n=1}^\infty 
	\frac{q^{2n}}{\aqprod{q^n}{q}{\infty}\aqprod{q^n}{q}
		{\infty}\aqprod{q^n}{q}{\infty}\aqprod{q^n}{q}{n+1}}
.
\end{align*}

The main result of this article are the following congruences
for $u(n)$ and $v(n)$.
\begin{theorem}\label{TheoremCongruences}
\begin{align*}
	u(3n) &\equiv 0 \pmod{3}
	,\\	
	u(5n) &\equiv 0 \pmod{5}
	,\\	
	u(5n+3) &\equiv 0 \pmod{5}
	,\\	
	u(7n) &\equiv 0 \pmod{7}
	,\\	
	u(7n+5) &\equiv 0 \pmod{7}
	,\\	
	u(13n) &\equiv 0 \pmod{13}
	,\\	
	v(3n+1) &\equiv 0 \pmod{3}
	,\\	
	v(5n+1) &\equiv 0 \pmod{5}
	,\\	
	v(5n+4) &\equiv 0 \pmod{5}
	,\\	
	v(13n+10) &\equiv 0 \pmod{13}
.
\end{align*}
\end{theorem}

To prove Theorem \ref{TheoremCongruences} we use $q$-series techniques
and identities between generalized Lambert series to completely determine
$U(q)$ and $V(q)$ modulo $\ell$ for $\ell=3$, $5$, $7$, and $13$. These results are stated
in the following Theorem. For brevity we 
use the notation $E(a)=\aqprod{q^a}{q^a}{\infty}$ and 
$P(a) = [q^{\ell a};q^{\ell^2}]_\infty$. That is to say, 
in the modulo 3 congruences $P(a) = [q^{3a};q^9]_\infty$, 
in the modulo 5 congruences $P(a) = [q^{5a};q^{25}]_\infty$,
in the modulo 7 congruences $P(a) = [q^{7a};q^{49}]_\infty$, and
in the modulo 13 congruences $P(a) = [q^{13a};q^{169}]_\infty$.
We format the modulo 13 congruence differently from the other congruences due
to the large number of terms.
\begin{theorem}\label{TheoremMain}
\begin{align}
	\label{EqTheoremUMod3}
	U(q)
	&\equiv
		\frac{qE(9)^2}{E(3)P(1)}
		+
		\frac{2q^2}{E(9)P(1)} 
		\sum_{n=-\infty}^\infty 
			\frac{(-1)^n q^{\frac{9n^2+15n}{2}}  }{1-q^{9n+3}}
	\pmod{3}
	,\\
	\label{EqTheoremVMod3}
	V(q)
	&\equiv
		\frac{2q^3}{E(9)P(1) }
 		\sum_{n=-\infty}^\infty 
			\frac{(-1)^n q^{\frac{9n^2+15n}{2}}  }{1-q^{9n+6}}
		+
		\frac{q^2}{E(9)P(1)}
		\sum_{n=-\infty}^\infty 
			\frac{(-1)^n q^{\frac{9n^2+15n}{2}} }{1-q^{9n+3}}
	\pmod{3}
	,\\
	\label{EqTheoremUMod5}
	U(q)
	&\equiv
		\frac{qE(25)^2P(2)}{E(5)P(1)}	
		+
		\frac{4q^2}{E(25) P(1)} 
			\sum_{n=-\infty}^\infty \frac{(-1)^n q^{\frac{25n^2+35n}{2}} }{(1-q^{25n+5})}
		+
		\frac{q^2E(25)^2}{E(5)}	
		+
		\frac{4q^4}{E(25)P(1)}	
			\sum_{n=-\infty}^\infty \frac{(-1)^n q^{\frac{25n^2+35n}{2}} }{(1-q^{25n+10})}
	\nonumber\\&\quad	
	\pmod{5}
	,\\
	\label{EqTheoremVMod5}
	V(q)	
	&\equiv
		\frac{4q^5}{E(25)P(1) }
			\sum_{n=-\infty}^\infty \frac{(-1)^n q^{\frac{25n^2+35n}{2}} }{(1-q^{25n+15})}
		+
		\frac{q^2}{E(25)P(1)} 
			\sum_{n=-\infty}^\infty \frac{(-1)^n q^{\frac{25n^2+35n}{2}} }{(1-q^{25n+5})}
		+	
		\frac{4q^3 E(25)^2 P(1)}
			{E(5)P(2)}
	\pmod{5}
	,\\
	\label{EqTheoremUMod7}
	U(q)
	&\equiv	
		\frac{5q}{E(49)P(1)} 
			\sum_{n=-\infty}^\infty
			\frac{(-1)^n q^{\frac{49n^2+63n}{2} }}{1-q^{49n+7}}
		+
		\frac{3q E(49)^4P(2)^2 P(3)}{E(7) P(1)^3}
		+\frac{4q^8 E(49)^4P(2)^3}{E(7) P(1) P(3)^2}
		+\frac{3q^8 E(49)^4P(1) P(3)^2}{E(7)  P(2)^3}		
		\nonumber\\&\quad
		+
		\frac{4q^2E(49)^4P(3)^2}{E(7) P(1)^2}
		+\frac{q^2 E(49)^4P(2)^3}{E(7) P(1)^3}
		+\frac{q^9 E(49)^4P(1) P(3)}{E(7) P(2)^2}
		+\frac{2q^9 E(49)^4P(2)}{E(7) P(3)}
		+
		\frac{3q^3 E(49)^4P(2) P(3)}{E(7) P(1)^2}
		\nonumber\\&\quad				
		+\frac{4q^3 E(49)^4P(2)^4}{E(7) P(1)^3 P(3)}
		+\frac{q^3 E(49)^4P(3)^3}{E(7) P(2)^2 P(1)}
		+\frac{4q^{10} E(49)^4P(1)}{E(7) P(2)}
		+\frac{5q^{10} E(49)^4P(2)^2}{E(7) P(3)^2}
		\nonumber\\&\quad						
		+
		\frac{2q^4}{E(49)P(1)}
			\sum_{n=-\infty}^\infty
			\frac{(-1)^n q^{\frac{49n^2+63n}{2} }}{1-q^{49n+14}}
		+
		\frac{4q^6}{E(49)P(1)} 
			\sum_{n=-\infty}^\infty
			\frac{(-1)^n q^{\frac{49n^2+63n}{2} }}{1-q^{49n+21}}
		+
		\frac{2q^6 E(49)^4P(3)^2}{E(7) P(2)^2}
		\nonumber\\&\quad		
		+\frac{q^6 E(49)^4P(2)}{E(7) P(1)}
		+\frac{4q^{13} E(49)^4P(1)^2}{E(7) P(2) P(3)}
		+\frac{6q^{13} E(49)^4P(1) P(2)^2}{E(7) P(3)^3}	
	\pmod{7}
	,\\	
	\label{EqTheoremVMod7}
	V(q)	
	&\equiv
		\frac{6 q^7 }{E(49)P(1)}
			\sum_{n=-\infty}^\infty
			\frac{(-1)^n q^{\frac{49n^2+63n}{2} }}{1-q^{49n+28}}
		+
		\frac{2 q }{E(49)P(1)}
			\sum_{n=-\infty}^\infty
			\frac{(-1)^n q^{\frac{49n^2+63n}{2} }}{1-q^{49n+7}}
		+
		\frac{5q E(49)^4 P(2)^2 P(3)}{E(7) P(1)^3}		
		\nonumber\\&\quad		
		+
		\frac{3q^8 E(49)^4 }{E(7) }
		+
		\frac{6q^8 E(49)^4 P(2)^3}{E(7) P(1) P(3)^2}
		+\frac{q^{15} E(49)^4 P(1)^3 }{E(7) P(3) P(2)^2}
		+\frac{q^2 E(49)^4 P(2)^3}{E(7) P(1)^3}
		+\frac{3q^9 E(49)^4 P(3) P(1)}{E(7) P(2)^2}		
		\nonumber\\&\quad		
		+\frac{q^9 E(49)^4 P(2)}{E(7) P(3)}
		+\frac{4q^3 E(49)^4  P(2)^4}{E(7) P(1)^3 P(3)}
		+\frac{5q^{10} E(49)^4 P(1)}{E(7) P(2)}
		+\frac{4q^{10} E(49)^4 P(2)^2}{E(7) P(3)^2}
		\nonumber\\&\quad				
		+
		\frac{q^4 }{E(49)P(1)}
			\sum_{n=-\infty}^\infty
			\frac{(-1)^n q^{\frac{49n^2+63n}{2} }}{1-q^{49n+14}}
		+\frac{4q^5 E(49)^4 P(3)}{E(7) P(1)}
		+\frac{5q^{12} E(49)^4 P(1)^2}{E(7) P(2)^2}
		+\frac{q^{12} E(49)^4 P(1) P(2)}{E(7) P(3)^2}
		\nonumber\\&\quad		
		+
		\frac{5 q^6 }{E(49) P(1)}
			\sum_{n=-\infty}^\infty
			\frac{(-1)^n q^{\frac{49n^2+63n}{2} }}{1-q^{49n+21}}
		+\frac{2q^6 E(49)^4 P(2)}{E(7) P(1)}
		+\frac{6q^{13} E(49)^4 P(1)^2}{E(7) P(2) P(3)}
		+\frac{4q^{13} E(49)^4 P(1) P(2)^2}{E(7) P(3)^3}
		\nonumber\\&\quad
	\pmod{7}	
	,\\
	\label{EqTheoremUMod13}
	U(q)
	&\equiv
		\frac{12q^3 }{E(169) P(1)}
		\sum_{n=-\infty}^\infty \frac{(-1)^nq^{\frac{169n^2+195n}{2}}}{1-q^{169n+39}}			
		+
		\frac{10q^{-8} }{E(169) P(1)}
		\sum_{n=-\infty}^\infty \frac{(-1)^nq^{\frac{169n^2+195n}{2}}}{1-q^{169n+13}}		
		\nonumber\\&\quad		
		+
		\frac{11q^7 }{E(169) P(1)}
		\sum_{n=-\infty}^\infty \frac{(-1)^nq^{\frac{169n^2+195n}{2}}}{1-q^{169n+52}}		
		+
		\frac{q^{10} }{E(169) P(1)}
		\sum_{n=-\infty}^\infty \frac{(-1)^nq^{\frac{169n^2+195n}{2}}}{1-q^{169n+65}}		
		\nonumber\\&\quad		
		+
		\frac{8q^{-2} }{E(169) P(1)}
		\sum_{n=-\infty}^\infty \frac{(-1)^nq^{\frac{169n^2+195n}{2}}}{1-q^{169n+26}}		
		+
		\frac{4q^{12} }{E(169) P(1)}
		\sum_{n=-\infty}^\infty \frac{(-1)^nq^{\frac{169n^2+195n}{2}}}{1-q^{169n+78}}		
		+
		A_{13}(q)
	\pmod{13}
	,\\
	\label{EqTheoremVMod13}
	V(q)
	&\equiv
		\frac{12 q^{13} }{E(169) P(1)}
		\sum_{n=-\infty}^\infty \frac{(-1)^nq^{\frac{169n^2+195n}{2}}}{1-q^{169n+91}}			
		+
		\frac{11 q^3 }{E(169) P(1)}
		\sum_{n=-\infty}^\infty \frac{(-1)^nq^{\frac{169n^2+195n}{2}}}{1-q^{169n+39}}			
		\nonumber\\&\quad
		+
		\frac{3  }{q^8 E(169) P(1)}
		\sum_{n=-\infty}^\infty \frac{(-1)^nq^{\frac{169n^2+195n}{2}}}{1-q^{169n+13}}			
		+
		\frac{12 q^7 }{E(169) P(1)}
		\sum_{n=-\infty}^\infty \frac{(-1)^nq^{\frac{169n^2+195n}{2}}}{1-q^{169n+52}}			
		\nonumber\\&\quad
		+
		\frac{9  }{q^2 E(169) P(1)}
		\sum_{n=-\infty}^\infty \frac{(-1)^nq^{\frac{169n^2+195n}{2}}}{1-q^{169n+26}}			
		+
		\frac{5 q^{12}  }{E(169) P(1)}
		\sum_{n=-\infty}^\infty \frac{(-1)^nq^{\frac{169n^2+195n}{2}}}{1-q^{169n+78}}			
		\nonumber\\&\quad		
		+
		B_{13}(q)
	\pmod{13}
	,
\end{align}
where
\begin{align*}
	A_{13}(q) 
	&=
		\frac{E(169)^4}{E(13)}
		\bigg(
		A_{13,0}(q^{13})+qA_{13,1}(q^{13})+q^2A_{13,2}(q^{13})+q^3A_{13,3}(q^{13})
		+q^4A_{13,4}(q^{13})
		\\&\quad		
		+q^5A_{13,5}(q^{13})
		+q^6A_{13,6}(q^{13})
		+q^7A_{13,7}(q^{13})+q^8A_{13,8}(q^{13})+q^9A_{13,9}(q^{13})
		+q^{10}A_{13,10}(q^{13})
		\\&\quad		
		+q^{11}A_{13,11}(q^{13})
		+q^{12}A_{13,12}(q^{13})
		\bigg)
	,\\
	B_{13}(q) &= 
		\frac{E(169)^4}{E(13)}
		\bigg(
		B_{13,0}(q^{13})+qB_{13,1}(q^{13})+q^2B_{13,2}(q^{13})+q^3B_{13,3}(q^{13})
		+q^4B_{13,4}(q^{13})
		\\&\quad		
		+q^5B_{13,5}(q^{13})
		+q^6B_{13,6}(q^{13})
		+q^7B_{13,7}(q^{13})+q^8B_{13,8}(q^{13})+q^9B_{13,9}(q^{13})
		+q^{10}B_{13,10}(q^{13})
		\\&\quad		
		+q^{11}B_{13,11}(q^{13})
		+q^{12}B_{13,12}(q^{13})
		\bigg)
	,\\
	A_{13,0}(q^{13})
	&=
		0
	,\\
	A_{13,1}(q^{13})
	&=
		\tfrac{P(3)^3 P(4)^6}{P(1)^2 P(5)^2 P(2) P(6)}
		+\tfrac{9q^{13} P(3)^5 P(4)^3}{P(1)^2 P(5)^2 P(6)}
		+\tfrac{10q^{26} P(3)^{10}}{P(1) P(5)^2 P(2)^2 P(6) P(4)}
		+\tfrac{6q^{26} P(3)^7 P(2)}{P(1)^2 P(5)^2 P(6)}		
		+\tfrac{6q^{39} P(3)^2 P(2)^5 P(4)}{P(1)^2 P(5)^2 P(6)}
		\\&\quad		
		+\tfrac{7q^{39} P(3)^5 P(2)^2}{P(1) P(5)^2 P(6)}
		+\tfrac{7q^{52} P(3)^3 P(2)^3}{P(5)^2 P(6)}
		+\tfrac{3q^{52} P(1) P(3)^6}{P(5)^2 P(6) P(4)}	
		+\tfrac{9q^{65} P(1) P(3) P(2)^4}{P(5)^2 P(6)}
		+\tfrac{12q^{65} P(1)^2 P(3)^4 P(2)}{P(5)^2 P(6) P(4)}
		\\&\quad		
		+\tfrac{9q^{78} P(1)^3 P(3)^2 P(2)^2}{P(5)^2 P(6) P(4)}
		+\tfrac{10q^{91} P(1)^6 P(4)}{P(5)^2 P(3) P(6)}	
		+\tfrac{9q^{91} P(1)^4 P(2)^3}{P(5)^2 P(6) P(4)}		
		+\tfrac{3q^{104} P(1)^8}{P(5)^2 P(2)^2 P(6)}
	,\\	
	A_{13,2}(q^{13})
	&=
		\tfrac{5 P(3)^3 P(4)^6}{P(1)^2 P(5) P(2) P(6)^2}	
		+\tfrac{10q^{13} P(3)^5 P(4)^3}{P(1)^2 P(5) P(6)^2}
		+\tfrac{8q^{26} P(3)^6 P(4)^2}{P(5) P(2)^2 P(6)^2}
		+\tfrac{q^{26} P(3)^7 P(2)}{P(1)^2 P(5) P(6)^2}
		+\tfrac{10q^{26} P(3)^{10}}{P(1) P(5) P(4) P(2)^2 P(6)^2}
		\\&\quad		
		+\tfrac{7q^{39} P(3)^2 P(4) P(2)^5}{P(1)^2 P(5) P(6)^2}
		+\tfrac{11q^{39} P(3)^5 P(2)^2}{P(1) P(5) P(6)^2}
		+\tfrac{6q^{52} P(4) P(2)^6}{P(1) P(5) P(6)^2}
		+\tfrac{9q^{52} P(3)^3 P(2)^3}{P(5) P(6)^2}
		+\tfrac{8q^{52} P(1) P(3)^6}{P(5) P(4) P(6)^2}
		\\&\quad		
		+\tfrac{3q^{65} P(1) P(3) P(2)^4}{P(5) P(6)^2}
		+\tfrac{11q^{65} P(1)^2 P(3)^4 P(2)}{P(5) P(4) P(6)^2}
		+\tfrac{12q^{78} P(1)^3 P(3)^2 P(2)^2}{P(5) P(4) P(6)^2}
		+\tfrac{10q^{91} P(1)^6 P(4)}{P(5) P(3) P(6)^2}
		+\tfrac{3q^{91} P(1)^4 P(2)^3}{P(5) P(4) P(6)^2}
		\\&\quad		
		+\tfrac{3q^{104} P(1)^8}{P(5) P(2)^2 P(6)^2}
	,\\	
	A_{13,3}(q^{13})
	&=
		\tfrac{3 P(4)^7 P(3)}{P(1)^2 P(5) P(6)^2}
		+\tfrac{8q^{13} P(4)^4 P(3)^3 P(2)}{P(1)^2 P(5) P(6)^2}
		+\tfrac{5q^{26} P(4)^3 P(3)^4}{P(5) P(6)^2 P(2)}
		+\tfrac{2q^{26} P(3)^8}{P(1) P(5) P(6)^2 P(2)}
		+\tfrac{10q^{39} P(4) P(3)^3 P(2)^3}{P(1) P(5) P(6)^2}
		\\&\quad		
		+\tfrac{10q^{39} P(3)^6}{P(5) P(6)^2}
		+\tfrac{11q^{52} P(4) P(3) P(2)^4}{P(5) P(6)^2}
		+\tfrac{5q^{52} P(1) P(3)^4 P(2)}{P(5) P(6)^2}
		+\tfrac{q^{52} P(3)^2 P(2)^7}{P(1)^2 P(4) P(5) P(6)^2}
		+\tfrac{9q^{65} P(1) P(4) P(2)^5}{P(5) P(6)^2 P(3)}
		\\&\quad		
		+\tfrac{5q^{65} P(1)^2 P(3)^2 P(2)^2}{P(5) P(6)^2}
		+\tfrac{12q^{65} P(2)^8}{P(1) P(4) P(5) P(6)^2}
		+\tfrac{10q^{78} P(1)^3 P(2)^3}{P(5) P(6)^2}
		+\tfrac{11q^{78} P(1)^4 P(3)^3}{P(4) P(5) P(6)^2}
		+\tfrac{5q^{91} P(1)^5 P(3) P(2)}{P(4) P(5) P(6)^2}
		\\&\quad		
		+\tfrac{10q^{104} P(1)^6 P(2)^2}{P(4) P(5) P(6)^2 P(3)}
	,\\	
	A_{13,4}(q^{13})
	&=
		\tfrac{5 P(4)^8 P(2)}{P(1)^2 P(5) P(6)^2 P(3)}	
		+\tfrac{4q^{13} P(4)^4 P(3)^4}{P(1) P(5) P(6)^2 P(2)}
		+\tfrac{4q^{26} P(4) P(3)^6}{P(1) P(5) P(6)^2}
		+\tfrac{5q^{39} P(4)^2 P(3) P(2)^4}{P(1) P(5) P(6)^2}
		+\tfrac{11q^{39} P(4) P(3)^4 P(2)}{P(5) P(6)^2}
		\\&\quad		
		+\tfrac{7q^{52} P(2)^7}{P(1) P(6)^2 P(3)}
		+\tfrac{11q^{52} P(1) P(4) P(3)^2 P(2)^2}{P(5) P(6)^2}
		+\tfrac{7q^{52} P(1)^2 P(3)^5}{P(5) P(6)^2 P(2)}
		+\tfrac{6q^{52} P(2)^8}{P(1)^2 P(5) P(6)^2}
		+\tfrac{4q^{52} P(3)^3 P(2)^5}{P(1) P(4) P(5) P(6)^2}
		\\&\quad		
		+\tfrac{5q^{65} P(1)^4 P(4)^3}{P(5) P(6)^2 P(3)}
		+\tfrac{2q^{65} P(1)^2 P(4) P(2)^3}{P(5) P(6)^2}
		+\tfrac{7q^{65} P(1)^3 P(3)^3}{P(5) P(6)^2}
		+\tfrac{9q^{65} P(3) P(2)^6}{P(4) P(5) P(6)^2}				
		+\tfrac{3q^{78} P(1)^4 P(3) P(2)}{P(5) P(6)^2}
		\\&\quad		
		+\tfrac{7q^{78} P(1) P(2)^7}{P(4) P(5) P(6)^2 P(3)}
		+\tfrac{7q^{91} P(1)^5 P(2)^2}{P(5) P(6)^2 P(3)}
		+\tfrac{6q^{104} P(1)^7}{P(4) P(5) P(6)^2}
	,\\	
	A_{13,5}(q^{13})
	&=
		\tfrac{3 P(4)^9}{q^{13} P(5)^2 P(6) P(1)^2  P(2)}		
		+\tfrac{11 P(4)^6 P(3)^2}{P(5)^2 P(6) P(1)^2}
		+\tfrac{8q^{13} P(4)^2 P(3)^7}{P(5)^2 P(6) P(1) P(2)^2}
		+\tfrac{11q^{26} P(4)^2 P(3)^5}{P(5)^2 P(6) P(2)}
		+\tfrac{q^{26}P(3)^6 P(2)^2}{P(5)^2 P(6) P(1)^2}
		\\&\quad		
		+\tfrac{3q^{26} P(3)^9}{P(5)^2 P(4) P(6) P(1) P(2)}
		+\tfrac{q^{39}P(4) P(3) P(2)^6}{P(5)^2 P(6) P(1)^2}
		+\tfrac{9q^{39} P(3)^4 P(2)^3}{P(5)^2 P(6) P(1)}
		+\tfrac{5q^{39} P(3)^7}{P(5)^2 P(4) P(6)}
		+\tfrac{12q^{52} P(4) P(2)^7}{P(5)^2 P(6) P(1) P(3)}
		\\&\quad		
		+\tfrac{12q^{52} P(3)^2 P(2)^4}{P(5)^2 P(6)}
		+\tfrac{10q^{65} P(1) P(2)^5}{P(5)^2 P(6)}
		+\tfrac{9q^{65} P(1)^2 P(3)^3 P(2)^2}{P(5)^2 P(4) P(6)}
		+\tfrac{4q^{78} P(1)^3 P(3) P(2)^3}{P(5)^2 P(4) P(6)}
		+\tfrac{12q^{91} P(1)^4 P(2)^4}{P(5)^2 P(4) P(6) P(3)}
		\\&\quad		
		+\tfrac{10q^{104} P(1)^8}{P(5)^2 P(6) P(3) P(2)}
	,\\	
	A_{13,6}(q^{13})
	&=
		\tfrac{5 P(3)^7 P(4)^3}{P(1)^2 P(6) P(5)^2 P(2)^2}
		+\tfrac{q^{13} P(3)^5 P(4)^3}{P(1) P(6) P(5)^2 P(2)}
		+\tfrac{8q^{13} P(3)^9}{P(1)^2 P(6) P(5)^2 P(2)}
		+\tfrac{5q^{26} P(3)^7}{P(1) P(6) P(5)^2}
		+\tfrac{9q^{39} P(1) P(3)^8}{P(4) P(6) P(5)^2 P(2)^2}
		\\&\quad		
		+\tfrac{7q^{39} P(3)^6 P(2)^4}{P(1)^2 P(4)^2 P(6) P(5)^2}
		+\tfrac{3q^{52} P(1) P(3)^3 P(2)^2}{P(6) P(5)^2}
		+\tfrac{3q^{52} P(1)^2 P(3)^6}{P(4) P(6) P(5)^2 P(2)}
		+\tfrac{4q^{52} P(3)^4 P(2)^5}{P(1) P(4)^2 P(6) P(5)^2}
		\\&\quad		
		+\tfrac{12q^{65} P(1)^2 P(3) P(2)^3}{P(6) P(5)^2}
		+\tfrac{12q^{65} P(1)^3 P(3)^4}{P(4) P(6) P(5)^2}
		+\tfrac{2q^{78} P(1)^3 P(2)^4}{P(3) P(6) P(5)^2}
		+\tfrac{2q^{78} P(1) P(2)^7}{P(4)^2 P(6) P(5)^2}
		+\tfrac{5q^{91} P(1)^5 P(2)^2}{P(4) P(6) P(5)^2}
		\\&\quad		
		+\tfrac{4q^{91} P(1)^6 P(3)^3}{P(4)^2 P(6) P(5)^2 P(2)}
		+\tfrac{6q^{104} P(1)^7 P(3)}{P(4)^2 P(6) P(5)^2}
	,\\	
	A_{13,7}(q^{13})
	&=
		\tfrac{8P(3)^4 P(4)^6}{P(1)^2 P(6)^2 P(5)^2 P(2)}
		+\tfrac{3q^{13} P(3)^6 P(4)^3}{P(1)^2 P(6)^2 P(5)^2}
		+\tfrac{7q^{26} P(3)^4 P(2) P(4)^3}{P(1) P(6)^2 P(5)^2}
		+\tfrac{2q^{26} P(3)^8 P(2)}{P(1)^2 P(6)^2 P(5)^2}
		+\tfrac{8q^{39} P(1) P(3)^5 P(4)^2}{P(6)^2 P(5)^2 P(2)}
		\\&\quad		
		+\tfrac{12q^{39} P(3)^3 P(2)^5 P(4)}{P(1)^2 P(6)^2 P(5)^2}
		+\tfrac{7q^{39} P(3)^6 P(2)^2}{P(1) P(6)^2 P(5)^2}
		+\tfrac{9q^{52} P(3) P(2)^6 P(4)}{P(1) P(6)^2 P(5)^2}
		+\tfrac{9q^{52} P(3)^4 P(2)^3}{P(6)^2 P(5)^2}
		+\tfrac{9q^{65} P(1) P(3)^2 P(2)^4}{P(6)^2 P(5)^2}
		\\&\quad	
		+\tfrac{5q^{65} P(1)^2 P(3)^5 P(2)}{P(6)^2 P(5)^2 P(4)}
		+\tfrac{6q^{78} P(1)^3 P(3)^3 P(2)^2}{P(6)^2 P(5)^2 P(4)}
		+\tfrac{12q^{91} P(1)^4 P(3) P(2)^3}{P(6)^2 P(5)^2 P(4)}
	,\\	
	A_{13,8}(q^{13})
	&=
		\tfrac{4 P(3)^6 P(4)^2}{P(1)^2 P(5) P(2) P(6)}
		+\tfrac{4q^{13} P(3)^4 P(4)^2}{P(1) P(5) P(6)}
		+\tfrac{9q^{13} P(3)^8}{P(1)^2 P(5) P(6) P(4)}
		+\tfrac{10q^{26} P(2) P(3)^2 P(4)^2}{P(5) P(6)}
		+\tfrac{3q^{26} P(2)^4 P(3)^3}{P(1)^2 P(5) P(6)}
		\\&\quad		
		+\tfrac{6q^{26} P(2) P(3)^6}{P(1) P(5) P(6) P(4)}
		+\tfrac{4q^{39} P(2)^4}{P(6)}
		+\tfrac{4q^{39} P(1)^2 P(3)^3 P(4)}{P(5) P(2) P(6)}
		+\tfrac{12q^{39} P(2)^5 P(3)}{P(1) P(5) P(6)}
		+\tfrac{4q^{39} P(2)^2 P(3)^4}{P(5) P(6) P(4)}
		\\&\quad		
		+\tfrac{8q^{52} P(1)^3 P(3) P(4)}{P(5) P(6)}
		+\tfrac{7q^{52} P(1) P(2)^3 P(3)^2}{P(5) P(6) P(4)}
		+\tfrac{11q^{65} P(1)^2 P(2)^4}{P(5) P(6) P(4)}
	,\\	
	A_{13,9}(q^{13})
	&=
		\tfrac{10 P(4)^2 P(3)^8}{P(1)^2 P(6) P(2)^2 P(5)^2}		
		+\tfrac{10q^{13} P(4)^2 P(3)^6}{P(1) P(6) P(2) P(5)^2}
		+\tfrac{3q^{13} P(3)^{10}}{P(1)^2 P(6) P(4) P(2) P(5)^2}
		+\tfrac{10q^{26} P(4)^2 P(3)^4}{P(6) P(5)^2}
		\\&\quad				
		+\tfrac{3q^{26} P(3)^8}{P(1) (6) P(4) P(5)^2}
		+\tfrac{11q^{39} P(2)^4 P(3)^3}{P(1) P(6) P(5)^2}
		+\tfrac{11q^{39} P(2)P(3)^6}{P(6) P(4) P(5)^2}
		+\tfrac{11q^{52} P(2)^5 P(3)}{P(6) P(5)^2}
		+\tfrac{9q^{52} P(1) P(2)^2 P(3)^4}{P(6) P(4) P(5)^2}
		\\&\quad		
		+\tfrac{7q^{52} P(2)^8 P(3)^2}{P(1)^2 P(6) P(4)^2 P(5)^2}
		+\tfrac{6q^{65} P(1) P(2)^6}{P(6) P(5)^2 P(3)}
		+\tfrac{10q^{65} P(1)^2 P(2)^3 P(3)^2}{P(6) P(4) P(5)^2}
		+\tfrac{10q^{65} P(1)^3 P(3)^5}{P(6) P(4)^2 P(5)^2}
		\\&\quad		
		+\tfrac{6q^{65} P(2)^9}{P(1) P(6) P(4)^2 P(5)^2}
		+\tfrac{5q^{78} P(1)^3 P(2)^4}{P(6) P(4) P(5)^2}			
		+\tfrac{3q^{78} P(1)^4 P(2) P(3)^3}{P(6) P(4)^2 P(5)^2}
		+\tfrac{2q^{91} P(1)^5 P(2)^2 P(3)}{P(6) P(4)^2 P(5)^2}
	,\\	
	A_{13,10}(q^{13})
	&=
		\tfrac{8P(4)^3 P(3)^6}{P(1)^2 P(2) P(6) P(5)^2}		
		+\tfrac{q^{13} P(4)^3 P(3)^4}{P(1) P(6) P(5)^2}
		+\tfrac{5q^{13} P(3)^8}{P(1)^2 P(6) P(5)^2}
		+\tfrac{2q^{26} P(2) P(3)^6}{P(1) P(6) P(5)^2}
		+\tfrac{7q^{39} P(2)^2 P(3)^4}{P(6) P(5)^2}
		\\&\quad		
		+\tfrac{7q^{39} P(1) P(3)^7}{P(4) P(2) P(6) P(5)^2}
		+\tfrac{10q^{39} P(2)^5 P(3)^5}{P(4)^2 P(1)^2 P(6) P(5)^2}
		+\tfrac{3q^{52} P(1) P(2)^3 P(3)^2}{P(6) P(5)^2}
		+\tfrac{7q^{52} P(1)^2 P(3)^5}{P(4) P(6) P(5)^2}
		\\&\quad		
		+\tfrac{12q^{52} P(2)^6 P(3)^3}{P(4)^2 P(1) P(6) P(5)^2}
		+\tfrac{5q^{65} P(1)^2 P(2)^4}{P(6) P(5)^2}
		+\tfrac{7q^{65} P(1)^3 P(2) P(3)^3}{P(4) P(6) P(5)^2}
		+\tfrac{4q^{65} P(2)^7 P(3)}{P(4)^2 P(6) P(5)^2}
		+\tfrac{10q^{78} P(1)^4 P(2)^2 P(3)}{P(4) P(6) P(5)^2}
		\\&\quad	
		+\tfrac{6q^{91} P(1)^6 P(3)^2}{P(4)^2 P(6) P(5)^2}
	,\\	
	A_{13,11}(q^{13})
	&=
		\tfrac{5 P(3) P(4)^5}{q^{13}P(1)^2  P(6)}
		+\tfrac{5q^{13} P(3)^5 P(4)}{P(1) P(5) P(6)}
		+\tfrac{12q^{26} P(2) P(3)^3 P(4)}{P(5) P(6)}
		+\tfrac{9q^{26} P(2)^4 P(3)^4}{P(1)^2 P(5) P(4) P(6)}
		+\tfrac{7q^{26} P(2) P(3)^7}{P(1) P(5) P(4)^2 P(6)}	
		\\&\quad		
		+\tfrac{6q^{39} P(2)^8}{P(3) P(1)^2 P(5) P(6)}
		+\tfrac{5q^{39} P(2)^5 P(3)^2}{P(1) P(5) P(4) P(6)}
		+\tfrac{3q^{39} P(2)^2 P(3)^5}{P(5) P(4)^2 P(6)}
		+\tfrac{6q^{52} P(3)^2 P(1)^3}{P(5) P(6)}
		+\tfrac{5q^{52} P(2)^3 P(3)^3 P(1)}{P(5) P(4)^2 P(6)}
		\\&\quad		
		+\tfrac{11q^{65} P(2)^4 P(3) P(1)^2}{P(5) P(4)^2 P(6)}
		+\tfrac{6q^{78} P(2)^5 P(1)^3}{P(3) P(5) P(4)^2 P(6)}
	,\\	
	A_{13,12}(q^{13})
	&=
		\tfrac{2P(4)^2 P(3)^5}{P(1)^2 P(5) P(6)}	
		+\tfrac{4q^{13} P(4) P(3)^6}{P(2)^2 P(5) P(6)}
		+\tfrac{11q^{13} P(3)^7 P(2)}{P(4) P(1)^2 P(5) P(6)}
		+\tfrac{5q^{26} P(2)^2 P(4)^2 P(3)}{P(5) P(6)}
		+\tfrac{4q^{26} P(4) P(3)^4 P(1)}{P(2) P(5) P(6)}
		\\&\quad		
		+\tfrac{4q^{26} P(3)^2 P(2)^5}{P(1)^2 P(5) P(6)}
		+\tfrac{5q^{26} P(3)^5 P(2)^2}{P(4) P(1) P(5) P(6)}
		+\tfrac{q^{39} P(2)^5}{P(3) P(6)}
		+\tfrac{10q^{39} P(4) P(3)^2 P(1)^2}{P(5) P(6)}
		+\tfrac{10q^{39} P(2)^6}{P(1) P(5) P(6)}
		\\&\quad		
		+\tfrac{12q^{39} P(3)^3 P(2)^3}{P(4) P(5) P(6)}
		+\tfrac{4q^{52} P(4) P(1)^3 P(2)}{P(5) P(6)}
		+\tfrac{9q^{52} P(3) P(1) P(2)^4}{P(4) P(5) P(6)}
		+\tfrac{9q^{65} P(3) P(1)^5}{P(2) P(5) P(6)}
		+\tfrac{q^{65} P(1)^2 P(2)^5}{P(4) P(3) P(5) P(6)}
	,\\
	B_{13,0}(q^{13})
	&=
	 \tfrac{6q^{13} P(4)^5 P(3)^2 P(2) }{P(1)^2 P(6)^2 P(5)}
	 +\tfrac{7q^{26} P(4) P(3)^7 }{P(1) P(6)^2 P(2) P(5)}
	 +\tfrac{6q^{39} P(4) P(3)^5 }{P(6)^2 P(5)}
	 +\tfrac{10q^{39} P(3)^6 P(2)^3 }{P(4) P(1)^2 P(6)^2 P(5)}
	 +\tfrac{3q^{39} P(3)^9 }{P(4)^2 P(1) P(6)^2 P(5)}
	\\&\quad	 
	 +\tfrac{2q^{52} P(3) P(2)^7 }{P(1)^2 P(6)^2 P(5)}
	 +\tfrac{12q^{52} P(3)^4 P(2)^4 }{P(4) P(1) P(6)^2 P(5)}
	 +\tfrac{6q^{52} P(3)^7 P(2) }{P(4)^2 P(6)^2 P(5)}
	 +\tfrac{10q^{65} P(1) P(3)^5 P(2)^2 }{P(4)^2 P(6)^2 P(5)}
	 +\tfrac{12q^{78} P(1)^2 P(3)^3 P(2)^3 }{P(4)^2 P(6)^2 P(5)}
	\\&\quad
	 +\tfrac{7q^{91} P(1)^3 P(3) P(2)^4 }{P(4)^2 P(6)^2 P(5)}
	 +\tfrac{3q^{104} P(1)^6 P(2)^2 }{P(3)^2 P(6)^2 P(5)}
	 +\tfrac{3q^{104} P(1)^4 P(2)^5 }{P(4)^2 P(3) P(6)^2 P(5)}
	 +\tfrac{10q^{117} P(1)^8 }{P(4) P(3) P(6)^2 P(5)}
	,\\
	B_{13,1}(q^{13})
	&=
	 \tfrac{5q^{13} P(4)^3 P(3)^4 }{P(6) P(1) P(5) P(2)}
	 +\tfrac{9q^{26} P(4)^3 P(3)^2 }{P(6) P(5)}
	 +\tfrac{12q^{26} P(4) P(3)^3 P(2)^3 }{P(6) P(1)^2 P(5)}
	 +\tfrac{9q^{26} P(3)^6 }{P(6) P(1) P(5)}
	 +\tfrac{10q^{39} P(4) P(3) P(2)^4 }{P(6) P(1) P(5)}
	\\&\quad	 
	 +\tfrac{7q^{39} P(3)^4 P(2) }{P(6) P(5)}
	 +\tfrac{10q^{52} P(1) P(3)^2 P(2)^2 }{P(6) P(5)}
	 +\tfrac{4q^{65} P(1)^2 P(2)^3 }{P(6) P(5)}
	 +\tfrac{7q^{78} P(1)^4 P(3) P(2) }{P(6) P(4) P(5)}
	,\\
	B_{13,2}(q^{13})
	&=
	 \tfrac{P(3)^3 P(4)^6}{P(6)^2 P(2) P(1)^2 P(5)}
	 +\tfrac{q^{13}P(3)^5 P(4)^3 }{P(6)^2 P(1)^2 P(5)}
	 +\tfrac{7q^{26} P(3)^6 P(4)^2 }{P(6)^2 P(2)^2 P(5)}
	 +\tfrac{2q^{26} P(2) P(3)^7 }{P(6)^2 P(1)^2 P(5)}
	 +\tfrac{9q^{26} P(3)^{10} }{P(6)^2 P(2)^2 P(1) P(4) P(5)}
	\\&\quad	 
	 +\tfrac{11q^{39} P(2)^5 P(3)^2 P(4) }{P(6)^2 P(1)^2 P(5)}
	 +\tfrac{8q^{39} P(2)^2 P(3)^5 }{P(6)^2 P(1) P(5)}
	 +\tfrac{6q^{52} P(2)^6 P(4) }{P(6)^2 P(1) P(5)}
	 +\tfrac{6q^{52} P(2)^3 P(3)^3 }{P(6)^2 P(5)}
	 +\tfrac{10q^{52} P(1) P(3)^6 }{P(6)^2 P(4) P(5)}
	\\&\quad
	 +\tfrac{8q^{65} P(2)^4 P(1) P(3) }{P(6)^2 P(5)}
	 +\tfrac{q^{65}P(2) P(1)^2 P(3)^4 }{P(6)^2 P(4) P(5)}
	 +\tfrac{q^{78}P(2)^2 P(1)^3 P(3)^2 }{P(6)^2 P(4) P(5)}
	 +\tfrac{9q^{91} P(1)^6 P(4) }{P(6)^2 P(3) P(5)}
	 +\tfrac{2q^{91} P(2)^3 P(1)^4 }{P(6)^2 P(4) P(5)}
	\\&\quad
	 +\tfrac{4q^{104} P(1)^8 }{P(6)^2 P(2)^2 P(5)}
	,\\
	B_{13,3}(q^{13})
	&=
	 \tfrac{6 P(3) P(4)^7}{P(6)^2 P(1)^2 P(5)}
	 +\tfrac{3q^{13} P(2) P(3)^3 P(4)^4 }{P(6)^2 P(1)^2 P(5)}
	 +\tfrac{4q^{26} P(3)^8 }{P(2) P(6)^2 P(1) P(5)}
	 +\tfrac{10q^{26} P(3)^4 P(4)^3 }{P(2) P(6)^2 P(5)}
	 +\tfrac{7q^{39} P(2)^3 P(3)^3 P(4) }{P(6)^2 P(1) P(5)}
	\\&\quad	 
	 +\tfrac{7q^{39} P(3)^6 }{P(6)^2 P(5)}
	 +\tfrac{9q^{52} P(2)^4 P(3) P(4) }{P(6)^2 P(5)}
	 +\tfrac{10q^{52} P(2) P(3)^4 P(1) }{P(6)^2 P(5)}
	 +\tfrac{2q^{52} P(2)^7 P(3)^2 }{P(6)^2 P(1)^2 P(5) P(4)}
	 +\tfrac{5q^{65} P(2)^5 P(1) P(4) }{P(3) P(6)^2 P(5)}
	\\&\quad
	 +\tfrac{10q^{65} P(2)^2 P(3)^2 P(1)^2 }{P(6)^2 P(5)}
	 +\tfrac{11q^{65} P(2)^8 }{P(6)^2 P(1) P(5) P(4)}
	 +\tfrac{7q^{78} P(2)^3 P(1)^3 }{P(6)^2 P(5)}
	 +\tfrac{9q^{78} P(3)^3 P(1)^4 }{P(6)^2 P(5) P(4)}
	 +\tfrac{10q^{91} P(2) P(3) P(1)^5 }{P(6)^2 P(5) P(4)}
	 \\&\quad
	 +\tfrac{7q^{104} P(2)^2 P(1)^6 }{P(3) P(6)^2 P(5) P(4)}
	,\\
	B_{13,4}(q^{13})
	&=
	 \tfrac{2 P(2) P(4)^8}{P(3) P(6)^2 P(1)^2 P(5)}
	 +\tfrac{q^{13} P(3)^4 P(4)^4 }{P(2) P(6)^2 P(1) P(5)}
	 +\tfrac{4q^{26} P(3)^6 P(4) }{P(6)^2 P(1) P(5)}
	 +\tfrac{12q^{39} P(2) P(3)^4 P(4) }{P(6)^2 P(5)}
	 +\tfrac{6q^{39} P(2)^4 P(3)^5 }{P(6)^2 P(1)^2 P(5) P(4)}
	\\&\quad	 
	 +\tfrac{10q^{52} P(2)^2 P(3)^2 P(1) P(4) }{P(6)^2 P(5)}
	 +\tfrac{11q^{52} P(3)^5 P(1)^2 }{P(2) P(6)^2 P(5)}
	 +\tfrac{q^{52} P(2)^5 P(3)^3 }{P(6)^2 P(1) P(5) P(4)}
	 +\tfrac{12q^{65} P(2)^3 P(1)^2 P(4) }{P(6)^2 P(5)}
	\\&\quad
	 +\tfrac{10q^{65} P(2)^6 P(3) }{P(6)^2 P(5) P(4)}
	 +\tfrac{7q^{78} P(2)^4 P(1)^3 P(4) }{P(3)^2 P(6)^2 P(5)}
	 +\tfrac{2q^{78} P(2) P(3) P(1)^4 }{P(6)^2 P(5)}
	 +\tfrac{9q^{78} P(2)^7 P(1) }{P(3) P(6)^2 P(5) P(4)}
	 +\tfrac{4q^{91} P(2)^2 P(1)^5 }{P(3) P(6)^2 P(5)}
	\\&\quad
	 +\tfrac{2q^{104} P(1)^7 }{P(6)^2 P(5) P(4)}
	 +\tfrac{11q^{117} P(2) P(1)^8 }{P(3)^2 P(6)^2 P(5) P(4)}
	,\\
	B_{13,5}(q^{13})
	&=
	 \tfrac{10 P(4)^6 P(3)^3}{q^{13} P(6) P(1)^2 P(5) P(2)^2 }
	 +\tfrac{3 P(4)^3 P(3)^5}{P(6) P(1)^2 P(5) P(2)}
	 +\tfrac{2q^{13} P(4)^3 P(3)^3 }{P(6) P(1) P(5)}
	 +\tfrac{10q^{26} P(1) P(4)^2 P(3)^4 }{P(6) P(5) P(2)^2}
	\\&\quad	 
	 +\tfrac{12q^{26} P(4) P(3)^2 P(2)^4 }{P(6) P(1)^2 P(5)}
	 +\tfrac{12q^{26} P(3)^5 P(2) }{P(6) P(1) P(5)}
	 +\tfrac{5q^{39} P(4) P(2)^5 }{P(6) P(1) P(5)}
	 +\tfrac{11q^{39} P(3)^3 P(2)^2 }{P(6) P(5)}
	 +\tfrac{q^{52} P(1) P(3) P(2)^3 }{P(6) P(5)}
	\\&\quad
	 +\tfrac{3q^{52} P(1)^2 P(3)^4 }{P(6) P(4) P(5)}
	 +\tfrac{3q^{65} P(1)^3 P(3)^2 P(2) }{P(6) P(4) P(5)}
	 +\tfrac{8q^{78} P(1)^4 P(2)^2 }{P(6) P(4) P(5)}
	,\\
	B_{13,6}(q^{13})
	&=
	 \tfrac{P(3)^3 P(4)^4}{P(5) P(6) P(1)^2}
	 +\tfrac{7q^{13} P(3)^8 }{P(5) P(6) P(2)^2 P(1)}
	 +\tfrac{8q^{26} P(3)^6 }{P(5) P(6) P(2)}
	 +\tfrac{5q^{26} P(3)^7 P(2)^2 }{P(5) P(4)^2 P(6) P(1)^2}
	 +\tfrac{6q^{39} P(3)^2 P(2)^6 }{P(5) P(4) P(6) P(1)^2}
	\\&\quad	 
	 +\tfrac{12q^{39} P(3)^5 P(2)^3 }{P(5) P(4)^2 P(6) P(1)}
	 +\tfrac{5q^{39} P(3)^4 P(1) }{P(5) P(6)}
	 +\tfrac{3q^{52} P(2)^6 }{P(3) P(4) P(6)}
	 +\tfrac{9q^{52} P(3)^5 P(1)^3 }{P(5) P(4) P(6) P(2)^2}
	 +\tfrac{4q^{52} P(2)^7 }{P(5) P(4) P(6) P(1)}
	\\&\quad
	 +\tfrac{11q^{52} P(3)^3 P(2)^4 }{P(5) P(4)^2 P(6)}
	 +\tfrac{2q^{65} P(2)^2 P(1)^3 }{P(5) P(6)}
	 +\tfrac{2q^{65} P(3)^3 P(1)^4 }{P(5) P(4) P(6) P(2)}
	 +\tfrac{12q^{65} P(3) P(2)^5 P(1) }{P(5) P(4)^2 P(6)}
	 +\tfrac{12q^{78} P(2)^3 P(1)^4 }{P(5) P(3)^2 P(6)}
	\\&\quad	
	 +\tfrac{12q^{78} P(3) P(1)^5 }{P(5) P(4) P(6)}
	 +\tfrac{2q^{78} P(2)^6 P(1)^2 }{P(5) P(3) P(4)^2 P(6)}
	 +\tfrac{10q^{91} P(2) P(1)^6 }{P(5) P(3) P(4) P(6)}
	 +\tfrac{4q^{104} P(1)^8 }{P(5) P(4)^2 P(6) P(2)}
	,\\
	B_{13,7}(q^{13})
	&=
	 \tfrac{4 P(2) P(4)^7}{P(6)^2 P(1)^2 P(5)}
	 +\tfrac{8 q^{13} P(3)^5 P(4)^3}{P(6)^2 P(1) P(5) P(2)}
	 +\tfrac{3 q^{26} P(3)^4 P(2)^3 P(4)}{P(6)^2 P(1)^2 P(5)}
	 +\tfrac{11 q^{26} P(3)^7}{P(6)^2 P(1) P(5)}
	 +\tfrac{4 q^{39} P(3)^2 P(2)^4 P(4)}{P(6)^2 P(1) P(5)}
	\\&\quad	 
	 +\tfrac{6 q^{39} P(3)^5 P(2)}{P(6)^2 P(5)}
	 +\tfrac{10 q^{52} P(5) P(2)^6}{P(6)^2 P(3) P(4)}
	 +\tfrac{9 q^{52}P(1)  P(3)^3 P(2)^2}{P(6)^2 P(5)}
	 +\tfrac{3 q^{52} P(3) P(2)^8}{P(6)^2 P(1)^2 P(5) P(4)}
	 +\tfrac{4 q^{65}P(1)^4  P(4)^2}{P(6)^2 P(5)}
	\\&\quad
	 +\tfrac{6  q^{65} P(1)^2 P(3) P(2)^3}{P(6)^2 P(5)}
	 +\tfrac{10 q^{65} P(2)^9}{P(6)^2 P(1) P(3) P(5) P(4)}
	 +\tfrac{2 q^{78} P(1)^4  P(3)^2 P(2)}{P(6)^2 P(5) P(4)}
	 +\tfrac{10 q^{91} P(1)^5  P(2)^2}{P(6)^2 P(5) P(4)}
	,\\
	B_{13,8}(q^{13})
	&=
	 \tfrac{10 P(3)^6 P(4)^2}{P(1)^2 P(6) P(5) P(2)}
	 +\tfrac{3q^{13} P(3)^4 P(4)^2 }{P(1) P(6) P(5)}
	 +\tfrac{3q^{13} P(3)^8 }{P(1)^2 P(4) P(6) P(5)}
	 +\tfrac{8q^{26} P(3)^2 P(4)^2 P(2) }{P(6) P(5)}
	 +\tfrac{9q^{26} P(3)^3 P(2)^4 }{P(1)^2 P(6) P(5)}
	\\&\quad	 
	 +\tfrac{q^{26} P(3)^6 P(2) }{P(1) P(4) P(6) P(5)}
	 +\tfrac{12q^{39} P(2)^4 }{P(6)}
	 +\tfrac{10q^{39} P(1)^2 P(3)^3 P(4) }{P(6) P(5) P(2)}
	 +\tfrac{10q^{39} P(3) P(2)^5 }{P(1) P(6) P(5)}
	 +\tfrac{8q^{39} P(3)^4 P(2)^2 }{P(4) P(6) P(5)}
	\\&\quad
	 +\tfrac{q^{52} P(1)^3 P(3) P(4) }{P(6) P(5)}
	 +\tfrac{6q^{52} P(1) P(3)^2 P(2)^3 }{P(4) P(6) P(5)}
	 +\tfrac{2q^{65} P(1)^2 P(2)^4 }{P(4) P(6) P(5)}
	,\\
	B_{13,9}(q^{13})
	&=
	 \tfrac{8 P(3)^4 P(4)^3}{P(5) P(6) P(1)^2}
	 +\tfrac{10q^{13} P(3)^5 P(4)^2 }{P(5) P(6) P(2)^2}
	 +\tfrac{5q^{13} P(3)^9 }{P(5) P(6) P(1) P(4) P(2)^2}
	 +\tfrac{9q^{26} P(3)^7 }{P(5) P(6) P(4) P(2)}
	 +\tfrac{9q^{39} P(3)^2  P(2)^3}{P(5) P(6)}
	\\&\quad	 
	 +\tfrac{12q^{39} P(1) P(3)^5 }{P(5) P(6) P(4)}
	 +\tfrac{7q^{39} P(3)^3  P(2)^6}{P(5) P(6) P(1)^2 P(4)^2}
	 +\tfrac{4q^{52} P(1)  P(2)^4}{P(5) P(6)}
	 +\tfrac{9q^{52} P(1)^2 P(3)^3  P(2)}{P(5) P(6) P(4)}
	 +\tfrac{10q^{52} P(3)  P(2)^7}{P(5) P(6) P(1) P(4)^2}
	\\&\quad
	 +\tfrac{9q^{65} P(1)^2  P(2)^5}{P(5) P(6) P(3)^2}
	 +\tfrac{2q^{65} P(1)^3 P(3)  P(2)^2}{P(5) P(6) P(4)}
	 +\tfrac{8q^{65} P(1)^4 P(3)^4 }{P(5) P(6) P(4)^2 P(2)}
	 +\tfrac{9q^{65} P(2)^8}{P(5) P(6) P(3) P(4)^2}
	 +\tfrac{4q^{78} P(1)^4 P(2)^3}{P(5) P(6) P(3) P(4)}
	\\&\quad
	 +\tfrac{9q^{78} P(1)^5 P(3)^2 }{P(5) P(6) P(4)^2}
	 +\tfrac{8q^{91} P(1)^6  P(2)}{P(5) P(6) P(4)^2}
	,\\
	B_{13,10}(q^{13})
	&=
	0
	,\\
	B_{13,11}(q^{13})
	&=
	 \tfrac{4 P(3) P(4)^5}{q^{13} P(6) P(1)^2}
	 +\tfrac{4 q^{13} P(3)^5 P(4)}{P(6) P(1) P(5)}
	 +\tfrac{7 q^{26} P(3)^3 P(2) P(4)}{P(6) P(5)}
	 +\tfrac{2 q^{26} P(3)^4 P(2)^4}{P(6) P(1)^2 P(4) P(5)}
	 +\tfrac{3 q^{26} P(3)^7 P(2)}{P(6) P(1) P(4)^2 P(5)}
	\\&\quad	 
	 +\tfrac{10 q^{39} P(2)^8}{P(6) P(1)^2 P(3) P(5)}
	 +\tfrac{4 q^{39} P(3)^2 P(2)^5}{P(6) P(1) P(4) P(5)}
	 +\tfrac{5 q^{39} P(3)^5 P(2)^2}{P(6) P(4)^2 P(5)}
	 +\tfrac{10 q^{52} P(1)^3 P(3)^2}{P(6) P(5)}
	 +\tfrac{4 q^{52} P(1) P(3)^3 P(2)^3}{P(6) P(4)^2 P(5)}
	\\&\quad	 
	 +\tfrac{q^{65} P(1)^2 P(3) P(2)^4}{P(6) P(4)^2 P(5)}
	 +\tfrac{10 q^{78} P(1)^3 P(2)^5}{P(6) P(3) P(4)^2 P(5)}
	,\\
	B_{13,12}(q^{13})
	&=
	\tfrac{9 P(4)^2 P(3)^5}{P(1)^2 P(6) P(5)}
	+\tfrac{5 q^{13} P(4) P(3)^6}{P(6) P(2)^2 P(5)}
	+\tfrac{4 q^{13} P(2) P(3)^7}{P(1)^2 P(4) P(6) P(5)}
	+\tfrac{8 q^{26} P(1) P(4) P(3)^4}{P(6) P(2) P(5)}
	+\tfrac{5 q^{26} P(2)^5 P(3)^2}{P(1)^2 P(6) P(5)}
	\\&\quad
	+\tfrac{3 q^{26} P(2)^2 P(3)^5}{P(1) P(4) P(6) P(5)}
	+\tfrac{11 q^{39} P(2)^5}{P(6) P(3)}
	+\tfrac{9 q^{39} P(1)^2 P(4) P(3)^2}{P(6) P(5)}
	+\tfrac{6 q^{39} P(2)^6}{P(1) P(6) P(5)}
	+\tfrac{2 q^{39} P(2)^3 P(3)^3}{P(4) P(6) P(5)}
	\\&\quad
	+\tfrac{5 q^{52} P(1)^3 P(4) P(2)}{P(6) P(5)}
	+\tfrac{5 q^{52} P(1) P(2)^4 P(3)}{P(4) P(6) P(5)}
	+\tfrac{3 q^{65} P(1)^4 P(4) P(2)^2}{P(6) P(5) P(3)^2}
	+\tfrac{8 q^{65} P(1)^5 P(3)}{P(6) P(2) P(5)}
	+\tfrac{q^{65} P(1)^2 P(2)^5}{P(4) P(6) P(5) P(3)}
	\\&\quad
	+\tfrac{10 q^{78} P(1)^6 }{P(6) P(5) P(3)}
.
\end{align*}
\end{theorem}
Theorem \ref{TheoremMain} implies Theorem \ref{TheoremCongruences} as the 
coefficients of the powers of $q$ for the appropriate residue classes are zero
in the above congruences.
It is worth noting that none of the representations of $U(q)$ and $V(q)$ that
we will use 
suggest that these functions are even congruent to modular forms, whereas
$F(q)$ is essentially a modular form. There are few partition functions that satisfy 
such simple congruences modulo $13$ that are not modular forms or closely
related to modular forms. However, we can recognize the series appearing on the
right hand side of Theorem \ref{TheoremMain} as mock modular forms. For example,
in the notation of \cite{Zwegers} we can write
\begin{align*}
	\frac{q^2}{E(9)P(1)}\sum_{n=-\infty}^\infty 
	\frac{(-1)^nq^{\frac{9n^2+15n}{2}}}{(1-q^{9n+3})}
	&=
		iq^{\frac{1}{8}}\mu(3\tau,3\tau;9\tau)
	,
\end{align*}
where $q=e^{2\pi i\tau}$. With this the functions $U(q)$ and
$V(q)$ should be quasimock mock modular forms, which one should consult
\cite{BLO1} to understand.  To establish these functions are indeed quasimock
theta functions one should use the identities in the next section to recognize
$U(q)$ and $V(q)$ as a second derivative of a certain Lambert series.

In Section 2 we develop the general identities and congruences necessary to prove
Theorem \ref{TheoremMain}. While we will only use these congruences modulo
$3$, $5$, $7$, and $13$, they hold for all odd $\ell>1$ (in particular $\ell$
need not be prime). In this way,
the identities of the article could also be used to determine $U(q)$, $V(q)$,
and related functions to other moduli.
In Section 3 we prove Theorem \ref{TheoremMain} and in Section 4 we give a few
concluding remarks.

\section{Preliminary Identities}

To begin we find another form of the generating functions for $U(q)$ and 
$V(q)$. We say a pair of sequences $(\alpha,\beta)$ is a Bailey pair relative
to $(a,q)$ if
\begin{align*}
	\beta_n &= \sum_{k=0}^\infty \frac{\alpha_k}{\aqprod{q}{q}{n-k}\aqprod{aq}{q}{n+k}}
.
\end{align*}
A limiting case of Bailey's Lemma gives that
\begin{align*}
	\sum_{n=0}^\infty 
	\aqprod{\rho_1,\rho_2}{q}{n}\left(\tfrac{aq}{\rho_1\rho_2}\right)^n\beta_n
	&=
	\frac{\aqprod{aq/\rho_1,aq/\rho_2}{q}{\infty}}
		{\aqprod{aq,\tfrac{aq}{\rho_1\rho_2}}{q}{\infty}}
	\sum_{n=0}^\infty
	\frac{\aqprod{\rho_1,\rho_2}{q}{n} \left(\tfrac{aq}{\rho_1\rho_2}\right)^n \alpha_n  }
		{\aqprod{aq/\rho_1,aq/\rho_2}{q}{n}}
	.
\end{align*}
One may consult \cite{Andrews1} for a history of Bailey pairs and Bailey's Lemma.
When $(\alpha,\beta)$ is relative to $(1,q)$
and we set $\rho_1=z$, $\rho_2=z^{-1}$ this reduces to
\begin{align*}
	\sum_{n=0}^\infty 
	\aqprod{z,z^{-1}}{q}{n}q^n\beta_n
	&=
	\frac{\aqprod{z,z^{-1}}{q}{\infty}}
		{\aqprod{q}{q}{\infty}^2}
	\sum_{n=0}^\infty
	\frac{ q^n \alpha_n  } {(1-zq^n)(1-z^{-1}q^n)}
	.
\end{align*}

We recall one version of the finite Jacobi triple product identity \cite[page 49]{AndrewsBook} is
\begin{align*}
	\frac{\aqprod{xq,x^{-1}}{q}{n}}{\aqprod{q}{q}{2n}}
	&=
	\sum_{j=-n}^n \frac{(-1)^jx^jq^{j(j+1)/2}}
		{\aqprod{q}{q}{n-j}\aqprod{q}{q}{n+j}}
	=
	\frac{1}{\aqprod{q}{q}{2n}}
	+
	\sum_{j=1}^n 
	\frac{(-1)^j q^{j(j-1)/2} (x^{-n}+x^nq^n) }
		{\aqprod{q}{q}{n-j}\aqprod{q}{q}{n+j}}
.
\end{align*}
In the language of Bailey pairs, this says that
$(\alpha'(x),\beta'(x))$ is a Bailey pair with respect to $(1,q)$
,where
\begin{align*}
	\beta'_n (x)
	&= 
	\frac{\aqprod{xq,x^{-1}}{q}{n}}{\aqprod{q}{q}{2n}}
	,\\
	\alpha'_n(x)
	&=
	\PieceTwo{1}{(-1)^n q^{n(n-1)/2} (x^{-n}+x^nq^n)}{n=0}{n\ge 1}
.
\end{align*}
A direct, but somewhat lengthy, calculation shows that
\begin{align*}
	\frac{d^2}{dx^2}\beta_n^\prime(x) \vert_{x=1}
	&= -2\frac{\aqprod{q}{q}{n-1}^2}{\aqprod{q}{q}{2n}}
	,\\
	\frac{d^2}{dx^2}\beta_n^\prime(x) \vert_{x=q^{-1}}
	&= -2q^{n+2}\frac{\aqprod{q}{q}{n-1}^2}{\aqprod{q}{q}{2n}}
.
\end{align*}
This gives the Bailey pairs, with respect to $(1,q)$,
\begin{align*}
	\beta^u_n 
		&= 
		\frac{\aqprod{q}{q}{n-1}^2}{\aqprod{q}{q}{2n}}	
	,
	&\alpha^u_n 
		&= 
		(-1)^{n+1} q^{n(n-1)/2}
		\left(\frac{n(n+1)}{2}+\frac{n(n-1)}{2}q^n\right)
	,\\
	\beta^v_n 
		&= \frac{q^n\aqprod{q}{q}{n-1}^2}{\aqprod{q}{q}{2n}}	
	,
	&\alpha^v_n 
		&= 
		(-1)^{n+1} q^{n(n-1)/2}
		\left(\frac{n(n+1)}{2}q^{n}+\frac{n(n-1)}{2}\right)
	.
\end{align*}
We note $\beta^u_0=\beta^v_0=\alpha^u_0=\alpha^v_0=0$.

By Bailey's Lemma, we then have that
\begin{align*}
	U(q)
	&=
	\frac{1}{\aqprod{q}{q}{\infty}^3}
	\sum_{n=1}^\infty \frac{\aqprod{q}{q}{n-1}^4 q^n}{\aqprod{q}{q}{2n}}
	\\
	&=
	\frac{1}{\aqprod{q,z,z^{-1}}{q}{\infty}}
	\sum_{n=1}^\infty \frac{\aqprod{z,z^{-1}}{q}{n}q^n\aqprod{q}{q}{n-1}^2}
		{\aqprod{q}{q}{2n}}
	\Bigg\vert_{z=1}
	\\
	&=
	\frac{1}{\aqprod{q,z,z^{-1}}{q}{\infty}}
	\sum_{n=0}^\infty \aqprod{z,z^{-1}}{q}{n}q^n\beta^u_n
	\Bigg\vert_{z=1}
	\\
	&=
	\frac{1}{\aqprod{q}{q}{\infty}^3}
	\sum_{n=0}^\infty 
	\frac{q^n\alpha^u_n}{(1-zq^n)(1-z^{-1}q^n)}
	\Bigg\vert_{z=1}
	\\
	&=
	\frac{-1}{2\aqprod{q}{q}{\infty}^3}
	\sum_{n=1}^\infty 
	\frac{(-1)^n q^{n(n+1)/2}(n(n+1)+n(n-1)q^n)}{(1-zq^n)(1-z^{-1}q^n)}
	\Bigg\vert_{z=1}
	\\
	&=
	\frac{-1}{2\aqprod{q}{q}{\infty}^3}
	\sum_{n=1}^\infty 
	\frac{(-1)^n q^{n(n+1)/2}(n(n+1)+n(n-1)q^n)}{(1-q^n)^2}
	\\
	&=
	\frac{-1}{2\aqprod{q}{q}{\infty}^3}
	\sideset{}{'}\sum_{n=-\infty}^\infty 
	\frac{(-1)^n q^{n(n+1)/2}n(n+1)}{(1-q^n)^2}
.
\end{align*}
Here and elsewhere, a prime superscript in a summation indicates to omit the
terms that correspond to a division by zero.
Similarly we have that
\begin{align*}
	V(q)
	&=
	\frac{-1}{2\aqprod{q}{q}{\infty}^3}
	\sideset{}{'}\sum_{n=-\infty}^\infty 
	\frac{(-1)^n q^{n(n+1)/2}n(n-1)}{(1-q^n)^2}
\end{align*}

We define the two series
\begin{align*}
	S_\ell(b)
	&=
	\sideset{}{'}\sum_{n=-\infty}^\infty 
	\frac{(-1)^{n} q^{n(n+1)/2 +bn} n(n+1) }{(1-q^{\ell n})}
	,\\
	T(z,w,q)
	&=
	\sum_{n=-\infty}^\infty \frac{(-1)^nq^{n(n+1)/2}w^n}{1-zq^n}
.
\end{align*}
The series $T(z,w,q)$ is a generalized Lambert series that appears quite often
in number theory. It has been used by Lewis \cite{Lewis}, Lewis and Santa-Gadea
\cite{LewisSantaGadea}, and Ekin \cite{Ekin} in studying the the crank of 
partition. Also up to a product, it is the Apell-Lerch sum studied in detail
by Zwegers \cite{Zwegers} and Hickerson and Mortenson \cite{HickersonMortenson}.

To work with the series $T(z,w,q)$, we use the $r=1$, $s=2$ case of Theorem 2.1
of \cite{Chan}, which is that
\begin{align}
	\label{ChanIdent1}
	\frac{\jacprod{a}{q}\aqprod{q}{q}{\infty}^2}
		{\jacprod{b_1,b_2}{q}}
	&=
	\frac{\jacprod{a/b_1}{q}}{\jacprod{b_2/b_1}{q}}T(b_1, a/b_2, q)
	+
	\frac{\jacprod{a/b_2}{q}}{\jacprod{b_1/b_2}{q}}T(b_2, a/b_1, q)
.
\end{align}
Additionally, by letting
$n\mapsto-n+1$, we have that
\begin{align}
	\label{IdentT3}
	T(z,w,q) &= z^{-1}w^{-1}q T(z^{-1}q,w^{-1}q,q)
.
\end{align}
For integers $a$, $b$, and $c$ we write
\begin{align*}
	T(a,b,c) := T(q^a,q^b,q^c).
\end{align*}

\begin{lemma}
\label{MainLemmaForSSeries}
For $\ell > 1$ odd and $b$ any integer,
\begin{align*}
	S_\ell(b)
	&\equiv
		\frac{2b (-1)^b q^{\ell m -\frac{b(b+1)}{2}} \aqprod{q}{q}{\infty}^3}
		{ \aqprod{q^{\ell^2}}{q^{\ell^2}}{\infty}\jacprod{q^{\ell m}}{q^{\ell^2}}}
		T\left( \tfrac{\ell(\ell-1)}{2} + \ell m -\ell b, \ell m, \ell^2   \right)
		\\&\quad
		+
		(-1)^{\frac{\ell+1}{2}+b} q^{\frac{\ell^2-1}{8}-\frac{b(b+1)}{2}+\ell m}
		\frac{\aqprod{q^{\ell^2}}{q^{\ell^2}}{\infty}^2}
			{\jacprod{q^{\ell m}, q^{\frac{\ell(\ell-1)}{2} + \ell m -\ell b}}{q^{\ell^2}}}
		\\&\quad
		\times
		\sum_{\substack{k=0,\\k\not\equiv b+\frac{1}{2}\pmod{\ell}}}^{\ell-1}
		(-1)^{k} q^{\frac{k(k-\ell)}{2}} (k-b-\tfrac{1}{2})(k-b+\tfrac{1}{2})
		\frac{\jacprod{q^{\frac{\ell(\ell-1)}{2} + \ell m +\ell k-\ell b}, q^{\ell k-\ell m}}{q^{\ell^2}}} 
			{\jacprod{q^{\frac{\ell(\ell-1)}{2}-\ell b+\ell k}}{q^{\ell^2}}}
	\pmod{\ell}
	,
\end{align*}
where $m$ is any integer such that $1\le m\le \ell-1$ and 
$m\not\equiv b+\frac{1}{2}\pmod{\ell}$.
\end{lemma}
\begin{proof}
In the series form of $S_\ell(b)$, we replace $n$ by $\ell n+k+c$, where
$c=\frac{\ell-1}{2}-b$ and $k=0,1,\dots,\ell-1$. Working modulo $\ell$, this 
gives that
\begin{align*}
	S_\ell(b)
	&=
	\sideset{}{'}\sum_{n=-\infty}^\infty
	\frac{(-1)^n q^{\frac{n(n+1)}{2}+bn}n(n+1)  }{1-q^{\ell n}}
	\\
	&\equiv
	\sum_{k=0}^{\ell-1}
	(-1)^{k+c} q^{\frac{c(c+1)}{2}+\frac{k(k+1)}{2}+ck+bc+bk } (k+c)(k+c+1)
	\sideset{}{'}\sum_{n=-\infty}^\infty
	\frac{ (-1)^n q^{\frac{\ell^2n^2+\ell n}{2} +\ell cn + \ell kn + \ell bn }  }
		{1-q^{ \ell^2n +\ell k + \ell c }}
	\\
	&\equiv
	\sum_{k=0}^{\ell-1}
	(-1)^{k+c} q^{\frac{c(c+1)}{2}+\frac{k(k+1)}{2}+ck+bc+bk } (k+c)(k+c+1)
	\sideset{}{'}\sum_{n=-\infty}^\infty
	\frac{ (-1)^n q^{\frac{\ell^2n(n+1)}{2} +\ell( c + k  + b -\frac{\ell-1}{2} )n }  }
		{1-q^{ \ell^2n +\ell k + \ell c }}
	\\
	&\equiv
	\sum_{k=0}^{\ell-1}
	(-1)^{k+c} q^{\frac{c(c+1)}{2}+\frac{k(k+1)}{2}+ck+bc+bk } (k+c)(k+c+1)
	\sideset{}{'}\sum_{n=-\infty}^\infty
	\frac{ (-1)^n q^{\frac{\ell^2n(n+1)}{2} +\ell kn }  }
		{1-q^{ \ell^2n +\ell k + \ell c }}
	\pmod{\ell}
	.
\end{align*}
Here we note that the only division by zero would occur in the inner series when
$k\equiv -c\equiv b+\frac{1}{2}\pmod{\ell}$. However, this entire term of the outer
sum is zero modulo $\ell$,
due to the factor of $(k+c)$. 
Additionally, by the restrictions on $m$, this
does not correspond to the $k=m$ term. Therefore we can omit the 
$k\equiv-c\pmod{\ell}$ term and
write the remaining inner series 
as $T(\ell k + \ell c, \ell k, \ell^2)$
With this in mind, we now have that
\begin{align*}
	S_\ell(b)
	&\equiv
	\sum_{\substack{k=0,\\k\not\equiv-c\pmod{\ell}}}^{\ell-1}
	\hspace{-10pt}
	(-1)^{k+c} q^{\frac{c(c+1)}{2}+\frac{k(k+1)}{2}+ck+bc+bk } (k+c)(k+c+1)
	T(\ell k + \ell c, \ell k, \ell^2)
	\\
	&\equiv
	(-1)^{\frac{\ell-1}{2}+b}q^{\frac{\ell^2-1}{8}-\frac{b(b+1)}{2}}
	\hspace{-10pt}
	\sum_{\substack{k=0,\\k\not\equiv-c\pmod{\ell}}}^{\ell-1}
	\hspace{-10pt}
	(-1)^{k} q^{\frac{k(k+\ell)}{2} } (k+c)(k+c+1)
	T(\ell k + \ell c, \ell k, \ell^2)
	\\
	&\equiv
		(-1)^{\frac{\ell-1}{2}+b+m} q^{\frac{\ell^2-1}{8}-\frac{b(b+1)}{2}+\frac{m(m+\ell)}{2} } 
		(m+c)(m+c+1) T(\ell m + \ell c, \ell m, \ell^2)
		\\&\quad
		+	
		(-1)^{\frac{\ell-1}{2}+b}q^{\frac{\ell^2-1}{8}-\frac{b(b+1)}{2}}
		\hspace{-15pt}
		\sum_{\substack{k=0,\\k\not=m,k\not\equiv-c\pmod{\ell}}}^{\ell-1}
		\hspace{-15pt}
		(-1)^{k} q^{\frac{k(k+\ell)}{2} } (k+c)(k+c+1)
		T(\ell k + \ell c, \ell k, \ell^2)
	\pmod{\ell}
	.
\end{align*}
For $k\not=m$ and $k\not\equiv-c\pmod{\ell}$, applying (\ref{ChanIdent1}) with 
$q\mapsto q^{\ell^2}$,
$b_1=q^{\ell m + \ell c}$, $b_2=q^{\ell k +\ell c}$, and
$a=q^{\ell m + \ell k + \ell c}$
yields
\begin{align*}
	\frac{ \aqprod{q^{\ell^2}}{q^{\ell^2}}{\infty}^2
		\jacprod{q^{\ell m+\ell k + \ell c}}{q^{\ell^2}}}
	{\jacprod{ q^{\ell m+\ell c}, q^{\ell k+\ell c}  }{q^{\ell^2}}}
	&=
	\frac{\jacprod{q^{\ell k}}{q^{\ell^2}}}{\jacprod{q^{\ell k - \ell m}}{q^{\ell^2}}}
	T(\ell m +\ell c, \ell m , \ell^2)
	+	
	\frac{\jacprod{q^{\ell m} }{q^{\ell^2}}}{\jacprod{q^{\ell m -\ell k}}{q^{\ell^2}}}
	T(\ell k +\ell c, \ell k, \ell^2)
,
\end{align*}
which we rearrange to 
\begin{align*}
	q^{\ell k-\ell m} T(\ell k +\ell c, \ell k, \ell^2)
	&=
	\frac{\jacprod{q^{\ell k}}{q^{\ell^2}}}{\jacprod{q^{\ell m}}{q^{\ell^2}}}
	T(\ell m +\ell c, \ell m, \ell^2)
	-
	\frac{ \aqprod{q^{\ell^2}}{q^{\ell^2}}{\infty}^2
		\jacprod{q^{\ell m +\ell k + \ell c}, q^{\ell k-\ell m }}{q^{\ell^2}}}
	{\jacprod{q^{\ell m},  q^{\ell m +\ell c}, q^{\ell k+\ell c}  }{q^{\ell^2}}}
.
\end{align*}
The restrictions on $m$ prevent a division by zero in the above identity.
We note this identity also holds for $k=m$.
Thus
\begin{align*}
	&S_\ell(b)
	\\
	&\equiv
		T(\ell m + \ell c, \ell m, \ell^2)
		(-1)^{\frac{\ell-1}{2}+b}q^{\frac{\ell^2-1}{8}-\frac{b(b+1)}{2}+ \ell m}
		\hspace{-15pt}
		\sum_{\substack{k=0,\\k\not\equiv-c\pmod{\ell}}}^{\ell-1}
		\hspace{-15pt}
		(-1)^{k} q^{\frac{k(k-\ell)}{2} } (k+c)(k+c+1)
		\frac{\jacprod{q^{\ell k}}{q^{\ell^2}}} {\jacprod{q^{\ell m}}{q^{\ell^2}}}
		\\&\quad
		+
		(-1)^{\frac{\ell+1}{2}+b}q^{\frac{\ell^2-1}{8}-\frac{b(b+1)}{2}+\ell m}
		\hspace{-15pt}
		\sum_{\substack{k=0,\\k\not\equiv-c\pmod{\ell}}}^{\ell-1}
		\hspace{-15pt}
		(-1)^{k} q^{\frac{k(k-\ell)}{2}} (k+c)(k+c+1)
		\frac{\aqprod{q^{\ell^2}}{q^{\ell^2}}{\infty}^2
			\jacprod{q^{\ell m +\ell k+\ell c}, q^{\ell k-\ell m}}{q^{\ell^2}}} 
			{\jacprod{q^{\ell m}, q^{\ell m+\ell c},q^{\ell k+\ell c}}{q^{\ell^2}}}
	\\
	&\equiv
		T(\ell m +\ell c, \ell m, \ell^2)
		\frac{(-1)^{\frac{\ell-1}{2}+b} q^{\frac{\ell^2-1}{8}-\frac{b(b+1)}{2}+\ell m} }
			{\jacprod{q^{\ell m}}{q^{\ell^2}}}
		\sum_{k=1}^{\ell-1}
		(-1)^{k} q^{\frac{k(k-\ell)}{2}} (k+c)(k+c+1)
		\jacprod{q^{\ell k}}{q^{\ell^2}}
		\\&\quad
		+
		(-1)^{\frac{\ell+1}{2}+b} q^{\frac{\ell^2-1}{8}-\frac{b(b+1)}{2}+\ell m}
		\frac{\aqprod{q^{\ell^2}}{q^{\ell^2}}{\infty}^2}
			{\jacprod{q^{\ell m}, q^{\ell m+\ell c}}{q^{\ell^2}}}
		\\&\quad\times
		\sum_{\substack{k=0,\\k\not\equiv-c\pmod{\ell}}}^{\ell-1}
		(-1)^{k} q^{\frac{k(k-\ell)}{2}} (k+c)(k+c+1)
		\frac{\jacprod{q^{\ell m +\ell k+\ell c}, q^{\ell k-\ell m}}{q^{\ell^2}}} 
			{\jacprod{q^{\ell k+\ell c}}{q^{\ell^2}}}
	\\
	&\equiv
		T(\tfrac{\ell(\ell-1)}{2} + \ell m -\ell b, \ell m, \ell^2)
		\frac{(-1)^{\frac{\ell-1}{2}+b} q^{\frac{\ell^2-1}{8}-\frac{b(b+1)}{2}+\ell m} }
			{\jacprod{q^{\ell m}}{q^{\ell^2}}}
		\sum_{k=1}^{\ell-1}
		(-1)^{k} q^{\frac{k(k-\ell)}{2}} (k+c)(k+c+1)
		\jacprod{q^{\ell k}}{q^{\ell^2}}
		\\&\quad
		+
		(-1)^{\frac{\ell+1}{2}+b} q^{\frac{\ell^2-1}{8}-\frac{b(b+1)}{2}+\ell m}
		\frac{\aqprod{q^{\ell^2}}{q^{\ell^2}}{\infty}^2}
			{\jacprod{q^{\ell m}, q^{\frac{\ell(\ell-1)}{2} + \ell m -\ell b}}{q^{\ell^2}}}
		\\&\quad
		\times
		\sum_{\substack{k=0,\\k\not\equiv b+\frac{1}{2}\pmod{\ell}}}^{\ell-1}
		(-1)^{k} q^{\frac{k(k-\ell)}{2}} (k-b-\tfrac{1}{2})(k-b+\tfrac{1}{2})
		\frac{\jacprod{q^{\frac{\ell(\ell-1)}{2}+\ell m +\ell k-\ell b}, q^{\ell k-\ell m}}{q^{\ell^2}}} 
			{\jacprod{q^{\frac{\ell(\ell-1)}{2}-\ell b+\ell k}}{q^{\ell^2}}}
	\pmod{\ell}
.
\end{align*}
To finish the proof, we must verify that
\begin{align*}
	\sum_{k=1}^{\ell-1}
		(-1)^{k} q^{\frac{k(k-\ell)}{2}} (k+c)(k+c+1) \jacprod{q^{\ell k}}{q^{\ell^2}}
	&\equiv
	2b (-1)^{\frac{\ell-1}{2}} q^{\frac{1-\ell^2}{8}}
	\frac{\aqprod{q}{q}{\infty}^3}{\aqprod{q^{\ell^2}}{q^{\ell^2}}{\infty}}
	\pmod{\ell}
.
\end{align*}
By $k\mapsto \ell-k$, we see that
\begin{align*}
	\sum_{k=1}^{\ell-1}
		(-1)^{k} q^{\frac{k(k-\ell)}{2}} (k+c)(k+c+1)
		\jacprod{q^{\ell k}}{q^{\ell^2}}
	&\equiv
	-
	\sum_{k=1}^{\ell-1}
		(-1)^{k} q^{\frac{k(k-\ell)}{2}} (c-k)(c-k+1)
		\jacprod{q^{\ell k}}{q^{\ell^2}}
	\pmod{\ell}
,
\end{align*}
and so
\begin{align*}
	&\sum_{k=1}^{\ell-1}
		(-1)^{k} q^{\frac{k(k-\ell)}{2}} (k+c)(k+c+1)
		\jacprod{q^{\ell k}}{q^{\ell^2}}
	\\
	&\equiv
	\frac{1}{2}
	\sum_{k=1}^{\ell-1}
		(-1)^{k} q^{\frac{k(k-\ell)}{2}}
		\jacprod{q^{\ell k}}{q^{\ell^2}}
		\left((k+c)(k+c+1)-(c-k)(c-k+1) \right)
 	\\
	&\equiv
	-2b
	\sum_{k=1}^{\ell-1}
		(-1)^{k} q^{\frac{k(k-\ell)}{2}} k
		\jacprod{q^{\ell k}}{q^{\ell^2}}
	\pmod{\ell}
.
\end{align*}
So we need to show
\begin{align}
	\label{EqMainLemmaLastReduction}
	\sum_{k=1}^{\ell-1}
	(-1)^{k} q^{\frac{k(k-\ell)}{2}} k \jacprod{q^{\ell k}}{q^{\ell^2}}
	&\equiv
	(-1)^{\frac{\ell+1}{2}}q^{\frac{1-\ell^2}{8}}
	\frac{\aqprod{q}{q}{\infty}^3}{\aqprod{q^{\ell^2}}{q^{\ell^2}}{\infty}}
	\pmod{\ell}
.
\end{align}
This actually follows just from Jacobi's identity
for $\aqprod{q}{q}{\infty}^3$,
which we write as
\begin{align*}
	\aqprod{q}{q}{\infty}^3
	&=
	\sum_{n=0}^\infty (-1)^n (2n+1) q^{\frac{n(n+1)}{2}}
	\\
	&=
		\sum_{n=0}^\infty (-1)^n n q^{\frac{n(n+1)}{2}}	
		+
		\sum_{n=0}^\infty (-1)^n n q^{\frac{n(n+1)}{2}}	
		+	
		\sum_{n=0}^\infty (-1)^n q^{\frac{n(n+1)}{2}}	
	\\
	&=
		\sum_{n=1}^\infty (-1)^{n-1} (n-1) q^{\frac{n(n-1)}{2}}	
		+
		\sum_{n=-\infty}^0 (-1)^{n+1} n q^{\frac{n(n-1)}{2}}	
		+	
		\sum_{n=0}^\infty (-1)^n q^{\frac{n(n+1)}{2}}	
	\\
	&=
		\sum_{n=-\infty}^\infty (-1)^{n+1} n q^{\frac{n(n-1)}{2}}	
		+
		\sum_{n=1}^\infty (-1)^{n} q^{\frac{n(n-1)}{2}}	
		+	
		\sum_{n=0}^\infty (-1)^n q^{\frac{n(n+1)}{2}}	
	\\
	&=
		\sum_{n=-\infty}^\infty (-1)^{n+1} n q^{\frac{n(n-1)}{2}}	
.
\end{align*}
We set $d=\frac{1-\ell}{2}$
and let
$n\mapsto \ell n + k +d$, for
$k=0,1,\dots,\ell-1$,
so that
\begin{align}
	\label{EqECubedModL}
	\aqprod{q}{q}{\infty}^3
	&\equiv
		\sum_{k=0}^{\ell-1}
		(-1)^{k+d+1} (k+d) q^{\frac{d(d-1)}{2}+\frac{k(k-1)}{2}+dk  }
		\sum_{n=-\infty}^\infty 
		(-1)^n q^{ \frac{\ell^2n^2-\ell n}{2}+\ell dn + \ell kn  }
	\nonumber\\
	&\equiv
		\sum_{k=0}^{\ell-1}
		(-1)^{k+d+1} k q^{\frac{d(d-1)}{2}+\frac{k(k-1)}{2}+dk  }
		\sum_{n=-\infty}^\infty 
		(-1)^n q^{ \frac{\ell^2n^2-\ell n}{2}+\ell dn + \ell kn  }
		\nonumber\\&\quad
		+
		d\sum_{k=0}^{\ell-1}
		(-1)^{k+d+1} q^{\frac{d(d-1)}{2}+\frac{k(k-1)}{2}+dk  }
		\sum_{n=-\infty}^\infty 
		(-1)^n q^{ \frac{\ell^2n^2-\ell n}{2}+\ell dn + \ell kn  }
	\nonumber\\
	&\equiv
		\sum_{k=0}^{\ell-1}
		(-1)^{k+d+1} k q^{\frac{d(d-1)}{2}+\frac{k(k-1)}{2}+dk  }
		\sum_{n=-\infty}^\infty 
		(-1)^n q^{ \frac{\ell^2n^2-\ell n}{2}+\ell dn + \ell kn  }
		+
		d\sum_{n=-\infty}^\infty 
		(-1)^{n+1} q^{ \frac{n(n-1)}{2} }
	\nonumber\\
	&\equiv
		\sum_{k=0}^{\ell-1}
		(-1)^{k+d+1} k q^{\frac{d(d-1)}{2}+\frac{k(k-1)}{2}+dk  }
		\sum_{n=-\infty}^\infty 
		(-1)^n q^{ \frac{\ell^2n^2-\ell n}{2}+\ell dn + \ell kn  }
	\nonumber\\
	&\equiv
		\sum_{k=1}^{\ell-1}
		(-1)^{k+d+1} k q^{\frac{d(d-1)}{2}+\frac{k(k-1)}{2}+dk  }
		\sum_{n=-\infty}^\infty 
		(-1)^n q^{\ell kn-\ell^2n} q^{ \frac{\ell^2n(n+1)}{2} }
	\nonumber\\
	&\equiv
		\sum_{k=1}^{\ell-1}
		(-1)^{k+d+1} k q^{\frac{d(d-1)}{2}+\frac{k(k-1)}{2}+dk  }
		\aqprod{q^{\ell^2}}{q^{\ell^2}}{\infty}
		\jacprod{q^{\ell k}}{q^{\ell^2}}
	\nonumber\\
	&\equiv
		(-1)^{\frac{1+\ell}{2}}  q^{ \frac{\ell^2-1}{8} }
		\sum_{k=1}^{\ell-1}
		(-1)^{k} k q^{\frac{k(k-\ell)}{2}  }
		\aqprod{q^{\ell^2}}{q^{\ell^2}}{\infty}
		\jacprod{q^{\ell k}}{q^{\ell^2}}
	\pmod{\ell}
	.
\end{align}
This immediately implies (\ref{EqMainLemmaLastReduction}) and we have proved
the Lemma. A similar congruence for $\aqprod{q}{q}{\infty}^3$ can
be found in Lemma 3 of \cite{AS}. 
\end{proof}

While it is not true that $S_\ell(\ell-b-1)\equiv\pm S_\ell(b)\pmod{\ell}$,
they are closely related.
\begin{lemma}
\label{SecondLemmaForSSeries}
For $\ell>1$ odd and $b$ any integer, 
\begin{align*}
	S_\ell(\ell-b-1)
	&\equiv
		\frac{2(b+1)(-1)^{b+1} q^{\ell m -\frac{b(b+1)}{2}} \aqprod{q}{q}{\infty}^3}
		{ \aqprod{q^{\ell^2}}{q^{\ell^2}}{\infty}\jacprod{q^{\ell m}}{q^{\ell^2}}}
		T\left( \tfrac{\ell(\ell-1)}{2} + \ell m -\ell b, \ell m, \ell^2   \right)
		\\&\quad
		-
		(-1)^{\frac{\ell+1}{2}+b} q^{\frac{\ell^2-1}{8}-\frac{b(b+1)}{2}+\ell m}
		\frac{\aqprod{q^{\ell^2}}{q^{\ell^2}}{\infty}^2}
			{\jacprod{q^{\ell m}, q^{\frac{\ell(\ell-1)}{2} + \ell m -\ell b}}{q^{\ell^2}}}
		\\&\quad
		\times
		\sum_{\substack{k=0,\\k\not\equiv b+\frac{1}{2}\pmod{\ell}}}^{\ell-1}
		(-1)^{k} q^{\frac{k(k-\ell)}{2}} (-k+b+\tfrac{1}{2})(-k+b+\tfrac{3}{2})
		\frac{\jacprod{q^{\frac{\ell(\ell-1)}{2} + \ell m +\ell k-\ell b}, q^{\ell k-\ell m}}{q^{\ell^2}}} 
			{\jacprod{q^{\frac{\ell(\ell-1)}{2}-\ell b+\ell k}}{q^{\ell^2}}}
	\\&\quad
	\pmod{\ell}
	,
\end{align*}
where $m$ is any integer such that $1\le m\le \ell-1$ and 
$m\not\equiv b+\frac{1}{2}\pmod{\ell}$.
\end{lemma}
\begin{proof}
In Lemma we use $b\mapsto\ell-b-1$ and $m\mapsto\ell-m$. 
We note that $m\equiv b+\frac{1}{2}\pmod{\ell}$ if and only if
$\ell-m\equiv\ell-b-1+\frac{1}{2}\pmod{\ell}$.
We note that
\begin{align*}
	\frac{\ell(\ell-1)}{2} + \ell(\ell-m) - \ell(\ell-b-1)
	&=
	\ell^2 - ( \tfrac{\ell(\ell-1)}{2} + \ell m -\ell b)
,
\end{align*}
so that
\begin{align*}
	&q^{\ell(\ell-m)-\frac{(\ell-b-1)(\ell-b)}{2}}
	T\left( \tfrac{\ell(\ell-1)}{2}+\ell(\ell-m)-\ell(\ell-b-1), \ell(\ell-m), \ell^2 \right)
	\\
	&=
	q^{\ell(\ell-m)-\frac{(\ell-b-1)(\ell-b)}{2}}
	T\left( \ell^2 - ( \tfrac{\ell(\ell-1)}{2} + \ell m -\ell b), \ell^2-\ell m, \ell^2 \right)
	\\
	&=
	q^{\ell(\ell-m)-\frac{(\ell-b-1)(\ell-b)}{2} - \ell^2 + \frac{\ell(\ell-1)}{2} +2\ell m - \ell b}
	T\left( \tfrac{\ell(\ell-1)}{2} + \ell m -\ell b, \ell m, \ell^2 \right)
	\\
	&=
	q^{\ell m-\frac{b(b+1)}{2}}
	T\left( \tfrac{\ell(\ell-1)}{2} + \ell m -\ell b, \ell m, \ell^2 \right)
,
\end{align*}
by (\ref{IdentT3}). In the sum over $k$, we use $k\mapsto\ell-k$ for 
$k$ from $1$ to $\ell-1$.
We note that
\begin{align*}
	\jacprod{ q^{ \frac{\ell(\ell-1)}{2} + \ell(\ell-m) - \ell(\ell-b-1) }}{q^{\ell^2}}
	&=
		\jacprod{q^{ \frac{\ell(\ell-1)}{2} + \ell m -\ell b}}{q^{\ell^2}}
	,
	\\
	\jacprod{ q^{\ell(\ell-k)-\ell(\ell-m)} }{q^{\ell^2}}
		&=
		\jacprod{ q^{\ell m-\ell k} }{q^{\ell^2}}
		=
		-q^{\ell m-\ell k}\jacprod{q^{\ell k-\ell m}}{q^{\ell^2}}
	,\\
	\jacprod{  q^{ \frac{\ell(\ell-1)}{2}+\ell(\ell-m)+\ell(\ell-k)-\ell(\ell-b-1) } }{q^{\ell^2}}
		&=
		\jacprod{ q^{ 2\ell^2-(\frac{\ell(\ell-1)}{2}+\ell m + \ell k - \ell b   )   }    }{q^{\ell^2}}
		\\
		&=
		\jacprod{ q^{ -\ell^2 + \frac{\ell(\ell-1)}{2}+\ell m + \ell k - \ell b     }    }{q^{\ell^2}}
		\\
		&=
		-q^{ -\ell^2 + \frac{\ell(\ell-1)}{2}+\ell m + \ell k - \ell b }
		\jacprod{ q^{ \frac{\ell(\ell-1)}{2}+\ell m + \ell k - \ell b     }    }{q^{\ell^2}}
	,\\
	\jacprod{ q^{\frac{\ell(\ell-1)}{2}-\ell(\ell-b-1)+\ell(\ell-k)  }  }{q^{\ell^2}}
		&=
		\jacprod{ q^{\ell^2 - ( \frac{\ell(\ell-1)}{2} -\ell b + \ell k)  }  }{q^{\ell^2}}
		=
		\jacprod{ q^{ \frac{\ell(\ell-1)}{2} -\ell b + \ell k  }  }{q^{\ell^2}}
	.
\end{align*}
Thus
\begin{align*}
	&
	q^{\ell(\ell-m)-\frac{(\ell-b-1)(\ell-b)}{2}}
	\frac{\jacprod{ q^{\frac{\ell(\ell-1)}{2} + \ell(\ell-m) +\ell(\ell-k)-\ell(\ell-b-1)}, 
			q^{\ell(\ell-k)-\ell(\ell-m)}}{q^{\ell^2}}}
	{\jacprod{ q^{\ell(\ell-m)}, 
		q^{ \frac{\ell(\ell-1)}{2} + \ell(\ell-m) -\ell(\ell-b-1)},
		q^{\frac{\ell(\ell-1)}{2}-\ell(\ell-b-1)+\ell(\ell-k)}
	}{q^{\ell^2}}}
	\\
	&=
	q^{\ell(\ell-m) - \frac{(\ell-b-1)(\ell-b)}{2} - \ell^2 
		+ \frac{\ell(\ell-1)}{2} + 2 \ell m - \ell b }
	\frac{\jacprod{q^{\frac{\ell(\ell-1)}{2} + \ell m +\ell k-\ell b}, 
		q^{\ell k-\ell m}}{q^{\ell^2}}} 
	{\jacprod{q^{\ell m}, q^{\frac{\ell(\ell-1)}{2} + \ell m -\ell b},
		q^{\frac{\ell(\ell-1)}{2}-\ell b+\ell k}}{q^{\ell^2}}}
	\\
	&=
	q^{\ell m - \frac{b(b+1)}{2}  }
	\frac{\jacprod{q^{\frac{\ell(\ell-1)}{2} + \ell m +\ell k-\ell b}, 
		q^{\ell k-\ell m}}{q^{\ell^2}}} 
	{\jacprod{q^{\ell m}, q^{\frac{\ell(\ell-1)}{2} + \ell m -\ell b},
		q^{\frac{\ell(\ell-1)}{2}-\ell b+\ell k}}{q^{\ell^2}}}
.	
\end{align*}
For the $k=0$ term we have
\begin{align*}
	&	
	q^{\ell(\ell-m)-\frac{(\ell-b-1)(\ell-b)}{2}}
		\frac{\jacprod{
			q^{\frac{\ell(\ell-1)}{2} + \ell(\ell-m) -\ell(\ell-b-1)}, 
			q^{-\ell(\ell-m)}}{q^{\ell^2}}}
		{\jacprod{
			q^{\ell(\ell-m)}, 
			q^{\frac{\ell(\ell-1)}{2} + \ell(\ell-m) -\ell(\ell-b-1)},
			q^{\frac{\ell(\ell-1)}{2}-\ell(\ell-b-1)}}{q^{\ell^2}}
		}
	\\
	&=
		-q^{\ell(\ell-m)-\frac{(\ell-b-1)(\ell-b)}{2} - \ell^2 + 2\ell m 
			+ \frac{\ell(\ell-1)}{2} - \ell b  }
		\frac{\jacprod{
			q^{\frac{\ell(\ell-1)}{2} + \ell m -\ell b}, 
			q^{-\ell m}}{q^{\ell^2}}}
		{\jacprod{
			q^{\ell m}, 
			q^{\frac{\ell(\ell-1)}{2} + \ell m -\ell b},
			q^{\frac{\ell(\ell-1)}{2}-\ell b}}{q^{\ell^2}}
		}
	\\
	&=
		-q^{\ell m-\frac{b(b+1)}{2}  }
		\frac{\jacprod{
			q^{\frac{\ell(\ell-1)}{2} + \ell m -\ell b}, 
			q^{-\ell m}}{q^{\ell^2}}}
		{\jacprod{
			q^{\ell m}, 
			q^{\frac{\ell(\ell-1)}{2} + \ell m -\ell b},
			q^{\frac{\ell(\ell-1)}{2}-\ell b}}{q^{\ell^2}}
		}
.
\end{align*}
With these identities we have that
\begin{align*}
	S_\ell(\ell-b-1)
	&\equiv
		\frac{2(b+1)(-1)^{b+1} q^{\ell m -\frac{b(b+1)}{2}} \aqprod{q}{q}{\infty}^3}
		{ \aqprod{q^{\ell^2}}{q^{\ell^2}}{\infty}\jacprod{q^{\ell m}}{q^{\ell^2}}}
		T\left( \tfrac{\ell(\ell-1)}{2} + \ell m -\ell b, \ell m, \ell^2   \right)
		\\&\quad
		-
		(-1)^{\frac{\ell+1}{2}+b} q^{\frac{\ell^2-1}{8}-\frac{b(b+1)}{2}+\ell m}
		\frac{\aqprod{q^{\ell^2}}{q^{\ell^2}}{\infty}^2}
			{\jacprod{q^{\ell m}, q^{\frac{\ell(\ell-1)}{2} + \ell m -\ell b}}{q^{\ell^2}}}
		\\&\quad
		\times
		\sum_{\substack{k=0,\\k\not\equiv b+\frac{1}{2}\pmod{\ell}}}^{\ell-1}
		(-1)^{k} q^{\frac{k(k-\ell)}{2}} (-k+b+\tfrac{1}{2})(-k+b+\tfrac{3}{2})
		\frac{\jacprod{q^{\frac{\ell(\ell-1)}{2} + \ell m +\ell k-\ell b}, q^{\ell k-\ell m}}{q^{\ell^2}}} 
			{\jacprod{q^{\frac{\ell(\ell-1)}{2}-\ell b+\ell k}}{q^{\ell^2}}}
	\\&\quad
	\pmod{\ell}
	.
\end{align*}
\end{proof}

While Lemmas \ref{MainLemmaForSSeries} and \ref{SecondLemmaForSSeries}
may appear complicated, their use is 
quite simple. We are trying to determine
formulas modulo $\ell$ for the $\ell$-dissections
of $U(q)$ and $V(q)$. 
Since 
\begin{align*}
	\frac{1}{(1-q^n)^2}
	&=
	\sum_{m=0}^\infty (m+1) q^{nm}
	\equiv
	\sum_{k=0}^{\ell-1} (k+1) q^{nk}
	\sum_{m=0}^\infty q^{\ell mn}
	\equiv
	\sum_{k=0}^{\ell-1} 
	\frac{(k+1)q^{nk}}{1-q^{\ell n}}
	\equiv
	\sum_{k=0}^{\ell-2} 
	\frac{(k+1)q^{nk}}{1-q^{\ell n}}
	\pmod{\ell},
\end{align*}
we have that
\begin{align*}
	U(q)
	&\equiv
	\frac{-1}{2\aqprod{q}{q}{\infty}^3}	
	\sum_{b=0}^{\ell-2}
	(b+1)	S_{\ell}(b)
	\pmod{\ell}
	,\\
	V(q)
	&\equiv
	\frac{-1}{2\aqprod{q}{q}{\infty}^3}	
	\sum_{b=0}^{\ell-2}
	(b+1)	S_{\ell}(b+1)
	\pmod{\ell}
	.
\end{align*}
For $\ell\ge 5$, we write these sums as
\begin{align}
	\label{EqUToS}
	U(q)
	&\equiv
	\frac{-1}{2\aqprod{q}{q}{\infty}^3}	
	\sum_{b=0}^{\ell-2}(b+1)S_\ell(b)
	\nonumber\\
	&\equiv
	\frac{-1}{2\aqprod{q}{q}{\infty}^3}	
	\bigg(
		S_\ell(0)		
		+
		\sum_{b=1}^{\frac{\ell-3}{2}}(b+1)S_\ell(b)
		+
		\tfrac{(\ell+1)}{2}S_\ell\left(\tfrac{\ell-1}{2}\right)
		+
		\sum_{b=\frac{\ell+1}{2}}^{\ell-2}(b+1)S_\ell(b)		
	\bigg)
	\nonumber\\
	&\equiv
	\frac{-1}{2\aqprod{q}{q}{\infty}^3}	
	\bigg(
		S_\ell(0)		
		+	
		\sum_{b=1}^{\frac{\ell-3}{2}}(b+1)S_\ell(b)
		+
		\tfrac{1}{2}S_\ell \left(\tfrac{\ell-1}{2}\right)
		+
		\sum_{b=1}^{\frac{\ell-3}{2}}(\ell-b)S_\ell(\ell-1-b)
	\bigg)
	\nonumber\\
	&\equiv
	\frac{-1}{2\aqprod{q}{q}{\infty}^3}	
	\bigg(
		S_\ell(0)		
		+
		\tfrac{1}{2}S_\ell\left(\tfrac{\ell-1}{2}\right)
		+
		\sum_{b=1}^{\frac{\ell-3}{2}} (b+1)S_\ell(b) - bS_\ell(\ell-1-b) 
	\bigg)
	\pmod{\ell}
	,\\
	\label{EqVToS}
	V(q)
	&\equiv
	\frac{-1}{2\aqprod{q}{q}{\infty}^3}	
	\bigg(
		-
		\tfrac{1}{2}S_\ell\left(\tfrac{\ell-1}{2}\right)
		-
		S_\ell(\ell-1)
		+
		\sum_{b=0}^{\frac{\ell-5}{2}} (b+1)S_\ell(b+1) - (b+2)S_\ell(\ell-2-b) 
	\bigg)
	\pmod{\ell}
	.
\end{align}
Lemmas \ref{MainLemmaForSSeries} and \ref{SecondLemmaForSSeries} give us
congruent expressions for $U(q)$ and $V(q)$ that 
consist of at most $\frac{\ell+1}{2}$ of the
$T(z,w,q)$ series that each contribute to only one of the
$\ell$-terms of the dissection
and infinite products which we can handle separately from the 
$T(z,w,q)$ series. That is to say, Lemmas \ref{MainLemmaForSSeries} and 
\ref{SecondLemmaForSSeries} turn the proof of Theorem \ref{TheoremMain}
into a verification of congruences and identities between infinite products.
We now proceed with the proof of Theorem \ref{TheoremMain}.

\section{Proof of Theorem \ref{TheoremMain}.}

\begin{proof}[Proof of (\ref{EqTheoremUMod3}) and (\ref{EqTheoremVMod3}).]
We recall that here $\ell=3$ and $P(a)=[q^{3a};q^9]_\infty$.
We have that
\begin{align*}
	U(q)
	&\equiv
	\frac{-1}{2E(1)^3}	
   (S_3(0) + 2S_3(1))
	\pmod{3}
	,\\
	V(q)
	&\equiv
	\frac{-1}{2E(1)^3}	
   (S_3(1) + 2S_3(2))
	\pmod{3}
	.
\end{align*}
We expand $S_3(0)$ and $S_3(1)$ with Lemma \ref{MainLemmaForSSeries} with $m=1$
and $S_3(2)$ with Lemma \ref{SecondLemmaForSSeries} with $m=1$. After elementary
simplifications we have that
\begin{align*}
   S_3(0) + 2_3S(1)
	&\equiv
		\frac{qE(9)^2}{P(1)}
		+
		\frac{2q^2E(1)^3 T(3, 3, 9)}{E(9)P(1)} 
	\pmod{3}
	,\\
	S_3(1) + 2S_3(2)
	&\equiv
		\frac{2q^3E(1)^3 T(6, 3, 9)}{E(9)P(1) }
		+
		\frac{q^2E(1)^3 T(3, 3 ,9)}{E(9)P(1)}		
	\pmod{3}
.
\end{align*}
But $E(1)^3=\aqprod{q}{q}{\infty}^3\equiv\aqprod{q^3}{q^3}{\infty}=E(3)\pmod{3}$, and so we see
that
\begin{align*}
	U(q)
	&\equiv
		\frac{qE(9)^2}{E(3)P(1)}
		+
		\frac{2q^2 T(3, 3, 9)}{E(9)P(1)} 
	\pmod{3}
	,\\
	V(q)
	&\equiv
		\frac{2q^3 T(6, 3, 9)}{E(9)P(1) }
		+
		\frac{q^2 T(3, 3 ,9)}{E(9)P(1)}
	\pmod{3}
,
\end{align*}
which are (\ref{EqTheoremUMod3}) and (\ref{EqTheoremVMod3}), respectively.

\end{proof}

\begin{proof}[Proof of (\ref{EqTheoremUMod5}).]
We recall here $\ell=5$ and $P(a)=[q^{5a};q^{25}]_\infty$.
Expanding (\ref{EqUToS}), when $\ell=5$, with Lemmas \ref{MainLemmaForSSeries} 
and \ref{SecondLemmaForSSeries} (with $m=1$ in each application), yields that
\begin{align}
	\label{EqUMod5Eq1}
	E(1)^3U(q)
	&\equiv
		\frac{2q^3 E(25)^2}{P(2)}
		+
		\frac{q E(25)^2 P(2)}
			{P(1)^2}
		+\frac{4q^2 E(1)^3 T(5, 5, 25)}
			{E(25) P(1)}
		+
		\frac{3q^2 E(25)^2}{P(1)}
		+
		\frac{4q^4 E(1)^3T(10, 5, 25)}
			{E(25) P(1)}	
	\nonumber\\&\quad		
	\pmod{5}		
\end{align}
Here we note that 
\begin{align*}
	\frac{1}{E(1)^3}
	=
	\frac{1}{\aqprod{q}{q}{\infty}^{3}}
	&\equiv  
	\frac{\aqprod{q}{q}{\infty}^2}{\aqprod{q^5}{q^5}{\infty}}
	=
	\frac{E(1)^2}{E(5)}
	\pmod{5}
.
\end{align*}

There are various ways to find the $5$-dissection of $E(1)$. We use
\cite[Entry 12(v)]{Berndt}, which is
\begin{align}
	\label{EqEDissection5}
	E(1)
	&=
	E(25)
	\left(
		\frac{P(2)}{P(1)}
		-
		q
		-
		\frac{q^2 P(1)}{P(2)}
	\right)		
.
\end{align}
Squaring (\ref{EqEDissection5}) gives that
\begin{align}
	\label{EqUMod5Eq2}
	E(1)^2
	&=
	E(25)^2
	\left(
		\frac{P(2)^2}{P(1)^2}
		-
		\frac{2q P(2)}{P(1)}
		-
		q^2
		+
		\frac{2q^3 P(1)}{P(2)}
		+
		\frac{q^4 P(1)^2}{P(2)^2}
	\right)
.
\end{align}
A direct calculation then gives that
\begin{align*}
	&E(1)^2
	\left(
		\frac{2q^3 E(25)^2}{P(2)}
		+
		\frac{q E(25)^2 P(2)}
			{P(1)^2}
		+
		\frac{3q^2 E(25)^2}{P(1)}
	\right)
	\\
	&\equiv
	E(25)^4
	\left(	
		\frac{q P(2)^3}{P(1)^4}
		+
		\frac{2q^6 P(1)}{P(2)^2}
		+
		\frac{q^2 P(2)^2}{P(1)^3}
		+
		\frac{2q^7 P(1)^2}{P(2)^3}
	\right)
	\pmod{5}		
.
\end{align*}
Thus multiplying (\ref{EqUMod5Eq1}) by (\ref{EqUMod5Eq2}) and reducing modulo $5$
yields
\begin{align*}
	U(q)
	&\equiv
		\frac{q E(25)^4P(2)^3}
			{E(5)P(1)^4}
		+
		\frac{2q^6 E(25)^4P(1)}
			{E(5)P(2)^2}
		+
		\frac{4q^2 T(5, 5, 25)}{E(25) P(1)} 
		+
		\frac{q^2 E(25)^4 P(2)^2}
			{E(5) P(1)^3}
		+
		\frac{2q^7 E(25)^4 P(1)^2}
			{E(5) P(2)^3}
		\\&\quad
		+
		\frac{4q^4 T(10, 5, 25)}{E(25) P(1)}	
	\pmod{5}
.
\end{align*}
We must reduce the products to complete the proof of (\ref{EqTheoremUMod5}).
For this we note that cubing (\ref{EqEDissection5}) gives
\begin{align*}
	E(1)^3 
	&=
	E(25)^3
	\left(
	\frac{P(2)^3}{P(1)^3}
	-\frac{3qP(2)^2}{P(1)^2}
	-\frac{q^6 P(1)^3}{P(2)^3}
	+5q^3
	-\frac{3q^5 P(1)^2}{P(2)^2}
	\right)
,
\end{align*}
whereas the $\ell=5$ case of (\ref{EqECubedModL}) gives that
\begin{align*}
	E(1)^3
	&\equiv
	E(25)(2qP(1)+P(2))
	\pmod{5}.
\end{align*}
By comparing the coefficients of $q$ we have that
\begin{align*}
	E(25)^3 \left( \frac{2qP(2)^2}{P(1)^2} + \frac{4q^6 P(1)^3}{P(2)^3} \right)
	&\equiv
	2qE(25)P(1)
	\pmod{5}
	,
\end{align*}
so that
\begin{align}
	\label{EqUMod5ProductReduction}
	\frac{P(2)^2}{P(1)^3} + \frac{2q^5P(1)^2}{P(2)^3}
	&\equiv
	\frac{1}{E(25)^2}
	\pmod{5}
	.
\end{align}
Thus
\begin{align*}
	\frac{q E(25)^4P(2)^3}{E(5)P(1)^4}
	+
	\frac{2q^6 E(25)^4P(1)}{E(5)P(2)^2}
	&\equiv
	\frac{qE(25)^2P(2)}{E(5)P(1)}	
	\pmod{5}	
	,\\
	\frac{q^2 E(25)^4P(2)^2}{E(5)P(1)^3}
	+
	\frac{2q^7 E(25)^4P(1)^2}{E(5)P(2)^3}
	&\equiv
	\frac{q^2E(25)^2}{E(5)}	
	\pmod{5},
\end{align*}
which finishes the proof of (\ref{EqTheoremUMod5}).

\end{proof}

\begin{proof}[Proof of (\ref{EqTheoremVMod5}).]
We recall that here $\ell=5$ and $P(a)=[q^{5a};q^{25}]_\infty$.
Expanding (\ref{EqVToS}),
when $\ell=5$, with Lemmas \ref{MainLemmaForSSeries} 
and \ref{SecondLemmaForSSeries} (with $m=1$ in each application and viewing
$S_\ell(\ell-1)$ as $S_\ell(\ell-0-1)$), yields that
\begin{align*}
	E(1)^3 V(q)	
	&\equiv
		\frac{4q^5 E(1)^3 T(15, 5, 25) }
			{E(25)P(1) }
		+
		\frac{q^2 E(1)^3 T(5, 5, 25)}
			{E(25) P(1)} 
		+
		\frac{4q^3 E(25)}{ P(2)}
		+
		\frac{3q^4 E(25)^2 P(1)}
			{P(2)^2}
	\pmod{5}.
\end{align*}
Again we use that 
$\frac{1}{\aqprod{q}{q}{\infty}^3}\equiv\frac{\aqprod{q}{q}{\infty}^2}{\aqprod{q^5}{q^5}{\infty}}\pmod{5}$
and with (\ref{EqUMod5Eq2}) we have that
\begin{align*}
	E(1)^2
	\left(
		\frac{4q^3 E(25)^2}{ P(2)}
		+
		\frac{3q^4 E(25)^2 P(1)}
			{P(2)^2}
	\right)
	&\equiv	
	E(25)^4
	\left(
		\frac{4q^3 P(2)}{P(1)^2}
		+
		\frac{3q^8 P(1)^3}{P(2)^4}
	\right)
	\pmod{5}.
\end{align*}
Thus 
\begin{align*}
	V(q)	
	&\equiv
		\frac{4q^5 T(15, 5, 25)}{E(25) P(1) }
		+
		\frac{q^2 T(5, 5, 25)}{E(25) P(1) } 
		+	
		\frac{4q^3 E(25)^4 P(2) }
			{E(5) P(1)^2}
		+
		\frac{3q^8 E(25)^4 P(1)^3}
			{E(5) P(2)^4}
	\pmod{5}.
\end{align*}
However, by (\ref{EqUMod5ProductReduction}) we see that
\begin{align*}
	\frac{4q^3 E(25)^4 P(2) }{E(5) P(1)^2}
	+
	\frac{3q^8 E(25)^4 P(1)^3}{E(5) P(2)^4}
	&\equiv
	\frac{4q^3 E(25)^2 P(1) }{E(5) P(2)}
	\pmod{5},
\end{align*}
which finishes the proof of (\ref{EqTheoremVMod5}).

\end{proof}

\begin{proof}[Proof of (\ref{EqTheoremUMod7}).]
We recall that here $\ell=7$ and so $P(a)=[q^{7a};q^{49}]_\infty$.
We expand the $\ell=7$ case of (\ref{EqUToS}) with Lemmas
\ref{MainLemmaForSSeries} and \ref{SecondLemmaForSSeries}, using
$m=1$ in each application. This gives that
\begin{align}
	\label{EqMod7UEq1}
	E(1)^3 U(q)
	&\equiv
		\frac{5q^7 E(49)^2 P(2)}{P(3)^2}
		+
		\frac{3q E(49)^2 P(2)}{P(1)^2}
		+
		\frac{5q E(1)^3 T(7, 7, 49)}{E(49)P(1)}  
		+
		\frac{6q^8 E(49)^2 P(1)}{P(2) P(3)}
		+
		\frac{2q^2 E(49)^2 P(2)^2}{P(1)^2 P(3)}
		\nonumber\\&\quad		
		+
		\frac{q^2 E(49)^2 P(3)}{P(1) P(2)}
		+
		\frac{2q^4 E(1)^3 T(14, 7, 49)}{E(49) P(1)}
		+
		\frac{5q^4 E(49)^2P(3)}{P(2)^2}
		+	
		\frac{4q^5 E(49)^2}{P(2)}
		+
		\frac{q^6 E(49)^2}{P(3)}
		\nonumber\\&\quad
		+
		\frac{4q^6 E(1)^3 T(21, 7, 49)}{E(49) P(1)} 
	\pmod{7}.
\end{align}

We note that
\begin{align*}
	\frac{1}{E(1)^3}
	&=
	\frac{1}{\aqprod{q}{q}{\infty}^3}
	\equiv
	\frac{\aqprod{q}{q}{\infty}^4}{\aqprod{q^7}{q^7}{\infty}}
	=\frac{E(1)^4}{E(7)}
	\pmod{7}.
\end{align*}
There are various ways to deduce the $7$-dissection of $E(1)$,
we use \cite[Entry 17(v)]{Berndt}, which is
\begin{align*}
	E(1)
	&=
	E(49)
	\left(
		\frac{P(2)}{P(1)}
		-\frac{q P(3)}{P(2)}
		-q^2
		+\frac{q^5 P(1)}{P(3)}
	\right)
	.
\end{align*}
By (\ref{EqECubedModL}) with $\ell=7$, 
\begin{align*}
	E(1)^3
	&\equiv
	E(49)
	\left(	
		5q^3P(1)
		+4qP(2)
		+P(3)
	\right)
	\pmod{7}
	.
\end{align*}
Thus
\begin{align*}
	E(1)^4
	&\equiv
	E(49)^2	
	\left(	
		\frac{P(2)P(3)}{P(1)}
		+
		\frac{4q P(2)^2}{P(1)}
		+
		\frac{6q P(3)^2}{P(2)}
		+
		\frac{5q^8 P(1)^2}{P(3)}
		+
		2q^2P(3)
		+
		q^3P(2)
		+
		\frac{2q^4 P(3)P(1)}{P(2)}
		\right.\nonumber\\&\quad\left.		
		+
		3q^5P(1)	
		+	
		\frac{4q^6 P(1)P(2)}{P(3)}	
	\right)
	\pmod{7}
	.
\end{align*}
We can simplify this slightly. We use \cite[Lemma 4]{AS}, which is
\begin{align*}
	P(b)^2P(c+d)P(c-d)
	-P(c)^2P(b+d)P(b-d)
	+q^{7(c-d)}P(d)^2P(b+c)P(b-d)
	&=
	0,
\end{align*}
With $b=3$, $c=2$, and $d=1$ this is
\begin{align}
	\label{EqUMod7AsIdent}
	P(3)^3P(1)
	-P(2)^3P(3)
	+q^7P(1)^3P(2)
	&=0
	,
\end{align}
which implies that
\begin{align*}
	\frac{5q^8 P(1)^2}{P(3)}
	&\equiv
	\frac{5q P(2)^2}{P(1)}
	+
	\frac{2q P(3)^2}{P(2)}
	\pmod{7}
.
\end{align*}
Thus
\begin{align}
	\label{EqMod7UEq2}	
	E(1)^4
	&\equiv
	E(49)^2	
	\left(	
		\frac{P(2)P(3)}{P(1)}
		+
		\frac{2q P(2)^2}{P(1)}
		+
		\frac{q P(3)^2}{P(2)}
		+
		2q^2P(3)
		+
		q^3P(2)
		+
		\frac{2q^4 P(3)P(1)}{P(2)}
		+
		3q^5P(1)
		\nonumber\right.\\&\quad\left.		
		+	
		\frac{4q^6 P(1)P(2)}{P(3)}	
	\right)
	\pmod{7}
	.
\end{align}
In fact it would appear that we can do even better. Calculations suggest that
\begin{align*}
	\frac{4q P(2)^2}{P(1)}
	+
	\frac{6q P(3)^2}{P(2)}
	+
	\frac{5q^8 P(1)^2}{P(3)}
	&\equiv
	\frac{3q E(7)^4}{E(49)^2}
	\pmod{7}
	,
\end{align*}
however for our applications this form is actually not as usual.	
Multiplying (\ref{EqMod7UEq1}) by (\ref{EqMod7UEq2}) and collecting terms yields that 
\begin{align*}
	U(q)
	&\equiv
    \frac{5q T(7, 7, 49)}{E(49)P(1)} 
	+
	\frac{2q^4 T(14, 7, 49)}{E(49)P(1)}
	+
	\frac{4q^6 T(21, 7, 49)}{E(49)P(1)} 
	+
	A_7(q)
	\pmod{7},
\end{align*}
where
\begin{align*}
	A_7(q) &=	
	\frac{\aqprod{q^{49}}{q^{49}}{\infty}^4}{\aqprod{q^7}{q^7}{\infty}}	
	\left(
	A_{7,0}(q^7) + qA_{7,1}(q^7) + q^2A_{7,2}(q^7) + q^3A_{7,3}(q^7)
	 	+ q^4A_{7,4}(q^7) + q^5A_{7,5}(q^7) 
		\right.\\&\quad\left.	 	
	 	+ q^6A_{7,6}(q^7)
	\right)
	,\\
	A_{7,0}(q^7)
	&= 
		\frac{4q^7 P(2)^2}{P(1) P(3)}
		+\frac{3q^7 P(3)}{P(2)}
		+\frac{3q^{14} P(1)^2}{P(3)^2}
	,\\ 
	A_{7,1}(q^7)
	&=	
		\frac{3 P(2)^2 P(3)}{P(1)^3}
		+\frac{4q^7 P(2)^3}{P(1) P(3)^2}
		+\frac{3q^7 P(1) P(3)^2}{ P(2)^3}
	,\\
	A_{7,2}(q^7)
	&=	
		\frac{4 P(3)^2}{P(1)^2}
		+\frac{P(2)^3}{P(1)^3}
		+\frac{q^7 P(1) P(3)}{P(2)^2}
		+\frac{2q^7 P(2)}{P(3)}
	,\\
	A_{7,3}(q^7)
	&=
		\frac{3 P(2) P(3)}{P(1)^2}
		+\frac{4 P(2)^4}{P(1)^3 P(3)}
		+\frac{P(3)^3}{P(2)^2 P(1)}
		+\frac{4q^7 P(1)}{P(2)}
		+\frac{5q^7 P(2)^2}{P(3)^2}
	,\\
	A_{7,4}(q^7)
	&=
		0	
	,\\
	A_{7,5}(q^7)
	&=
		\frac{2 P(2)^3}{P(1)^2 P(3)}
		+\frac{5 P(3)^3}{P(2)^3}
		+\frac{5q^7 P(1)^2}{P(2)^2}
		+\frac{5q^7 P(1) P(2)}{P(3)^2}	
	,\\
	A_{7,6}(q^7)
	&=
		\frac{2 P(3)^2}{P(2)^2}
		+\frac{P(2)}{P(1)}
		+\frac{4q^7 P(1)^2}{P(2) P(3)}
		+\frac{6q^7 P(1) P(2)^2}{P(3)^3}				
.	
\end{align*}
Upon verifying that $A_0\equiv A_5\equiv 0\pmod{7}$, we find the above is
equivalent to (\ref{EqTheoremUMod7}).
Multiplying (\ref{EqUMod7AsIdent}) by $\frac{3q^7}{P(1)P(2)P(3)^2}$ yields 
$A_0\equiv0\pmod{7}$.
Multiplying (\ref{EqUMod7AsIdent}) by $\frac{5}{P(1)P(2)^3}$
and $\frac{5}{P(1)^2P(3)^2}$,  gives that
\begin{align*}
	5\frac{P(3)^3}{P(2)^3} -5\frac{P(3)}{P(1)} + 5q^7\frac{P(1)^2}{P(2)^2}  &= 0
	,\\
	5\frac{P(3)}{P(1)} -5\frac{P(2)^3}{P(1)^2P(3)} + 5q^7\frac{P(1)P(2)}{P(3)^2}  &= 0
	,
\end{align*}
which when added together imply $A_5\equiv 0\pmod{7}$.
While we can use (\ref{EqUMod7AsIdent}) to reduce $A_1$, $A_2$, $A_3$, 
and $A_6$, it would appear that the reductions are not significant and so we 
leave these terms as they are.

\end{proof}

\begin{proof}{Proof of (\ref{EqTheoremVMod7}).}
We recall that here $\ell=7$ and $P(a)=[q^{7a};q^{49}]_\infty$.
We expand the $\ell=7$ case of (\ref{EqVToS}) with Lemmas
\ref{MainLemmaForSSeries} and \ref{SecondLemmaForSSeries},
with $m=1$ in each case, to find that
\begin{align}
	\label{EqMod7VEq1}
	E(1)^3V(q)
	&\equiv
	\frac{6 q^7 E(1)^3 T(28, 7, 49)}{E(49) P(1)}
	+\frac{q^7 E(49)^2 P(2)}{P(3)^2}
	+\frac{2 q E(1)^3 T(7, 7, 49)}{E(49) P(1)}
	+\frac{5 q E(49)^2 P(2)}{P(1)^2}
	\nonumber\\&\quad	
	+\frac{4 q^8 E(49)^2 P(1)}{P(2) P(3)}
	+\frac{6 q^2 E(49)^2 P(3)}{P(1) P(2)}	
	+\frac{q^2 E(49)^2 P(2)^2}{P(1)^2 P(3)}
	+\frac{q^3 E(49)^2}{P(1)}
	+\frac{q^4 E(1)^3 T(14, 7, 49)}{E(49) P(1)}
	\nonumber\\&\quad	
	+\frac{3 q^4 E(49)^2 P(2)}{P(1) P(3)}
	+\frac{3 q^4 E(49)^2 P(3)}{P(2)^2}
	+\frac{2 q^5 E(49)^2}{P(2)}
	+\frac{5 q^6 E(1)^3 T(21, 7, 49)}{E(49) P(1)}
	+\frac{3 q^6 E(49)^2}{P(3)}
\nonumber\\&\quad
\pmod{7}.
\end{align}
We multiply (\ref{EqMod7VEq1}) by (\ref{EqMod7UEq2}) to find that
\begin{align*}
	V(q)
	&\equiv
	\frac{6 q^7 T(28, 7, 49)}{E(49)P(1)}
	+\frac{2 q T(7, 7, 49)}{E(49)P(1)}
	+\frac{5 q^6 T(21, 7, 49)}{E(49) P(1)}
	+\frac{q^4 T(14, 7, 49)}{E(49)P(1)}
	+B_7(q)
	\pmod{7}
,
\end{align*}
where
\begin{align*}
	B_7(q)
	&=
	\frac{E(49)^4}{E(7)}\bigg(
		B_{7,0}(q^7)+qB_{7,1}(q^7)+q^2B_{7,2}(q^7)+q^3B_{7,3}(q^7)
		+q^4B_{7,4}(q^7)+q^5B_{7,5}(q^7)+q^6B_{7,6}(q^7)
	\bigg)
	,\\
	B_{7,0}(q^7)
	&=
		\frac{2 q^7 P(3)}{P(2)}
		+\frac{5 P(2)^2 q^7}{P(3) P(1)}
		+\frac{2 q^{14} P(1)^2}{P(3)^2}
	,\\
	B_{7,1}(q^7)
	&=
		\frac{5 P(2)^2 P(3)}{P(1)^3}
		+\frac{6 q^7 P(2)^3}{P(1) P(3)^2}
		+\frac{6 q^7 P(1) P(3)^2}{P(2)^3}
		+4 q^7
	,\\
	B_{7,2}(q^7)
	&=
		\frac{4 P(3)^2}{P(1)^2}
		+\frac{4 P(2)^3}{P(1)^3}
		+\frac{5 q^7 P(2)}{P(3)}
		+\frac{3 q^7 P(3) P(1)}{P(2)^2}
	,\\
	B_{7,3}(q^7)
	&=
		\frac{2 P(2)^4}{P(1)^3 P(3)}
		+\frac{6 P(3)^3}{P(2)^2 P(1)}
		+\frac{3 P(2) P(3)}{P(1)^2}
		+\frac{4 q^7 P(1)}{P(2)}
		+\frac{6 q^7 P(2)^2}{P(3)^2}
	,\\
	B_{7,4}(q^7)
	&=
		\frac{5 P(2)^2}{P(1)^2}
		+\frac{2 P(3)^2}{P(2) P(1)}
		+\frac{2 q^7 P(1)}{P(3)}
	,\\
	B_{7,5}(q^7)
	&=
		\frac{P(3)}{P(1)}
		+\frac{3 P(3)^3}{P(2)^3}
		+\frac{q^7 P(1)^2}{P(2)^2}
		+\frac{q^7 P(1) P(2)}{P(3)^2}
	,\\
	B_{7,6}(q^7)
	&=
		\frac{3 P(2)}{P(1)}
		+\frac{6 P(3)^2}{P(2)^2}
		+\frac{5 q^7 P(1)^2}{P(2) P(3)}
		+\frac{4 q^7 P(1) P(2)^2}{P(3)^3}
.
\end{align*}
We use a computer algebra system (in particular we used Maple) to reduce the
$B_{7,i}$
with (\ref{EqUMod7AsIdent}). By doing so we find each $B_{7,i}$ reduces modulo 
$7$ to the corresponding products in (\ref{EqTheoremVMod7}), which completes
the proof of (\ref{EqTheoremVMod7}).

\end{proof}

\begin{proof}[Proof of (\ref{EqTheoremUMod13}).]
We recall here $\ell=13$ and $P(a)=[q^{13a};q^{169}]_\infty$.
We expand the $\ell=13$ case of (\ref{EqUToS}) with Lemmas
\ref{MainLemmaForSSeries} and \ref{SecondLemmaForSSeries},
using $m=1$ in each application, to obtain 
\begin{align}
	\label{EqMod13UEq1}
	E(1)^3U(q)
	&\equiv
	\tfrac{12q^{13}E(169)^2 P(2)}{P(1) P(4)}
	+\tfrac{12q^{13}E(169)^2 P(2) P(4)}{P(1) P(3) P(6)}
	+\tfrac{9q^{13}E(169)^2 P(5)}{P(3) P(4)}
	+\tfrac{8q^{13}E(169)^2 P(3) P(4)}{P(1) P(5)^2}
	+\tfrac{9qE(169)^2 P(5)}{P(1) P(3)}
	\nonumber\\&\quad	
	+\tfrac{qE(169)^2 P(2)}{P(1)^2}
	+\tfrac{3qE(169)^2 P(3) P(6)}{P(1) P(2) P(5)}
	+\tfrac{12q^{14}E(169)^2 P(4)}{P(2) P(6)}
	+\tfrac{q^2E(169)^2 P(3)}{P(1) P(2)}
	+\tfrac{2q^2E(169)^2 P(2) P(5)}{P(1)^2 P(6)}
	\nonumber\\&\quad	
	+\tfrac{7q^{15}E(169)^2}{P(3)}
	+\tfrac{4q^{15}E(169)^2 P(3)^2}{P(1) P(4) P(6)}
	+\tfrac{2q^{15}E(169)^2 P(2)}{P(1) P(5)}
	+\tfrac{12q^3E(1)^3 T(39, 13, 169)}{E(169) P(1)}
	\nonumber\\&\quad	
	+\tfrac{3q^3E(169)^2 P(2) P(4)}{P(1)^2 P(5)}
	+\tfrac{10q^3E(169)^2 P(3) P(5)}{P(1) P(2) P(6)}
	+\tfrac{6q^3E(169)^2 P(4)}{P(1) P(3)}
	+\tfrac{8q^3E(169)^2 P(6)}{P(1) P(4)}
	+\tfrac{7q^{16}E(169)^2 P(2)}{P(1) P(6)}
	\nonumber\\&\quad
	+\tfrac{12q^4E(169)^2 P(5)}{P(1) P(4)}
	+\tfrac{5q^4E(169)^2 P(4)}{P(2)^2}
	+\tfrac{4q^4E(169)^2 P(2) P(3)}{P(1)^2 P(4)}
	+\tfrac{10q^4E(169)^2 P(4) P(5)}{P(1) P(3) P(6)}
	+\tfrac{q^{17}E(169)^2 P(6)}{P(4) P(5)}
	\nonumber\\&\quad
	+\tfrac{10q^{-8}E(1)^3 T(13, 13, 169)}{E(169) P(1)}
	+\tfrac{3q^{-8}E(169)^2 P(5)}{P(1)^2}
	+\tfrac{6q^5E(169)^2 P(3) P(4)}{P(1) P(2) P(6)}
	+\tfrac{5q^5E(169)^2 P(2)^2}{P(1)^2 P(3)}
	\nonumber\\&\quad	
	+\tfrac{9q^{18}E(169)^2 P(2) P(4)}{P(1) P(5) P(6)}
	+\tfrac{10q^{18}E(169)^2}{P(4)}
	+\tfrac{4q^{-7}E(169)^2 P(4) P(6)}{P(1)^2 P(5)}
	+\tfrac{2q^6E(169)^2}{P(1)}
	+\tfrac{9q^6E(169)^2 P(5)}{P(2) P(3)}
	\nonumber\\&\quad	
	+\tfrac{6q^6E(169)^2 P(4)^2}{P(1) P(3) P(6)}
	+\tfrac{7q^{19}E(169)^2 P(2) P(3)}{P(1) P(4) P(6)}
	+\tfrac{11q^7E(1)^3 T(52, 13, 169)}{E(169) P(1)}
	+\tfrac{2q^7E(169)^2 P(3)^2}{P(1) P(2) P(5)}
	\nonumber\\&\quad	
	+\tfrac{11q^7E(169)^2 P(3) P(5)}{P(1) P(4)^2}
	+\tfrac{4q^7E(169)^2 P(2) P(6)}{P(1) P(3) P(4)}
	+\tfrac{3q^{20}E(169)^2}{P(5)}
	+\tfrac{2q^{-5}E(169)^2 P(3) P(5)}{P(1)^2 P(4)}
	+\tfrac{q^8E(169)^2 P(4)}{P(2) P(3)}
	\nonumber\\&\quad	
	+\tfrac{7q^8E(169)^2 P(4)}{P(1) P(5)}
	+\tfrac{5q^{21}E(169)^2}{P(6)}
	+\tfrac{2q^9E(169)^2 P(4)}{P(1) P(6)}
	+\tfrac{11q^9E(169)^2 P(3)}{P(1) P(4)}
	+\tfrac{2q^9E(169)^2 P(2) P(6)}{P(1) P(3) P(5)}
	\nonumber\\&\quad
	+\tfrac{q^{22}E(169)^2 P(5)}{P(6)^2}
	+\tfrac{q^{10}E(1)^3 T(65, 13, 169)}{E(169) P(1)}
	+\tfrac{10q^{10}E(169)^2 P(3) P(5)}{P(1) P(4) P(6)}
	+\tfrac{8q^{10}E(169)^2 P(2)}{P(1) P(3)}
	+\tfrac{4q^{10}E(169)^2 P(6)}{P(2) P(5)}
	\nonumber\\&\quad	
	+\tfrac{10q^{23}E(169)^2 P(4)}{P(5) P(6)}
	+\tfrac{8q^{-2}E(1)^3 T(26, 13, 169)}{E(169) P(1)}
	+\tfrac{5q^{-2}E(169)^2 P(4) P(6)}{P(1) P(2) P(5)}
	+\tfrac{5q^{-2}E(169)^2 P(2) P(4)}{P(1)^2 P(3)}
	\nonumber\\&\quad	
	+\tfrac{q^{11}E(169)^2 P(3)}{P(1) P(5)}
	+\tfrac{7q^{11}E(169)^2}{P(2)}
	+\tfrac{5q^{-1}E(169)^2 P(3) P(6)}{P(1) P(2) P(4)}
	+\tfrac{6q^{-1}E(169)^2 P(4)}{P(1) P(2)}
	+\tfrac{4q^{12}E(1)^3 T(78, 13, 169)}{E(169) P(1)}
	\nonumber\\&\quad	
	+\tfrac{8q^{12}E(169)^2 P(5)}{P(2) P(6)}
	+\tfrac{11q^{12}E(169)^2 P(2) P(4)}{P(1) P(3) P(5)}
	\pmod{13}
.
\end{align}

We note that
\begin{align*}
	\frac{1}{E(1)^3}
	&=
	\frac{1}{\aqprod{q}{q}{\infty}^3}
	\equiv
	\frac{\aqprod{q^{10}}{q^{10}}{\infty}^{10}}{\aqprod{q^{13}}{q^{13}}{\infty}}
	=
	\frac{E(1)^{10}}{E(13)}
.
\end{align*}
By \cite[(51)]{BilgiciEkin}, with a small correction to the $q^7$ term,
we have that
\begin{align*}
	E(q)^{10}
	&\equiv
	E(169)^2
	\left(	
		P(2)P(4)P(5)P(6)
		+3qP(3)^2 P(4)P(6)	
        +9q^2P(1)P(5)P(6)^2
		\right.\\&\quad
		+9q^3P(2)P(3)P(5)P(6)
        +12q^4P(2)P(3)P(5)^2
        +11q^5P(2)P(3)P(4)P(6)
		+6q^5P(1)P(4)P(5)P(6)        
		\\&\quad		
		+5q^{18}P(1)P(2)P(3)P(5)
		+q^{32}P(1)^2P(2)P(3)
        +9q^{20}P(1)P(2)P(3)P(4)        		
        +4q^8P(1)P(4)^2P(5)
        \\&\quad
        +10q^9P(2)^2P(4)P(6)
		+q^{10}P(1)P(3)P(4)P(6)
		+10q^{11}P(1)P(3)P(4)P(5)   
		\\&\quad\left.    
		+3q^{12}P(1)P(2)P(5)P(6)                
	\right)
	\pmod{13}
\end{align*}
As noted in \cite{ObrienThesis} and \cite{BilgiciEkin},
one can use \cite[(LVII2)]{MolkTannery}, to get that
\begin{align}
	\label{MolkTanneryProductId}
	0
	&=
	P(a+d)P(a-d)P(b+c)P(b-c)
	-P(a+c)P(a-c)P(b+d)P(b-d)
	\nonumber\\&\quad
	+q^{13(b-c)}P(a+b)P(a-b)P(c+d)P(c-d)
	.
\end{align}
Setting $a=5$, $b=3$, $c=2$, and $d=1$ in (\ref{MolkTanneryProductId})
gives
\begin{align*}
	P(6)P(4)P(5)P(1)-P(6)P(3)P(4)P(2)+q^{13}P(5)P(2)P(3)P(1)=0
	, 
\end{align*}
so that
\begin{align}
	\label{EqMod13UEq2}
	E(q)^{10}
	&\equiv
	E(169)^2
	\left(	
		P(2)P(4)P(5)P(6)
		+3qP(3)^2 P(4)P(6)	
        +9q^2P(1)P(5)P(6)^2
		\right.\nonumber\\&\quad
		+9q^3P(2)P(3)P(5)P(6)
        +12q^4P(2)P(3)P(5)^2
        +3q^5P(2)P(3)P(4)P(6)
		+q^5P(1)P(4)P(5)P(6)        
		\nonumber\\&\quad		
		+q^{32}P(1)^2P(2)P(3)
        +9q^{20}P(1)P(2)P(3)P(4)        		
        +4q^8P(1)P(4)^2P(5)
        \nonumber\\&\quad
        +10q^9P(2)^2P(4)P(6)
		+q^{10}P(1)P(3)P(4)P(6)
		+10q^{11}P(1)P(3)P(4)P(5)   
		\nonumber\\&\quad\left.    
		+3q^{12}P(1)P(2)P(5)P(6)                
	\right)
	\pmod{13}
\end{align}
Again it would appear we could do even perhaps, calculations suggest that
\begin{align*}
	11q^5P(2)P(3)P(4)P(6)
	+6q^5P(1)P(4)P(5)P(6)        		
	+5q^{18}P(1)P(2)P(3)P(5)
	&\equiv
	\frac{E(13)^{10}}{E(169)^2}
,
\end{align*}
however this form is not as useful for our calculations.
We multiply (\ref{EqMod13UEq1}) by (\ref{EqMod13UEq2}),
collect terms, and reduce modulo $13$, to find that
\begin{align*}
	U(q)
	&\equiv
	\tfrac{12q^3 T(39, 13, 169)}{E(169) P(1)}
	+\tfrac{10q^{-8} T(13, 13, 169)}{E(169) P(1)}
	+\tfrac{11q^7 T(52, 13, 169)}{E(169) P(1)}
	+\tfrac{q^{10} T(65, 13, 169)}{E(169) P(1)}	
	+\tfrac{8q^{-2} T(26, 13, 169)}{E(169) P(1)}
	\\&\quad	
	+\tfrac{4q^{12} T(78, 13, 169)}{E(169) P(1)}
	+A'_{13}(q)
	\pmod{13}
,
\end{align*}
where
\begin{align*}
	A'_{13}(q) &= 
		\frac{E(169)^4}{E(13)}
		\bigg(
		A'_{13,0}(q^{13})+qA'_{13,1}(q^{13})+q^2A'_{13,2}(q^{13})+q^3A'_{13,3}(q^{13})
		+q^4A'_{13,4}(q^{13})
		\\&\quad		
		+q^5A'_{13,5}(q^{13})
		+q^6A'_{13,6}(q^{13})
		+q^7A'_{13,7}(q^{13})+q^8A'_{13,8}(q^{13})+q^9A'_{13,9}(q^{13})
		+q^{10}A'_{13,10}(q^{13})
		\\&\quad		
		+q^{11}A'_{13,11}(q^{13})
		+q^{12}A'_{13,12}(q^{13})
		\bigg)
	,\\
	A'_{13,0}(q^{13})
	&=
		\tfrac{6P(2) P(5) P(4) P(6)^2}{P(1) P(3)}
		+\tfrac{2P(5)^2P(3) P(6)}{P(1)}
		+\tfrac{2P(3)^3P(6)^2}{P(1) P(2)}
		+\tfrac{12P(5)^2 P(4)^2}{P(1)}
		+\tfrac{6P(4) P(6)^3}{P(2)}
		\\&\quad		
		+\tfrac{6P(2) P(5) P(3)^2 P(6)}{P(1)^2}
		+\tfrac{5P(4)^2 P(3)^2 P(6)}{P(1) P(2)}
		+\tfrac{2q^{13}P(2)^2 P(5) P(4)^2}{P(1) P(3)}
		+\tfrac{q^{13}P(2) P(5)^2 P(3)^2}{P(1) P(4)}		
		\\&\quad	
		+\tfrac{11q^{13}P(5) P(4)^3 P(3)}{P(2) P(6)}		
		+\tfrac{10q^{13}P(2) P(4)^2 P(3) P(6)}{P(1) P(5)}
		+\tfrac{5q^{13}P(2) P(5)^2 P(4) P(3)}{P(1) P(6)}		
		+\tfrac{q^{13}P(1) P(5) P(4)^2 P(6)}{P(2) P(3)}
		\\&\quad		
		+\tfrac{q^{13}P(2)^3 P(3) P(6)}{P(1)^2}
		+\tfrac{10q^{13}P(2)^2 P(5) P(6)}{P(1)}
		+\tfrac{10q^{13}P(2) P(5)^2 P(6)}{P(3)}		
		+\tfrac{5q^{13}P(5) P(4) P(3)^2}{P(2)}
		\\&\quad		
		+\tfrac{11q^{13}P(1) P(5) P(6)^2}{P(2)}
		+{\scriptstyle q^{13} P(4)^2 P(6)}
		+{\scriptstyle 10q^{13} P(3) P(6)^2}
		+\tfrac{10q^{26}P(2)^2 P(6)^2}{P(5)}
		+\tfrac{12q^{26}P(2) P(4)^3}{P(6)}
		\\&\quad		
		+{\scriptstyle 7q^{26}P(2) P(4) P(3)}
		+{\scriptstyle 11q^{26} P(1) P(5) P(4)}		
		+\tfrac{q^{26}P(5) P(3)^3}{P(6)}
		+\tfrac{12q^{26}P(2) P(5)^3 P(3)}{P(6)^2}
		\\&\quad				
		+\tfrac{11q^{39}P(2)^2 P(3)^2}{P(6)}
		+\tfrac{2q^{39}P(1)P(3)^3}{P(5)}
		+\tfrac{4q^{39}P(1) P(2)^2 P(6)}{P(4)}
		+\tfrac{11q^{39}P(1) P(2) P(5) P(3)^2}{P(4)^2}
		+\tfrac{3q^{52}P(1)^2 P(2) P(3)}{P(5)}
	,\\
	A'_{13,1}(q^{13})
	&=
		\tfrac{3 P(4)^3 P(6)}{P(1)}
		+\tfrac{9 P(4) P(3) P(6)^2}{P(1)}
		+\tfrac{2 P(5) P(4) P(6)^2}{P(2)}
		+\tfrac{9 P(2) P(5)^2 P(4) P(6)}{P(1) P(3)}
		+\tfrac{11 P(2)^2 P(5) P(4) P(6)}{P(1)^2}
		\\&\quad		
		+\tfrac{6 P(5) P(3) P(6)^3}{P(2) P(4)}		
		+\tfrac{11 q^{13} P(2)^4 P(4) P(6)}{P(1)^2 P(3)}
		+\tfrac{6 q^{13} P(2)^2 P(4) P(6)^2}{P(1) P(5)}
		+\tfrac{11 q^{13} P(4)^2 P(3)^3 P(6)}{P(1) P(5)^2}
		\\&\quad		
		+\tfrac{10 q^{13} P(1) P(5)^2 P(4)^2}{P(2) P(3)}
		+\tfrac{9 q^{13} P(5)^2 P(4) P(3)^2}{P(2) P(6))}
		+\tfrac{5 q^{13} P(1) P(4)^2 P(3) P(6)}{P(2)^2}
		+\tfrac{2 q^{13} P(2) P(4)^2 P(3)}{P(1)}
		\\&\quad		
		+\tfrac{10 q^{13} P(2) P(4) P(6)^2}{P(3)}
		+\tfrac{4 q^{13} P(2) P(3)^2 P(6)}{P(1)}
		+\tfrac{11 q^{13} P(5) P(4)^4}{P(3) P(6)}
		+\tfrac{7 q^{13} P(1) P(5)^2 P(6)}{P(2)}
		+\tfrac{11 q^{13} P(2)^2 P(5)^2}{P(1)}
		\\&\quad		
		+\tfrac{3 q^{13} P(2) P(5)^3 P(3)^2}{P(1) P(4) P(6)}
		+{\scriptstyle 10 q^{13} P(5) P(3) P(6)}	
		+{\scriptstyle q^{13} P(5) P(4)^2 }
		+\tfrac{8 q^{26} P(1) P(2) P(5) P(6)}{P(3)}
		+\tfrac{12 q^{26} P(2)^3 P(4)^2}{P(1) P(5)}
		\\&\quad		
		+\tfrac{9 q^{26} P(2) P(5) P(3)^2}{P(4)}
		+\tfrac{q^{26} P(1) P(5)^2 P(4)}{P(6)}
		+\tfrac{q^{26} P(1) P(3) P(6)^2}{P(5)}
		+\tfrac{5 q^{26} P(4) P(3)^3}{P(5)}
		+{\scriptstyle 12 q^{26} P(2)^2 P(6) }
		\\&\quad		
		+\tfrac{8 q^{39} P(1) P(2) P(4) P(3)}{P(5)}
		+{\scriptstyle q^{39} P(1)^2 P(4)}
		+\tfrac{5 q^{52} P(1)^2 P(2) P(3)}{P(6)}
	,\\
	A'_{13,2}(q^{13})
	&=
		\tfrac{9P(4)  P(3)^3  P(6)^2}{P(1)  P(2)  P(5)}
		+\tfrac{P(2)^2  P(4)^2  P(6)^2}{P(1)^2  P(5)}
		+\tfrac{3P(2)  P(4)  P(3)^2  P(6)}{P(1)^2}
		+\tfrac{2P(5)  P(4)  P(3)  P(6)}{P(1)}
		+\tfrac{6P(5)  P(3)^2  P(6)^2}{P(1)  P(4)}
		\\&\quad
		+\tfrac{10P(2)^2  P(5)^2  P(4)}{P(1)^2}
		+\tfrac{3q^{13}  P(1)  P(5)^2  P(6)^2}{P(4)  P(3)}
		+\tfrac{11q^{13}  P(1)  P(5)  P(4)^2  P(3)}{P(2)^2}
		+\tfrac{9q^{13}  P(5)^2  P(4)^2}{P(6)}
		\\&\quad		
		+\tfrac{11q^{13}  P(4)^2  P(3)^2}{P(2)}
		+\tfrac{8q^{13}  P(2)^2  P(4)^3}{P(1)  P(3)}
		+\tfrac{3q^{13}  P(2)^2  P(4)  P(6)}{P(1)}
		+\tfrac{2q^{13}  P(2)  P(5)  P(6)^2}{P(4)}
		+\tfrac{6q^{13}  P(4)  P(3)  P(6)^2}{P(5)}
		\\&\quad		
		+\tfrac{4q^{13}  P(1)  P(4)  P(6)^2}{P(2)}
		+{\scriptstyle 9q^{13}  P(5)^2  P(3)}
		+\tfrac{11q^{26}  P(2)  P(4)^2  P(3)}{P(5)}
		+\tfrac{5q^{26}  P(2)^3  P(3)}{P(1)}
		+{\scriptstyle 7q^{26}  P(1)  P(3)  P(6)}
		\\&\quad		
		+{\scriptstyle 8q^{26}  P(2)^2  P(5)}
		+{\scriptstyle 5q^{26}  P(1)  P(4)^2}
		+\tfrac{8q^{39}  P(1)  P(2)  P(4)  P(3)}{P(6)}
		+\tfrac{2q^{39}  P(1)  P(2)^2  P(6)}{P(5)}
		+\tfrac{11q^{39}  P(1)  P(2)  P(3)^2}{P(4)}		
		\\&\quad		
		+\tfrac{q^{52}  P(1)^2  P(2)  P(5)  P(3)}{P(6)^2}
	,\\
	A'_{13,3}(q^{13})
	&=
		\tfrac{6 P(2) P(5) P(4) P(3)^2  }{P(1)^2}
		+\tfrac{6 P(4)^2 P(3) P(6)^2  }{P(1) P(5)}
		+\tfrac{8 P(5)^2 P(3)^2 P(6)  }{P(1) P(4)}
		+\tfrac{3 P(4) P(3)^3 P(6)  }{P(1) P(2)}
		+\tfrac{3 P(5)^2 P(6)^2  }{P(3)}
		\\&\quad		
		+\tfrac{5 P(4)^2 P(6)^2  }{P(2}
		+\tfrac{P(3) P(6)^3  }{P(2)}
		+\tfrac{11 P(2) P(5) P(4)^2 P(6)  }{P(1) P(3)}
		+\tfrac{5 P(2)^2 P(4)^2 P(6)  }{P(1)^2}
		+\tfrac{4 P(2) P(5) P(6)^2  }{P(1)}
		\\&\quad		
		+\tfrac{3 P(5)^2 P(4) P(3)  }{P(1)}
		+\tfrac{q^{13} P(2)^3 P(6)^2  }{P(1) P(3)}
		+\tfrac{4 q^{13} P(1) P(5) P(4)^3  }{P(2) P(3)}
		+\tfrac{8 q^{13} P(5) P(4)^2 P(3)^2  }{P(2) P(6)}
		+\tfrac{9 q^{13} P(1) P(5) P(4) P(6)  }{P(2)}
		\\&\quad
		+\tfrac{12 q^{13} P(3)^4  }{P(1)}
		+\tfrac{9 q^{13} P(2)^2 P(5) P(3) P(6)  }{P(1) P(4)}
		+\tfrac{10 q^{13} P(2) P(4) P(3)^2 P(6)  }{P(1) P(5)}
		+\tfrac{11 q^{13} P(2)^2 P(5) P(4)  }{P(1)}
		\\&\quad		
		+\tfrac{4 q^{13} P(2) P(5)^2 P(4)  }{P(3)}
		+\tfrac{5 q^{13} P(5)^3 P(3)  }{P(6)}
		+{\scriptstyle 8 q^{13} P(4)^3}
		+{\scriptstyle 6 q^{13} P(4) P(3) P(6)}
		+\tfrac{9 q^{26} P(2)^2 P(4) P(6)  }{P(5)}
		\\&\quad		
		+\tfrac{4 q^{26} P(2) P(4)^2 P(3)  }{P(6)}
		+\tfrac{3 q^{26} P(1) P(2) P(6)^2  }{P(4)}
		+\tfrac{7 q^{26} P(1) P(5) P(4)^2  }{P(6)}
		+{\scriptstyle 9 q^{26} P(1) P(5) P(3)}
		+{\scriptstyle 2 q^{26} P(2) P(3)^2}
		\\&\quad		
		+\tfrac{9 q^{39} P(1) P(2) P(5) P(4) P(3)  }{P(6)^2}
		+\tfrac{10 q^{39} P(1) P(2) P(5) P(3)^2  }{P(4) P(6)}
		+\tfrac{4 q^{39} P(1)^2 P(3) P(6)  }{P(5)}
		+{\scriptstyle 8 q^{39} P(1) P(2)^2}
		\\&\quad
		+\tfrac{10 q^{52} P(1)^2 P(2) P(4) P(3)  }{P(5) P(6)}
	,\\
	A'_{13,4}(q^{13})
	&=
		\tfrac{9 P(2) P(4)^2 P(3)^2 P(6)}{P(1)^2 P(5)}
		+\tfrac{7 P(2)^2 P(5) P(3) P(6)}{P(1)^2}
		+\tfrac{10 P(2) P(5)^2 P(4)^2}{P(1) P(3)}
		+\tfrac{4 P(5) P(4) P(3)^3}{P(1) P(2)}
		+\tfrac{3 P(2) P(5)^2 P(6)}{P(1)}
		\\&\quad		
		+\tfrac{11 P(4)^2 P(3) P(6)}{P(1)}
		+\tfrac{11 P(5) P(4)^2 P(6)}{P(2)}
		+\tfrac{P(5) P(3) P(6)^2}{P(2)}
		+\tfrac{P(3)^2 P(6)^2}{P(1)}
		+\tfrac{3 q^{13} P(2) P(4)^2 P(6)}{P(3)}
		\\&\quad		
		+\tfrac{11 q^{13} P(1) P(5) P(6)^2}{P(3)}
		+\tfrac{8 q^{13} P(5) P(3)^2 P(6)}{P(4)}
		+\tfrac{7 q^{13} P(1) P(5)^2 P(4)}{P(2)}
		+\tfrac{4 q^{13} P(2) P(5)^3}{P(4)}
		+\tfrac{3 q^{13} P(5) P(4)^3}{P(6)}
		\\&\quad		
		+\tfrac{2 q^{13} P(2)^3 P(5) P(6)}{P(1) P(3)}
		+\tfrac{12 q^{13} P(2)^2 P(4)^2 P(6)}{P(1) P(5)}
		+\tfrac{q^{13} P(2)^2 P(5)^2 P(4)}{P(1) P(6)}
		+\tfrac{q^{13} P(2)^2 P(5)^2 P(3)}{P(1) P(4)}
		\\&\quad		
		+\tfrac{2 q^{13} P(4) P(3)^3 P(6)}{P(2) P(5)}
		+{\scriptstyle 10 q^{13} P(2) P(6)^2}
		+{\scriptstyle 6 q^{13} P(5) P(4) P(3)}
		+\tfrac{4 q^{26} P(1) P(2) P(5) P(6)}{P(4)}
		\\&\quad		
		+\tfrac{5 q^{26} P(1) P(4) P(3) P(6)}{P(5)}
		+\tfrac{4 q^{26} P(2) P(5) P(3)^2}{P(6)}
		+\tfrac{4 q^{26} P(1) P(5)^2 P(4)^2}{P(6)^2}
		+{\scriptstyle 11 q^{26} P(2)^2 P(4)}
		+\tfrac{q^{39} P(1) P(2) P(3)^2}{P(5)}
		\\&\quad		
		+\tfrac{12 q^{39} P(1) P(2) P(4)^2 P(3)}{(P(5) P(6)}
		+{\scriptstyle 7 q^{39} P(1)^2 P(3)}
	,\\
	A'_{13,5}(q^{13})
	&=
		+\tfrac{3 P(2) P(5)^2 P(4) P(6)}{q^{13} P(1)^2}
		+\tfrac{5 P(2)^3 P(5) P(4) P(6)}{P(1)^2 P(3)}
		+\tfrac{4 P(2) P(5)^3}{P(1)}
		+\tfrac{7 P(5) P(6)^3}{P(4)}
		+\tfrac{4 P(2)^2 P(5)^2 P(3)}{P(1)^2}
		\\&\quad		
		+\tfrac{12 P(2) P(3)^3 P(6)}{P(1)^2}
		+\tfrac{10 P(5) P(4)^2 P(3)}{P(1)}
		+\tfrac{2 P(5) P(4) P(6)^2}{P(3)}
		+\tfrac{5 P(5) P(3)^2 P(6)}{P(1)}
		+\tfrac{12 P(5)^2 P(3) P(6)}{P(2)}
		\\&\quad		
		+\tfrac{2 P(4)^2 P(3)^2 P(6)}{P(2)^2}
		+\tfrac{10 q^{13} P(2) P(5) P(3)^3}{P(1) P(4)}
		+\tfrac{6 q^{13} P(2)^2 P(4)^2}{P(1)}
		+\tfrac{3 q^{13} P(3)^2 P(6)^2}{P(5)}
		+\tfrac{6 q^{13} P(5)^2 P(3)^2}{P(4)}
		\\&\quad		
		+\tfrac{7 q^{13} P(4) P(3)^3}{P(2)}
		+\tfrac{8 q^{13} P(2)^2 P(3) P(6)}{P(1)}
		+\tfrac{10 q^{13} P(2) P(5) P(4)^2}{P(3)}
		+\tfrac{8 q^{13} P(4)^2 P(3) P(6)}{P(5)}
		+\tfrac{2 q^{13} P(5)^2 P(4) P(3)}{P(6)}
		\\&\quad		
		+\tfrac{10 q^{13} P(1) P(5)^2 P(6)}{P(3)}
		+\tfrac{4 q^{13} P(1) P(4)^2 P(6)}{P(2)}
		+\tfrac{7 q^{13} P(2)^3 P(4) P(6)^2}{P(1) P(5) P(3)}
		+\tfrac{11 q^{13} P(2) P(4)^2 P(3)^2 P(6)}{P(1) P(5)^2}
		\\&\quad		
		+\tfrac{q^{26} P(1) P(4)^3}{P(6)}
		+\tfrac{10 q^{26} P(2)^2 P(5) P(4)}{P(6)}
		+\tfrac{8 q^{26} P(2)^2 P(5) P(3)}{P(4)}
		+\tfrac{9 q^{26} P(2) P(4) P(3)^2}{P(5)}
		+\tfrac{5 q^{26} P(1) P(3)^2 P(6)}{P(4)}
		\\&\quad		
		+\tfrac{11 q^{39} P(1) P(2)^2 P(4)}{P(5)}
		+\tfrac{8 q^{39} P(1)^2 P(5) P(3)}{P(6)}
	,\\
	A'_{13,6}(q^{13})
	&=
		\tfrac{4 P(2) P(4)^2 P(6)^2}{q^{13} P(1)^2}
		+\tfrac{9 P(5) P(4) P(3)^2 P(6)}{q^{13} P(1)^2}
		+\tfrac{4 P(2) P(5) P(3) P(6)^2}{P(1) P(4)}
		+\tfrac{5 P(5)^2 P(3)^2}{P(1)}
		+\tfrac{4 P(5)^2 P(6)^2}{P(4)}
		\\&\quad		
		+\tfrac{7 P(4)^3 P(6)}{P(2)}
		+\tfrac{6 P(2)^2 P(4) P(3) P(6)}{P(1)^2}
		+\tfrac{11 P(2)^2 P(5)^3 P(3)}{P(1)^2 P(6)}
		+\tfrac{6 P(2) P(5) P(4) P(6)}{P(1)}
		+\tfrac{9 P(4) P(3)^2 P(6)^2}{P(1) P(5)}
		\\&\quad		
		+\tfrac{6 P(1) P(5) P(4) P(6)^2}{P(2)^2}
		+\tfrac{4 P(5)^2 P(4) P(6)}{P(3)}
		+\tfrac{5 P(4)^2 P(3)^3}{P(1) P(2)}
		+\tfrac{3 P(4) P(3) P(6)^2}{P(2)}
		+\tfrac{9 q^{13} P(1) P(6)^3}{P(4)}
		\\&\quad		
		+\tfrac{2 q^{13} P(2)^3 P(4) P(6)}{P(1) P(3)}
		+\tfrac{12 q^{13} P(2)^2 P(5) P(6)^2}{P(4) P(3)}
		+\tfrac{q^{13} P(2) P(4)^2 P(3)^2}{P(1) P(5)}
		+\tfrac{7 q^{13} P(2) P(5)^2 P(3) P(6)}{P(4)^2}
		\\&\quad		
		+\tfrac{9 q^{13} P(2) P(5)^2 P(3)^3}{P(1) P(4) P(6)}
		+\tfrac{12 q^{13} P(2)^2 P(5) P(3)}{P(1)}
		+\tfrac{3 q^{13} P(2) P(4) P(6)^2}{P(5)}
		+\tfrac{11 q^{13} P(1) P(5) P(4)^2}{P(2)}
		+{\scriptstyle 6 q^{13} P(2) P(5)^2}
		\\&\quad
		+{\scriptstyle q^{13} P(3)^2 P(6)}
		+{\scriptstyle 10 q^{13} P(4)^2 P(3)}
		+\tfrac{4 q^{26} P(2)^2 P(4)^2}{P(5)}
		+{\scriptstyle 9 q^{26} P(1) P(2) P(6)}
		+\tfrac{6 q^{26} P(1) P(5) P(4) P(3)}{P(6)}
		\\&\quad		
		+\tfrac{8 q^{39} P(1) P(2) P(4) P(3)^2}{P(5)^2}
		+\tfrac{12 q^{39} P(1) P(2)^2 P(4)}{P(6)}
		+\tfrac{12 q^{39} P(1) P(2)^2 P(3)}{P(4)}
		+\tfrac{9 q^{39} P(1)^2 P(2) P(5)}{P(4)}
	,\\
	A'_{13,7}(q^{13})
	&=
		+\tfrac{12 P(4)^2 P(3)^2 P(6)^2}{q^{13} P(1)^2 P(5)}
		+\tfrac{P(5)^2 P(6)^2}{q^{13} P(1)}
		+\tfrac{5 P(4)^3 P(3)}{P(1)}
		+\tfrac{11 P(5) P(4)^3}{P(2)}
		+\tfrac{8 P(2) P(5)^2 P(4)}{P(1)}
		+\tfrac{11 P(4) P(3)^2 P(6)}{P(1)}
		\\&\quad		
		+\tfrac{3 P(5)^3 P(3)^2}{(P(1) P(6)}
		+\tfrac{11 P(2)^3 P(4)^2 P(6)}{P(1)^2 P(3)}
		+\tfrac{3 P(2)^2 P(5) P(4) P(3)}{P(1)^2}
		+\tfrac{10 P(2)^2 P(5) P(6)^2}{P(1) P(3)}
		+\tfrac{11 P(2) P(4)^2 P(6)^2}{P(1) P(5)}
		\\&\quad		
		+\tfrac{4 P(5) P(4) P(3) P(6)}{P(2)}
		+\tfrac{10 P(2)^2 P(5) P(3)^2 P(6)}{P(1)^2 P(4)}
		+\tfrac{5 q^{13} P(2) P(4)^3}{P(3)}
		+\tfrac{7 q^{13} P(2) P(3)^3}{P(1)}
		+\tfrac{3 q^{13} P(2)^2 P(4) P(3) P(6)}{P(1) P(5)}
		\\&\quad		
		+\tfrac{4 q^{13} P(1) P(4) P(3) P(6)^2}{P(2) P(5)}
		+\tfrac{9 q^{13} P(2) P(3) P(6)^2}{P(4)}
		+\tfrac{7 q^{13} P(5) P(4)^2 P(3)}{P(6)}
		+\tfrac{12 q^{13} P(1) P(5) P(6)^2}{P(4)}
		\\&\quad		
		+\tfrac{6 q^{13} P(2)^2 P(5)^2 P(3)}{(P(1) P(6)}
		+\tfrac{10 q^{13} P(1) P(5) P(4) P(6)}{P(3)}
		+\tfrac{6 q^{13} P(1) P(5)^2 P(4)^2}{(P(2) P(6)}
		+{\scriptstyle 5 q^{13} P(2) P(4) P(6)}
		+{\scriptstyle 7 q^{13} P(5) P(3)^2}
		\\&\quad		
		+\tfrac{7 q^{26} P(2) P(4)^2 P(3)^2}{P(5)^2}
		+\tfrac{10 q^{26} P(1) P(4)^2 P(3)}{P(5)}
		+\tfrac{3 q^{26} P(1) P(3)^2 P(6)}{P(5)}
		+\tfrac{4 q^{26} P(2)^2 P(4)^2}{P(6)}
		\\&\quad		
		+\tfrac{10 q^{26} P(1) P(5)^2 P(4) P(3)}{P(6)^2}
		+{\scriptstyle 5 q^{26} P(2)^2 P(3)}
		+{\scriptstyle q^{26} P(1) P(2) P(5)}
		+\tfrac{12 q^{39} P(1)^2 P(4) P(3)}{P(6)}
	,\\
	A'_{13,8}(q^{13})
	&=
		\tfrac{3 P(2) P(5)^2 P(3) P(6)}{q^{13} P(1)^2}
		+\tfrac{10 P(4) P(6)^3}{q^{13} P(1)}
		+\tfrac{6 P(2)^3 P(5) P(6)}{P(1)^2}
		+\tfrac{2 P(2) P(4)^2 P(6)}{P(1)}
		+\tfrac{8 P(2) P(3) P(6)^2}{P(1)}
		\\&\quad
		+\tfrac{9 P(5) P(4)^2 P(6)}{P(3)}
		+\tfrac{6 P(5) P(4) P(3)^2}{P(1)}
		+\tfrac{5 P(5)^2 P(4) P(3)}{P(2)}
		+\tfrac{9 P(2)^2 P(4)^2 P(3) P(6)}{P(1)^2 P(5)}
		+\tfrac{6 P(4) P(3)^4 P(6)}{P(1) P(2) P(5)}
		\\&\quad		
		+\tfrac{9 P(2)^2 P(5)^2 P(3)^2}{P(1)^2 P(4)}
		+\tfrac{3 P(2) P(5)^3 P(4)}{P(1) P(6)}
		+\tfrac{P(2) P(5)^3 P(3)}{P(1) P(4)}
		+\tfrac{7 P(5) P(3)^3 P(6)}{P(1) P(4)}
		+\tfrac{5 P(4)^2 P(3) P(6)^2}{P(2) P(5)}
		\\&\quad		
		+\tfrac{3 P(1) P(5)^2 P(6)^2}{P(2) P(3)}
		+\tfrac{9 q^{13} P(2) P(5) P(4)^3}{(P(3) P(6)}
		+\tfrac{3 q^{13} P(2) P(5) P(3) P(6)}{P(4)}
		+\tfrac{8 q^{13} P(1) P(4) P(3) P(6)}{P(2)}
		\\&\quad		
		+\tfrac{7 q^{13} P(2)^2 P(4) P(3)}{P(1)}
		+\tfrac{10 q^{13} P(1) P(5)^2 P(4)}{P(3)}
		+\tfrac{11 q^{13} P(4) P(3)^2 P(6)}{P(5)}
		+\tfrac{6 q^{13} P(2)^2 P(6)^2}{P(3)}
		+\tfrac{6 q^{13} P(4)^3 P(3)}{P(5)}
		\\&\quad		
		+\tfrac{9 q^{13} P(5)^2 P(3)^2}{P(6)}
		+\tfrac{6 q^{13} P(2)^3 P(4)^2 P(6)}{(P(1) P(5) P(3))}
		+{\scriptstyle 8 q^{13} P(2) P(5) P(4)}
		+{\scriptstyle q^{26} P(1) P(3)^2}
		+\tfrac{2 q^{26} P(2)^2 P(5) P(3)}{P(6)}
		\\&\quad		
		+\tfrac{3 q^{26} P(1) P(2) P(5)^2}{P(6)}
		+\tfrac{2 q^{39} P(1) P(2)^2 P(3)}{P(5)}
		+\tfrac{4 q^{39} P(1) P(2) P(3)^3}{(P(4) P(6)}
		+{\scriptstyle 7 q^{39} P(1)^2 P(2)}
	,\\
	A'_{13,9}(q^{13})
	&=
		\tfrac{10 P(2) P(4) P(3) P(6)^2}{q^{13} P(1)^2}
		+\tfrac{10 P(2) P(5)^3 P(3)}{q^{13} P(1)^2}
		+\tfrac{6 P(5) P(3)^3 P(6)}{q^{13} P(1)^2}
		+\tfrac{8 P(2)^3 P(5)^2}{P(1)^2}
		+\tfrac{7 P(5)^2 P(4)^2}{P(3)}
		+\tfrac{10 P(3)^2 P(6)^2}{P(2)}
		\\&\quad		
		+\tfrac{2 P(2)^2 P(4) P(6)^2}{(P(1) P(3)}
		+\tfrac{4 P(2) P(5) P(3) P(6)}{P(1)}
		+\tfrac{10 P(2) P(5) P(6)^3}{P(4) P(3)}
		+\tfrac{8 P(4)^2 P(3)^2 P(6)}{P(1) P(5)}
		+\tfrac{7 P(5)^2 P(4) P(3)^2}{P(1) P(6)}
		\\&\quad		
		+\tfrac{5 P(1) P(5) P(4)^2 P(6)}{P(2)^2}
		+\tfrac{12 P(2)^2 P(3)^2 P(6)}{P(1)^2}
		+\tfrac{10 P(2) P(5) P(4)^2}{P(1)}
		+\tfrac{8 P(5)^2 P(3) P(6)^2}{P(4)^2}
		+\tfrac{8 P(4)^2 P(3) P(6)}{P(2)}
		\\&\quad		
		+{\scriptstyle 2 P(5)^2 P(6)}
		+\tfrac{11 q^{13} P(2)^2 P(5) P(3)^2}{P(1) P(4)}
		+\tfrac{9 q^{13} P(1) P(5) P(4)^3}{P(2) P(6)}
		+\tfrac{3 q^{13} P(2)^3 P(6)}{P(1)}
		+\tfrac{3 q^{13} P(2)^3 P(4)^2}{P(1) P(3)}
		\\&\quad		
		+\tfrac{10 q^{13} P(2)^2 P(5) P(6)}{P(3)}
		+\tfrac{11 q^{13} P(2) P(4)^2 P(6)}{P(5)}
		+\tfrac{q^{13} P(2) P(5)^2 P(4)}{P(6)}
		+\tfrac{3 q^{13} P(2) P(3) P(6)^2}{P(5)}
		\\&\quad		
		+\tfrac{7 q^{13} P(2) P(5)^2 P(3)}{P(4)}
		+\tfrac{2 q^{13} P(2)^2 P(4)^2 P(3) P(6)}{P(1) P(5)^2}
		+\tfrac{9 q^{13} P(2)^2 P(5) P(4) P(3)}{P(1) P(6)}
		+{\scriptstyle q^{13} P(1) P(6)^2}
		+{\scriptstyle 8 q^{13} P(4) P(3)^2}		
		\\&\quad		
		+\tfrac{8 q^{26} P(2)^2 P(4) P(3)}{P(5)}
		+\tfrac{10 q^{26} P(1) P(5) P(3)^2}{P(6)}
		+\tfrac{8 q^{26} P(1) P(2) P(3) P(6)}{P(4)}
		+\tfrac{10 q^{26} P(2) P(3)^3}{P(6)}
		+{\scriptstyle 8 q^{26} P(1) P(2) P(4)}
		\\&\quad		
		+\tfrac{7 q^{39} P(1) P(2)^2 P(3)}{P(6)}
	,\\
	A'_{13,10}(q^{13})
	&=
		\tfrac{5 P(5)^2 P(3) P(6)^2}{q^{13} P(1) P(4)}
		+\tfrac{3 P(5)^2 P(4) P(6)}{q^{13} P(1)}
		+\tfrac{5 P(2) P(5) P(4) P(3) P(6)}{q^{13} P(1)^2}
		+\tfrac{P(2)^2 P(5) P(4) P(6)}{P(1) P(3)}
		\\&\quad		
		+\tfrac{10 P(2) P(4) P(3) P(6)^2}{P(1) P(5)}
		+\tfrac{8 P(2) P(5)^2 P(3)^2 P(6)}{P(1) P(4)^2}
		+\tfrac{9 P(1) P(5) P(4) P(6)^2}{P(2) P(3)}
		+\tfrac{8 P(2) P(5)^2 P(3)}{P(1)}
		\\&\quad		
		+\tfrac{5 P(5) P(4)^2 P(3)}{P(2)}
		+\tfrac{11 P(5) P(3)^2 P(6)}{P(2)}
		+\tfrac{11 P(4)^2 P(3)^2}{P(1)}
		+\tfrac{12 P(3)^3 P(6)}{P(1)}
		+\tfrac{10 P(2)^2 P(5) P(6)^2}{P(1) P(4)}
		\\&\quad		
		+\tfrac{2 P(2) P(5)^2 P(4)^2}{((1) P(6)}
		+\tfrac{12 P(2)^3 P(4) P(6)}{P(1)^2}
		+{\scriptstyle 4 P(5)^3}
		+{\scriptstyle 4 P(4) P(6)^2}
		+\tfrac{2 q^{13} P(2)^2 P(4)^2 P(3)}{P(1) P(5)}
		\\&\quad		
		+\tfrac{2 q^{13} P(1) P(5)^2 P(4) P(3)}{P(2) P(6)}
		+\tfrac{2 q^{13} P(1) P(5) P(4)^2}{P(3)}
		+\tfrac{8 q^{13} P(4)^2 P(3)^2 P(6)}{P(5)^2}
		+\tfrac{5 q^{13} P(5) P(4) P(3)^2}{P(6)}
		\\&\quad		
		+\tfrac{6 q^{13} P(2)^2 P(5)^2 P(3)^2}{P(1) P(4) P(6)}
		+{\scriptstyle q^{13} P(2) P(3) P(6)}
		+{\scriptstyle 7 q^{13} P(1) P(5) P(6)}
		+{\scriptstyle 11 q^{13} P(2) P(4)^2}
		+\tfrac{10 q^{26} P(1) P(2) P(5) P(4)}{P(6)}
		\\&\quad		
		+\tfrac{12 q^{26} P(1) P(2) P(5) P(3)}{P(4)}
		+\tfrac{11 q^{26} P(2)^2 P(4) P(3)}{P(6)}
		+\tfrac{5 q^{26} P(1)^2 P(4) P(3)}{P(2)}
		+\tfrac{4 q^{26} P(2)^2 P(3)^2}{P(4)}
		\\&\quad
		+\tfrac{q^{39} P(1)^2 P(2) P(3) P(6)}{P(5) P(4)}
	,\\
	A'_{13,11}(q^{13})
	&=
		\tfrac{5 P(2)^2 P(5) P(4)^2 P(6)}{q^{13} P(1)^2 P(3)}
		+\tfrac{12 P(2) P(4)^2 P(3) P(6)^2}{q^{13} P(1)^2 P(5)}
		+\tfrac{5 P(2) P(5)^2 P(3)^2 P(6)}{q^{13} P(1)^2 P(4)}
		+\tfrac{9 P(4)^2 P(6)^2}{q^{13} P(1)}
		+\tfrac{10 P(5) P(3) P(6)^2}{P(4)}
		\\&\quad		
		+\tfrac{7 P(2)^3 P(5) P(4)}{P(1)^2}
		+\tfrac{4 P(2) P(4)^3}{P(1)}
		+\tfrac{5 P(2) P(6)^3}{P(3)}
		+\tfrac{4 P(5) P(4)^3}{P(3)}
		+\tfrac{2 P(5) P(3)^3}{P(1)}
		+\tfrac{9 P(2)^2 P(5)^2 P(6)}{P(1) P(4)}
		\\&\quad
		+\tfrac{P(2) P(4) P(3) P(6)}{P(1)}		
		+\tfrac{2 P(2) P(5)^3 P(3)^2}{P(1) P(4)^2}
		+\tfrac{P(5)^2 P(4)^2 P(3)}{P(2) P(6)}
		+\tfrac{2 P(4) P(3)^2 P(6)^2}{P(2) P(5)}
		+\tfrac{9 P(1) P(5)^2 P(4) P(6)}{P(2) P(3)}
		\\&\quad		
		+{\scriptstyle 5 P(5) P(4) P(6)}		
		+\tfrac{12 q^{13} P(2)^2 P(4) P(6)}{P(3)}
		+\tfrac{4 q^{13} P(2) P(5) P(4)^2}{P(6)}
		+\tfrac{5 q^{13} P(1) P(4)^2 P(3)}{P(2)}
		+\tfrac{6 q^{13} P(2)^2 P(3)^2}{P(1)}
		\\&\quad		
		+\tfrac{6 q^{13} P(4)^2 P(3)^2}{P(5)}
		+\tfrac{7 q^{13} P(2)^3 P(4) P(6)}{P(1) P(5)}
		+{\scriptstyle 7 q^{13} P(2) P(5) P(3)}
		+{\scriptstyle 6 q^{13} P(1) P(5)^2}
		+\tfrac{9 q^{26} P(1) P(2) P(3) P(6)}{P(5)}
		\\&\quad	
		+\tfrac{6 q^{26} P(1) P(4) P(3)^2}{P(6)}
		+{\scriptstyle 5 q^{26} P(2)^3}
		+\tfrac{10 q^{39} P(1)^2 P(2) P(3)}{P(4)}
		+\tfrac{9 q^{39} P(1) P(2)^2 P(4) P(3)}{P(5) P(6)}
	,\\
	A'_{13,12}(q^{13})
	&=
		+\tfrac{2 P(4)^2 P(3)^2 P(6)^2}{q^{13} P(1) P(2) P(5)}
		+\tfrac{6 P(5) P(4)^2 P(6)}{q^{13} P(1)}
		+\tfrac{5 P(5) P(3) P(6)^2}{q^{13} P(1)}
		+\tfrac{2 P(2) P(4)^2 P(3) P(6)}{q^{13} P(1)^2}
		+\tfrac{11 P(2) P(5)^3 P(3)^2}{(q^{13} P(1)^2 P(4)}
		\\&\quad		
		+\tfrac{11 P(2) P(5) P(6)^2}{P(3)}
		+\tfrac{P(5)^2 P(4)^3}{(P(3) P(6)}
		+\tfrac{10 P(5)^2 P(3) P(6)}{P(4)}
		+\tfrac{4 P(2)^3 P(4)^2 P(6)}{P(1)^2 P(5)}
		+\tfrac{6 P(2)^2 P(4)^2 P(6)}{P(1) P(3)}
		\\&\quad		
		+\tfrac{10 P(2) P(5) P(4) P(3)}{P(1)}
		+\tfrac{9 P(4) P(3)^3 P(6)}{P(1) P(5)}
		+\tfrac{2 P(4) P(3)^2 P(6)}{P(2)}
		+\tfrac{7 P(1) P(5) P(4)^3}{P(2)^2}	
		+\tfrac{2 P(2) P(5) P(3)^2 P(6)}{P(1) P(4)}
		\\&\quad
		+\tfrac{6 P(2)^2 P(6)^2}{P(1)}
		+\tfrac{10 P(1) P(6)^3}{P(2)}
		+{\scriptstyle 2 P(5)^2 P(4)}
		+\tfrac{11 q^{13} P(2)^3 P(4)}{P(1)}
		+\tfrac{2 q^{13} P(4)^2 P(3)^2}{P(6)}
		+\tfrac{10 q^{13} P(2)^2 P(5) P(4)}{P(3)}
		\\&\quad
		+\tfrac{10 q^{13} P(2)^2 P(5) P(6)}{P(4)}
		+\tfrac{4 q^{13} P(2) P(5)^2 P(3)}{P(6)}
		+\tfrac{3 q^{13} P(1) P(5) P(4)^2 P(3)}{P(2) P(6)}
		+\tfrac{q^{13} P(1) P(2) P(5)^2 P(6)}{P(4) P(3)}
		+{\scriptstyle 4 q^{13} P(3)^3}
		\\&\quad		
		+\tfrac{6 q^{26} P(1) P(2) P(4)^2}{P(6)}
		+\tfrac{3 q^{26} P(2)^2 P(4)^2 P(3)}{P(5) P(6)}
		+{\scriptstyle q^{26} P(1) P(2) P(3)}
		+{\scriptstyle 9 q^{26} P(1)^2 P(5)}
		+\tfrac{7 q^{39} P(1) P(2)^2 P(3)^2}{P(4) P(6)}
\end{align*}

To finish the proof of (\ref{EqTheoremUMod13}) we must show that
$A'_{13}\equiv A_{13}\pmod{13}$. We do this by reducing each $A'_{13,i}$ modulo $13$.
For each of the $A'_{13,i}$ we clear denominators
and use a computer algebra system (in particular we used Maple)
to reduce 
according the following rules
(with priority of eliminating $P(6)$, $P(5)$, $P(4)$, $P(3)$, $P(2)$, and then
$P(1)$):
\begin{align*}
	P(3)^3P(1)-P(4)P(2)^3+q^{13}P(5)P(1)^3 &=0
	,\\ 
	P(4)^3P(2)-P(5)P(3)^3+q^{26}P(6)P(1)^3 &=0
	,\\ 
	P(5)^3P(1)-P(6)P(3)^3+q^{13}P(5)P(2)^3 &=0
	,\\
	P(5)^3P(3)-P(6)P(4)^3+q^{39}P(4)P(1)^3 &=0
	,\\ 
	P(6)^3P(1)-P(4)^3P(3)+q^{13}P(3)^3P(2) &=0
	,\\ 
	P(6)^3P(2)-P(5)P(4)^3+q^{26}P(3)P(2)^3 &=0
	,\\ 
	P(6)^3P(3)-P(5)^3P(4)+q^{39}P(2)^3P(1) &=0
	,\\ 
	P(6)^3P(4)-P(6)P(5)^3+q^{52}P(2)P(1)^3 &=0
	,\\ 
	P(4)^2P(3)P(1)-P(5)P(3)P(2)^2+q^{13}P(6)P(2)P(1)^2 &=0
	,\\ 
	P(4)^2P(5)P(1)-P(6)P(2)P(3)^2+q^{13}P(6)P(1)P(2)^2 &=0
	,\\ 
	P(5)^2P(3)P(1)-P(6)P(2)^2P(4)+q^{13}P(6)P(1)^2P(3) &=0
	,\\ 
	P(5)^2P(4)P(2)-P(6)P(4)P(3)^2+q^{26}P(5)P(2)P(1)^2 &=0
	,\\ 
	P(5)^2P(6)P(1)-P(5)P(2)P(4)^2+q^{13}P(4)P(1)P(3)^2 &=0
	,\\ 
	P(5)^2P(6)P(2)-P(6)P(3)P(4)^2+q^{26}P(4)P(1)P(2)^2 &=0
	,\\ 
	P(6)^2P(3)P(1)-P(5)P(2)^2P(6)+q^{13}P(4)P(1)^2P(5) &=0
	,\\ 
	P(6)^2P(4)P(1)-P(5)^2P(2)P(3)+q^{13}P(4)^2P(1)P(2) &=0
	,\\ 
	P(6)^2P(4)P(2)-P(6)P(3)^2P(5)+q^{26}P(4)P(1)^2P(3) &=0
	,\\ 
	P(6)^2P(5)P(1)-P(5)P(4)P(3)^2+q^{13}P(4)P(3)P(2)^2 &=0
	,\\ 
	P(6)^2P(5)P(2)-P(5)^2P(4)P(3)+q^{26}P(3)^2P(2)P(1) &=0
	,\\ 
	P(6)^2P(5)P(3)-P(6)P(5)P(4)^2+q^{39}P(3)P(2)P(1)^2 &=0
	,\\ 
	P(6)P(4)P(5)P(1)-P(6)P(3)P(4)P(2)+q^{13}P(5)P(2)P(3)P(1) &=0
.
\end{align*}
These equations all follow from (\ref{MolkTanneryProductId}).
Doing so gives that the $A'_{13,i}(q^{13})$ reduce to the $A_{13,i}(q^{13})$
stated in Theorem \ref{TheoremMain}. 
Given the number of identities coming from (\ref{MolkTanneryProductId}) and
that we are working modulo $13$, it may be possible to find
$A_{13,i}$ that consist of fewer terms.
This finishes the proof of (\ref{EqTheoremUMod13}).

\end{proof}

\begin{proof}[Proof of (\ref{EqTheoremVMod13}).]
We recall here $\ell=13$ and $P(a)=[q^{13a};q^{169}]_\infty$.
We expand the $\ell=13$ case of (\ref{EqVToS}) with Lemmas
\ref{MainLemmaForSSeries} and \ref{SecondLemmaForSSeries},
using $m=1$ in each application, to obtain 
\begin{align}
	\label{EqMod13VEq1}
	E(1)^3V(q)
	&\equiv	
	\tfrac{12 q^{13} E(1)^3 T(91, 13, 169)}{E(169) P(1)}
	+\tfrac{6 q^{13} E(169)^2 P(2)}{P(1) P(4)}
	+\tfrac{5 q^{13} E(169)^2 P(5)}{P(3) P(4)}
	+\tfrac{2 q^{13} E(169)^2 P(2) P(4)}{P(1) P(3) P(6)}
	+\tfrac{8 q E(169)^2 P(3) P(6)}{P(1) P(2) P(5)}
	\nonumber\\&\quad
	+\tfrac{12 q E(169)^2 P(5)}{P(1) P(3)}
	+\tfrac{10 q^{14} E(169)^2 P(4)}{P(2) P(6)}
	+\tfrac{4 q^{14} E(169)^2 P(2) P(5)}{P(1) P(4) P(6)}
	+\tfrac{12 q^2 E(169)^2 P(3)}{P(1) P(2)}
	+\tfrac{q^2 E(169)^2 P(2) P(5)}{P(1)^2 P(6)}
	\nonumber\\&\quad
	+\tfrac{10 q^{15} E(169)^2 P(2)}{P(1) P(5)}
	+\tfrac{4 q^{15} E(169)^2}{P(3)}
	+\tfrac{11 q^3 E(1)^3 T(39, 13, 169)}{E(169) P(1)}
	+\tfrac{11 q^3 E(169)^2 P(4)}{P(1) P(3)}
	+\tfrac{4 q^3 E(169)^2 P(6)}{P(1) P(4)}
	\nonumber\\&\quad
	+\tfrac{12 q^3 E(169)^2 P(3) P(5)}{P(1) P(2) P(6)}
	+\tfrac{2 q^3 E(169)^2 P(2) P(4)}{P(1)^2 P(5)}
	+\tfrac{4 q^{16} E(169)^2 P(2)}{P(1) P(6)}
	+\tfrac{4 q^4 E(169)^2 P(4) P(5)}{P(1) P(3) P(6)}
	\nonumber\\&\quad
	+\tfrac{3 q^4 E(169)^2 P(2) P(3)}{P(1)^2 P(4)}
	+\tfrac{3 q^4 E(169)^2 P(2) P(6)}{P(1) P(3)^2}
	+\tfrac{3 q^4 E(169)^2 P(4)}{P(2)^2}
	+\tfrac{8 q^4 E(169)^2 P(5)}{P(1) P(4)}
	+\tfrac{6 q^{17} E(169)^2 P(6)}{P(4) P(5)}
	\nonumber\\&\quad
	+\tfrac{10 E(169)^2 P(5)}{q^8 P(1)^2}
	+\tfrac{3 E(1)^3 T(13, 13, 169)}{q^8 E(169) P(1)}
	+\tfrac{4 q^5 E(169)^2 P(2)^2}{P(1)^2 P(3)}
	+\tfrac{8 q^5 E(169)^2 P(3) P(4)}{P(1) P(2) P(6)}
	+\tfrac{q^{18} E(169)^2 P(2) P(4)}{P(1) P(5) P(6)}
	\nonumber\\&\quad
	+\tfrac{q^{18} E(169)^2}{P(4)}
	+\tfrac{9 E(169)^2 P(4) P(6)}{q^7 P(1)^2 P(5)}
	+\tfrac{7 q^6 E(169)^2 P(5)}{P(2) P(3)}
	+\tfrac{8 q^6 E(169)^2}{P(1)}
	+\tfrac{3 q^6 E(169)^2 P(4)^2}{P(1) P(3) P(6)}
	\nonumber\\&\quad
	+\tfrac{3 q^{19} E(169)^2 P(2) P(3)}{P(1) P(4) P(6)}
	+\tfrac{12 q^7 E(1)^3 T(52, 13, 169)}{E(169) P(1)}
	+\tfrac{4 q^7 E(169)^2 P(3)^2}{P(1) P(2) P(5)}
	+\tfrac{7 q^7 E(169)^2 P(3) P(5)}{P(1) P(4)^2}
	\nonumber\\&\quad
	+\tfrac{7 q^7 E(169)^2 P(2) P(6)}{P(1) P(3) P(4)}
	+\tfrac{5 q^{20} E(169)^2}{P(5)}
	+\tfrac{11 E(169)^2 P(3) P(5)}{q^5 P(1)^2 P(4)}
	+\tfrac{12 q^8 E(169)^2 P(4)}{P(1) P(5)}
	+\tfrac{11 q^8 E(169)^2 P(6)}{P(2) P(4)}
	\nonumber\\&\quad
	+\tfrac{4 q^8 E(169)^2 P(4)}{P(2) P(3)}
	+\tfrac{11 q^{21} E(169)^2}{P(6)}
	+\tfrac{12 q^9 E(169)^2 P(4)}{P(1) P(6)}
	+\tfrac{7 q^9 E(169)^2 P(2) P(6)}{P(1) P(3) P(5)}
	+\tfrac{8 q^{22} E(169)^2 P(5)}{P(6)^2}
	\nonumber\\&\quad
	+\tfrac{11 q^{10} E(169)^2 P(2)}{P(1) P(3)}
	+\tfrac{2 q^{10} E(169)^2 P(6)}{P(2) P(5)}
	+\tfrac{8 q^{10} E(169)^2 P(3) P(5)}{P(1) P(4) P(6)}
	+\tfrac{2 q^{23} E(169)^2 P(4)}{P(5) P(6)}
	+\tfrac{9 E(1)^3 T(26, 13, 169)}{q^2 E(169) P(1)}
	\nonumber\\&\quad
	+\tfrac{7 E(169)^2 P(5)^2}{q^2 P(1) P(2) P(6)}
	+\tfrac{8 E(169)^2 P(2) P(4)}{q^2 P(1)^2 P(3)}
	+\tfrac{10 E(169)^2 P(4) P(6)}{q^2 P(1) P(2) P(5)}
	+\tfrac{10 q^{11} E(169)^2 P(3)}{P(1) P(5)}
	+\tfrac{9 q^{11} E(169)^2}{P(2)}
	\nonumber\\&\quad
	+\tfrac{E(169)^2 P(3) P(6)}{q P(1) P(2) P(4)}
	+\tfrac{5 q^{12} E(1)^3 T(78, 13, 169)}{E(169) P(1)}
	+\tfrac{6 q^{12} E(169)^2 P(5)}{P(2) P(6)}
	+\tfrac{8 q^{12} E(169)^2 P(2) P(4)}{P(1) P(3) P(5)}
.
\end{align}
We multiply (\ref{EqMod13VEq1}) by (\ref{EqMod13UEq2}),
collect terms, and reduce modulo $13$, to find that
\begin{align*}
	V(q)
	&\equiv
		\tfrac{12 q^{13} T(91, 13, 169)}{E(169) P(1)}
		+
		\tfrac{11 q^3 T(39, 13, 169)}{E(169) P(1)}
		+
		\tfrac{3  T(13, 13, 169)}{q^8 E(169) P(1)}
		+
		\tfrac{12 q^7 T(52, 13, 169)}{E(169) P(1)}
		+
		\tfrac{9  T(26, 13, 169)}{q^2 E(169) P(1)}
		+
		\tfrac{5 q^{12}  T(78, 13, 169)}{E(169) P(1)}
		\\&\quad		
		+
		B'_{13}(q)
	\pmod{13}
,
\end{align*}
where
\begin{align*}
	B'_{13}(q) &= 
		\frac{E(169)^4}{E(13)}
		\bigg(
		B'_{13,0}(q^{13})+qB'_{13,1}(q^{13})+q^2B'_{13,2}(q^{13})+q^3B'_{13,3}(q^{13})
		+q^4B'_{13,4}(q^{13})
		\\&\quad		
		+q^5B'_{13,5}(q^{13})
		+q^6B'_{13,6}(q^{13})
		+q^7B'_{13,7}(q^{13})+q^8B'_{13,8}(q^{13})+q^9B'_{13,9}(q^{13})
		+q^{10}B'_{13,10}(q^{13})
		\\&\quad		
		+q^{11}B'_{13,11}(q^{13})
		+q^{12}B'_{13,12}(q^{13})
		\bigg)
	,\\
	B'_{13,0}(q^{13})
	&=
		\tfrac{11 P(5)^2 P(6) P(3) }{P(1)}
		+\tfrac{3 P(6)^2 P(3)^3 }{P(1) P(2)}
		+\tfrac{P(5)^2 P(4)^2 }{P(1)}
		+\tfrac{12 P(6)^3 P(4) }{P(2)}
		+\tfrac{11 P(5)^3 P(6) }{P(2)}
		+\tfrac{7 P(5) P(6)^2 P(2) P(4) }{P(1) P(3)}
		\\&\quad	
		+\tfrac{7 P(5) P(6) P(2) P(3)^2 }{P(1)^2}		
		+\tfrac{4q^{13} P(6) P(2)^3 P(3)  }{P(1)^2}
		+\tfrac{9q^{13} P(5) P(6) P(2)^2  }{P(1)}
		+\tfrac{2q^{13} P(5)^2 P(6) P(2)  }{P(3)}
		\\&\quad
		+\tfrac{7q^{13} P(5) P(3)^2 P(4)  }{P(2)}
		+\tfrac{q^{13} P(5) P(6)^2 P(1)  }{P(2)}
		+\tfrac{4q^{13} P(6)^2 P(2)^3 P(4)  }{P(1) P(3)^2}
		+\tfrac{6q^{13} P(5) P(2)^2 P(4)^2  }{P(1) P(3)}
		\\&\quad
		+\tfrac{7q^{13} P(5)^2 P(2) P(3)^2  }{P(1) P(4)}
		+\tfrac{6q^{13} P(5) P(3) P(4)^3  }{P(6) P(2)}
		+\tfrac{11q^{13} P(5)^2 P(2) P(3) P(4)  }{P(6) P(1)}
		+\tfrac{4q^{13} P(5) P(6) P(1) P(4)^2  }{P(2) P(3)}
		\\&\quad
		+\tfrac{8q^{26} P(6)^2 P(2)^2  }{P(5)}
		+\tfrac{12q^{26} P(5)^2 P(2)^2  }{P(4)}
		+\tfrac{5q^{26} P(2) P(4)^3  }{P(6)}
		+\tfrac{5q^{26} P(5)^3 P(2) P(3)  }{P(6)^2}
		+{\scriptstyle 6q^{26} P(2) P(3) P(4) }
		\\&\quad
		+{\scriptstyle 5q^{26} P(5) P(1) P(4) }
		+\tfrac{q^{39} P(2)^2 P(3)^2  }{P(6)}
		+\tfrac{4q^{39} P(1) P(3)^3  }{P(5)}
		+\tfrac{7q^{39} P(6) P(1) P(2)^2  }{P(4)}
		+\tfrac{7q^{39} P(5) P(1) P(2) P(3)^2 }{P(4)^2}
		\\&\quad
		+\tfrac{5 P(1)^2 P(2) P(3) q^52 }{P(5)}
	,\\
	B'_{13,1}(q^{13})
	&=
		\tfrac{12 P(5)^2 P(6) P(2) P(4) }{P(1) P(3)}
		+\tfrac{10 P(6) P(4)^3 }{P(1)}
		+\tfrac{11 P(5)^3 P(3) }{P(1)}
		+\tfrac{3 P(5) P(6) P(2)^2 P(4) }{P(1)^2}
		+\tfrac{9 P(5) P(6)^3 P(3) }{P(2) P(4)}
		\\&\quad		
		+\tfrac{7 P(6)^2 P(3) P(4) }{P(1)}
		+\tfrac{q^{13} P(6) P(2)^4 P(4) }{P(1)^2 P(3)}
		+\tfrac{8 q^{13} P(6)^2 P(2)^2 P(4) }{P(5) P(1)}
		+\tfrac{12 q^{13} P(2) P(3) P(4)^2 }{P(1)}
		+\tfrac{4 q^{13} P(6)^2 P(2) P(4) }{P(3)}
		\\&\quad
		+\tfrac{9 q^{13} P(5)^2 P(2)^2 }{P(1)}
		+\tfrac{7 q^{13} P(6) P(2) P(3)^2 }{P(1)}
		+\tfrac{12 q^{13} P(5) P(4)^4 }{P(6) P(3)}
		+\tfrac{2 q^{13} P(5)^2 P(6) P(1) }{P(2)}
		+\tfrac{5 q^{13} P(5)^3 P(2) P(3)^2 }{P(6) P(1) P(4)}
		\\&\quad
		+\tfrac{2 q^{13} P(5)^2 P(1) P(4)^2 }{P(2) P(3)}
		+\tfrac{3 q^{13} P(5)^2 P(3)^2 P(4) }{P(6) P(2)}
		+\tfrac{3 q^{13} P(1) P(6) P(3) P(4)^2 }{P(2)^2}
		+{\scriptstyle 12 q^{13} P(5) P(4)^2}
		\\&\quad
		+{\scriptstyle 9 q^{13} P(5) P(6) P(3)}
		+\tfrac{9 q^{26} P(5) P(2) P(3) P(4) }{P(6)}
		+\tfrac{12 q^{26} P(1) P(5) P(6) P(2) }{P(3)}
		+{\scriptstyle 12 q^{26} P(6) P(2)^2}
		\\&\quad
		+\tfrac{8 q^{26} P(5)^2 P(1) P(4) }{P(6)}
		+\tfrac{6 q^{26} P(6)^2 P(1) P(3) }{P(5)}
		+\tfrac{10 q^{26} P(3)^3 P(4) }{P(5)}
		+\tfrac{10 q^{26} P(2)^3 P(4)^2 }{P(5) P(1)}
		+\tfrac{9 q^{26} P(2) P(5) P(3)^2 }{P(4)}
		\\&\quad
		+\tfrac{5 q^{39} P(1) P(2) P(3) P(4) }{P(5)}
		+\tfrac{11 q^{39} P(1)^2 P(6) P(3) }{P(4)}		
		+{\scriptstyle 4 q^{39} P(1)^2 P(4)}
		+\tfrac{11 q^{52} P(1)^2 P(2) P(3) }{P(6)}
	,\\
	B'_{13,2}(q^{13})
	&=
		\tfrac{6 P(5)^2 P(2)^2 P(4) }{P(1)^2}
		+\tfrac{6 P(5)^4 P(3) }{P(6) P(1)}
		+\tfrac{11 P(6)^2 P(3)^3 P(4) }{P(5) P(1) P(2)}
		+\tfrac{9 P(5) P(6)^2 P(3)^2 }{P(1) P(4)}
		+\tfrac{12 P(6)^2 P(2)^2 P(4)^2 }{P(5) P(1)^2}
		\\&\quad		
		+\tfrac{9 P(5) P(6) P(3) P(4) }{P(1)}
		+\tfrac{12 q^{13} P(5) P(6) P(2) P(4) }{P(3)}
		+\tfrac{6 q^{13} P(5)^2 P(6)^2 P(1) }{P(3) P(4)}
		+\tfrac{4 q^{13} P(1) P(5) P(3) P(4)^2 }{P(2)^2}
		\\&\quad
		+\tfrac{6 q^{13} P(6)^2 P(3) P(4) }{P(5)}
		+\tfrac{2 q^{13} P(6)^2 P(1) P(4) }{P(2)}
		+\tfrac{4 q^{13} P(2)^2 P(4)^3 }{P(1) P(3)}
		+\tfrac{10 q^{13} P(6) P(2)^2 P(4) }{P(1)}
		+\tfrac{12 q^{13} P(5) P(2) P(3)^2 }{P(1)}
		\\&\quad
		+\tfrac{q^{13} P(5)^2 P(4)^2 }{P(6)}
		+\tfrac{2 q^{13} P(3)^2 P(4)^2 }{P(2)}
		+\tfrac{q^{13} P(5) P(6)^2 P(2) }{P(4)}
		+{\scriptstyle 2 q^{13} P(5)^2 P(3)}
		+\tfrac{4 q^{26} P(2)^3 P(3) }{P(1)}
		\\&\quad
		+\tfrac{11 q^{26} P(2) P(3) P(4)^2 }{P(5)}
		+{\scriptstyle 12 q^{26} P(5) P(2)^2}
		+{\scriptstyle 6 q^{26} P(1) P(4)^2}
		+{\scriptstyle 4 q^{26} P(1) P(3) P(6)}
		+\tfrac{7 q^{39} P(1) P(2) P(3) P(4) }{P(6)}
		\\&\quad
		+\tfrac{7 q^{39} P(6) P(1) P(2)^2 }{P(5)}
		+\tfrac{8 q^{52} P(5) P(1)^2 P(2) P(3) }{P(6)^2}
	,\\
	B'_{13,3}(q^{13})
	&=
		\tfrac{4 P(5)^2 P(6)^2 }{P(3)}
		+\tfrac{7 P(5)^3 P(4) }{P(2)}
		+\tfrac{7 P(6)^3 P(3) }{P(2)}
		+\tfrac{10 P(6)^2 P(4)^2 }{P(2)}
		+\tfrac{4 P(5) P(6)^2 P(2) }{P(1)}
		+\tfrac{8 P(5)^2 P(3) P(4) }{P(1)}
		\\&\quad		
		+\tfrac{10 P(6) P(3)^3 P(4) }{P(1) P(2)}
		+\tfrac{6 P(5) P(6) P(2) P(4)^2 }{P(1) P(3)}
		+\tfrac{3 P(5) P(2) P(3)^2 P(4) }{P(1)^2}
		+\tfrac{12 P(5)^2 P(6) P(3)^2 }{P(1) P(4)}
		\\&\quad
		+\tfrac{9 q^{13} P(5) P(6)^2 P(2)^2 }{P(3)^2}
		+\tfrac{12 q^{13} P(5)^2 P(2) P(4) }{P(3)}		
		+\tfrac{7 q^{13} P(5)^3 P(3) }{P(6)}
		+\tfrac{3 q^{13} P(5) P(6) P(2)^2 P(3) }{P(1) P(4)}
		\\&\quad
		+\tfrac{9 q^{13} P(6) P(2) P(3)^2 P(4) }{P(5) P(1)}
		+\tfrac{12 q^{13} P(4)^3+q^13 P(5)^2 P(6) P(2) }{P(4)}
		+\tfrac{5 q^{13} P(6)^2 P(2)^3 }{P(1) P(3)}
		+\tfrac{2 q^{13} P(5) P(2)^2 P(4) }{P(1)}
		\\&\quad
		+\tfrac{3 q^{13} P(1) P(5) P(4)^3 }{P(2) P(3)}
		+\tfrac{2 q^{13} P(5) P(3)^2 P(4)^2 }{P(6) P(2)}
		+\tfrac{3 q^{13} P(1) P(5) P(6) P(4) }{P(2)}
		+{\scriptstyle 5 q^{13} P(6) P(3) P(4)}
		\\&\quad
		+\tfrac{q^{26} P(2) P(3) P(4)^2 }{P(6)}
		+\tfrac{5 q^{26} P(6)^2 P(1) P(2) }{P(4)}
		+\tfrac{5 q^{26} P(5) P(1) P(4)^2 }{P(6)}
		+\tfrac{9 q^{26} P(6) P(2)^2 P(4) }{P(5)}
		+{\scriptstyle 10 q^{26} P(1) P(3) P(5)}
		\\&\quad
		+{\scriptstyle 3 q^{26} P(2) P(3)^2}
		+\tfrac{2 q^{39} P(6) P(1)^2 P(3) }{P(5)}
		+\tfrac{7 q^{39} P(5) P(1) P(2) P(3) P(4) }{P(6)^2}
		+\tfrac{8 q^{39} P(5) P(1) P(2) P(3)^2 }{P(6) P(4)}
		\\&\quad
		+{\scriptstyle 11 q^{39} P(1) P(2)^2}
		+\tfrac{2 q^{52} P(1)^2 P(2) P(3) P(4) }{P(5) P(6)}
	,\\
	B'_{13,4}(q^{13})
	&=
		\tfrac{3 P(5) P(6)^2 P(2)^2 P(4) }{P(1) P(3)^2}
		+\tfrac{6 P(6) P(2) P(3)^2 P(4)^2 }{P(5) P(1)^2}
		+\tfrac{9 P(5) P(6) P(2)^2 P(3) }{P(1)^2}
		+\tfrac{4 P(5)^2 P(2) P(4)^2 }{P(1) P(3)}
		+\tfrac{9 P(6)^2 P(3)^2 }{P(1)}
		\\&\quad		
		+\tfrac{10 P(5) P(3)^3 P(4) }{P(1) P(2)}
		+\tfrac{12 P(5)^2 P(6) P(2) }{P(1)}
		+\tfrac{6 P(6) P(3) P(4)^2 }{P(1)}
		+\tfrac{3 P(5) P(6) P(4)^2 }{P(2)}
		+\tfrac{5 P(5) P(6)^2 P(3) }{P(2)}
		\\&\quad
		+\tfrac{12 q^{13} P(5) P(6) P(2)^3 }{P(1) P(3)}
		+\tfrac{q^{13} P(6) P(2)^2 P(4)^2 }{P(5) P(1)}
		+\tfrac{11 q^{13} P(5)^2 P(2)^2 P(4) }{P(6) P(1)}
		+\tfrac{4 q^{13} P(5)^2 P(2)^2 P(3) }{P(1) P(4)}
		\\&\quad
		+\tfrac{4 q^{13} P(6) P(3)^3 P(4) }{P(5) P(2)}
		+\tfrac{11 q^{13} P(2) P(6) P(4)^2 }{P(3)}
		+\tfrac{8 q^{13} P(5)^3 P(2) }{P(4)}
		+\tfrac{7 q^{13} P(5) P(6) P(3)^2 }{P(4)}
		+\tfrac{11 q^{13} P(5)^2 P(1) P(4) }{P(2)}
		\\&\quad
		+\tfrac{12 q^{13} P(2) P(3)^2 P(4) }{P(1)}
		+\tfrac{10 q^{13} P(5) P(6)^2 P(1) }{P(3)}
		+{\scriptstyle 4 q^{13} P(5) P(3) P(4)}
		+{\scriptstyle 5 q^{13} P(6)^2 P(2)}
		+\tfrac{3 q^{26} P(1) P(5) P(6) P(2) }{P(4)}
		\\&\quad
		+\tfrac{7 q^{26} P(6) P(1) P(3) P(4) }{P(5)}
		+\tfrac{7 q^{26} P(5)^2 P(1) P(3) }{P(6)}
		+\tfrac{11 q^{26} P(5) P(2) P(3)^2 }{P(6)}
		+\tfrac{6 q^{26} P(5)^2 P(1) P(4)^2 }{P(6)^2}
		\\&\quad
		+{\scriptstyle 12 q^{26} P(2)^2 P(4)}
		+\tfrac{10 q^{39} P(1) P(2) P(3)^2 }{P(5)}
		+\tfrac{5 q^{39} P(1) P(2) P(3) P(4)^2 }{P(5) P(6)}	
		+{\scriptstyle 9 q^{39} P(1)^2 P(3)}
	,\\
	B'_{13,5}(q^{13})
	&=
		\tfrac{10 P(5)^2 P(6) P(2) P(4) }{q^{13} P(1)^2}		
		+\tfrac{10 P(5) P(6)^3 }{P(4)}
		+\tfrac{P(5)^3 P(2) }{P(1)}
		+\tfrac{8 P(5) P(6)^2 P(4) }{P(3)}
		+\tfrac{5 P(5) P(6) P(3)^2 }{P(1)}
		\\&\quad		
		+\tfrac{4 P(5)^2 P(6) P(3) }{P(2)}
		+\tfrac{9 P(6) P(3)^2 P(4)^2 }{P(2)^2}
		+\tfrac{9 P(5)^2 P(2)^2 P(3) }{P(1)^2}
		+\tfrac{2 P(6)^2 P(2) P(4) }{P(1)}
		+\tfrac{9 P(6) P(2) P(3)^3 }{P(1)^2}
		\\&\quad
		+\tfrac{7 P(5) P(3) P(4)^2 }{P(1)}
		+\tfrac{4 P(5) P(6) P(2)^3 P(4) }{P(1)^2 P(3)}
		+\tfrac{5 q^{13} P(6)^2 P(2)^3 P(4) }{P(5) P(1) P(3)}
		+\tfrac{9 q^{13} P(5)^3 P(2)^2 P(3) }{P(6) P(1) P(4)}
		+\tfrac{q^{13} P(3)^3 P(4) }{P(2)}
		\\&\quad
		+\tfrac{11 q^{13} P(6)^2 P(1) P(3) }{P(2)}
		+\tfrac{11 q^{13} P(6)^2 P(1) P(3) }{P(2)}
		+\tfrac{4 q^{13} P(2)^2 P(4)^2 }{P(1)}
		+\tfrac{3 q^{13} P(5) P(2) P(4)^2 }{P(3)}
		\\&\quad
		+\tfrac{11 q^{13} P(6) P(3) P(4)^2 }{P(5)}
		+\tfrac{7 q^{13} P(5)^2 P(3) P(4) }{P(6)}
		+\tfrac{12 q^{13} P(6) P(1) P(4)^2 }{P(2)}
		+\tfrac{5 q^{13} P(6)^2 P(3)^2 }{P(5)}
		+\tfrac{5 q^{13} P(5)^2 P(3)^2 }{P(4)}
		\\&\quad
		+\tfrac{4 q^{13} P(6) P(2)^2 P(3) }{P(1)}
		+{\scriptstyle 6 q^{13} P(5) P(6) P(2)}
		+\tfrac{8 q^{26} P(1) P(4)^3 }{P(6)}
		+\tfrac{2 q^{26} P(5) P(2)^2 P(4) }{P(6)}
		+\tfrac{9 q^{26} P(5) P(2)^2 P(3) }{P(4)}
		\\&\quad
		+\tfrac{12 q^{26} P(2) P(3)^2 P(4) }{P(5)}
		+\tfrac{q^{26} P(6) P(1) P(3)^2 }{P(4)}
		+{\scriptstyle 12 q^{26} P(1) P(3) P(4)}
		+\tfrac{8 q^{39} P(1) P(2)^2 P(4) }{P(5)}
		\\&\quad
		+\tfrac{6 q^{39} P(5) P(1)^2 P(3) }{P(6)}
	,\\
	B'_{13,6}(q^{13})
	&=
		\tfrac{4 P(5) P(6) P(3)^2 P(4) }{q^{13} P(1)^2}
		+\tfrac{9 P(6)^2 P(2) P(4)^2 }{q^{13} P(1)^2}
		+\tfrac{11 P(5) P(6)^2 P(2) P(3) }{P(1) P(4)}
		+\tfrac{11 P(5)^2 P(3)^2 }{P(1)}
		+\tfrac{P(6) P(4)^3 }{P(2)}
		\\&\quad		
		+\tfrac{7 P(5)^2 P(6)^2 }{P(4)}
		+\tfrac{P(5) P(6)^3 P(2) }{P(3)^2}
		+\tfrac{3 P(5)^2 P(6) P(4) }{P(3)}
		+\tfrac{2 P(5)^3 P(4)^2 }{P(6) P(2)}
		+\tfrac{11 P(3)^3 P(4)^2 }{P(1) P(2)}
		+\tfrac{8 P(6)^2 P(3) P(4) }{P(2)}
		\\&\quad
		+\tfrac{P(5) P(6)^2 P(1) P(4) }{P(2)^2}
		+\tfrac{4 P(6) P(2)^2 P(3) P(4) }{P(1)^2}
		+\tfrac{12 P(5)^3 P(2)^2 P(3) }{P(6) P(1)^2}
		+\tfrac{9 P(5) P(2) P(4)^3 }{P(1) P(3)}
		+\tfrac{11 P(6)^2 P(3)^2 P(4) }{P(5) P(1)}	
		\\&\quad
		+\tfrac{6 q^{13} P(6) P(2)^3 P(4) }{P(1) P(3)}
		+\tfrac{8 q^{13} P(5) P(6)^2 P(2)^2 }{P(3) P(4)}
		+\tfrac{3 q^{13} P(2) P(3)^2 P(4)^2 }{P(5) P(1)}
		+\tfrac{8 q^{13} P(5)^2 P(6) P(2) P(3) }{P(4)^2}
		\\&\quad
		+\tfrac{6 q^{13} P(5) P(6) P(1) P(3) }{P(2)}
		+\tfrac{2 q^{13} P(6)^3 P(1) }{P(4)}
		+\tfrac{q^{13} P(6)^2 P(2) P(4) }{P(5)}
		+\tfrac{8 q^{13} P(5) P(1) P(4)^2 }{P(2)}
		+\tfrac{4 q^{13} P(5) P(2)^2 P(3) }{P(1)}
		\\&\quad
		+{\scriptstyle 11 q^{13} P(6) P(3)^2}
		+{\scriptstyle 7 q^{13} P(3) P(4)^2}
		+\tfrac{3 q^{26} P(5) P(1) P(3) P(4) }{P(6)}
		+\tfrac{q^{26} P(2)^2 P(4)^2 }{P(5)}
		+{\scriptstyle 2 q^{26} P(6) P(1) P(2)}
		\\&\quad
		+\tfrac{5 q^{39} P(5) P(1)^2 P(2) }{P(4)}
		+\tfrac{2 q^{39} P(1) P(2)^2 P(4) }{P(6)}
		+\tfrac{6 q^{39} P(1) P(2)^2 P(3) }{P(4)}
	,\\
	B'_{13,7}(q^{13})
	&=
		\tfrac{12 P(5)^2 P(6)^2 }{q^{13} P(1)}
		+\tfrac{P(6)^2 P(3)^2 P(4)^2 }{q^{13} P(5) P(1)^2}
		+\tfrac{9 P(3) P(4)^3 }{P(1)}
		+\tfrac{5 P(5)^2 P(2) P(4) }{P(1)}
		+\tfrac{2 P(6) P(2)^3 P(4)^2 }{P(1)^2 P(3)}
		\\&\quad		
		+\tfrac{P(5) P(2)^2 P(3) P(4) }{P(1)^2}
		+\tfrac{12 P(6) P(3)^2 P(4) }{P(1)}
		+\tfrac{P(5)^3 P(3)^2 }{P(6) P(1)}
		+\tfrac{5 P(5) P(6)^2 P(2)^2 }{P(1) P(3)}
		+\tfrac{9 P(6)^2 P(2) P(4)^2 }{P(5) P(1)}
		\\&\quad
		+\tfrac{6 P(5) P(6) P(3) P(4) }{P(2)}
		+\tfrac{P(5) P(6) P(2)^2 P(3)^2 }{P(1)^2 P(4)}
		+\tfrac{4 P(5)^2 P(6) P(2) P(3) }{P(1) P(4)}
		+\tfrac{6 q^{13} P(2) P(4)^3 }{P(3)}
		+\tfrac{9 q^{13} P(5)^2 P(2)^2 P(3) }{P(6) P(1)}
		\\&\quad
		+\tfrac{3 q^{13} P(5) P(6) P(1) P(4) }{P(3)}
		+\tfrac{11 q^{13} P(5)^2 P(1) P(4)^2 }{P(6) P(2)}
		+\tfrac{9 q^{13} P(2) P(3)^3 }{P(1)}
		+\tfrac{2 q^{13} P(6)^2 P(2) P(3) }{P(4)}
		\\&\quad
		+\tfrac{3 q^{13} P(5) P(3) P(4)^2 }{P(6)}
		+\tfrac{3 q^{13} P(5) P(6)^2 P(1) }{P(4)}
		+\tfrac{2 q^{13} P(6)^2 P(1) P(3) P(4) }{P(5) P(2)}
		+{\scriptstyle 8 q^{13} P(5) P(3)^2}
		+{\scriptstyle 9 q^{13} P(6) P(2) P(4)}
		\\&\quad
		+\tfrac{2 q^{26} P(5)^2 P(1) P(3) P(4) }{P(6)^2}
		+\tfrac{2 q^{26} P(1) P(3) P(4)^2 }{P(5)}
		+\tfrac{8 q^{26} P(6) P(1) P(3)^2 }{P(5)}
		+\tfrac{5 q^{26} P(2)^2 P(4)^2 }{P(6)}
		+{\scriptstyle 2 q^{26} P(2)^2 P(3)}
		\\&\quad
		+{\scriptstyle 12 q^{26} P(5) P(1) P(2)}		
		+\tfrac{10 q^{39} P(1)^2 P(3) P(4) }{P(6)}
		+\tfrac{4 q^{39} P(5) P(1) P(2)^2 P(3) }{P(6) P(4)}
	,\\
	B'_{13,8}(q^{13})
	&=
		\tfrac{3 P(6)^3 P(4) }{q^{13} P(1)}
		+\tfrac{10 P(5)^2 P(6) P(2) P(3) }{q^{13} P(1)^2}
		+\tfrac{6 P(6) P(2)^2 P(3) P(4)^2 }{P(5) P(1)^2}
		+\tfrac{12 P(6) P(3)^4 P(4) }{P(5) P(1) P(2)}
		+\tfrac{10 P(5)^2 P(6) P(2)^2 }{P(1) P(3)}
		\\&\quad		
		+\tfrac{10 P(5)^2 P(2)^2 P(3)^2 }{P(1)^2 P(4)}
		+\tfrac{9 P(5)^3 P(2) P(4) }{P(6) P(1)}
		+\tfrac{3 P(5)^2 P(3) P(4) }{P(2)}
		+\tfrac{10 P(5) P(6) P(2)^3 }{P(1)^2}
		+\tfrac{3 P(6) P(2) P(4)^2 }{P(1)}
		\\&\quad
		+\tfrac{4 P(6)^2 P(2) P(3) }{P(1)}
		+\tfrac{3 P(5) P(6) P(4)^2 }{P(3)}
		+\tfrac{4 P(5) P(3)^2 P(4) }{P(1)}
		+\tfrac{5 P(5)^3 P(2) P(3) }{P(1) P(4)}
		+\tfrac{8 P(5) P(6) P(3)^3 }{P(1) P(4)}
		\\&\quad
		+\tfrac{10 P(6)^2 P(3) P(4)^2 }{P(5) P(2)}
		+\tfrac{11 P(5)^2 P(6)^2 P(1) }{P(2) P(3)}		
		+{\scriptstyle 9 P(5) P(6)^2}
		+\tfrac{8 q^{13} P(5) P(2) P(4)^3 }{P(6) P(3)}
		+\tfrac{4 q^{13} P(5) P(6) P(2) P(3) }{P(4)}
		\\&\quad
		+\tfrac{3 q^{13} P(6) P(1) P(3) P(4) }{P(2)}
		+\tfrac{2 q^{13} P(6) P(2)^3 P(4)^2 }{P(5) P(1) P(3)}
		+\tfrac{8 q^{13} P(6)^2 P(2)^2 }{P(3)}
		+\tfrac{2 q^{13} P(5)^2 P(3)^2 }{P(6)}
		+\tfrac{6 q^{13} P(6) P(3)^2 P(4) }{P(5)}
		\\&\quad
		+\tfrac{8 q^{13} P(2)^2 P(3) P(4) }{P(1)}
		+\tfrac{7 q^{13} P(5)^2 P(1) P(4) }{P(3)}
		+{\scriptstyle 2 q^{13} P(5) P(2) P(4)}
		+\tfrac{11 q^{26} P(5)^2 P(1) P(2) }{P(6)}
		+\tfrac{6 q^{26} P(1) P(3) P(4)^2 }{P(6)}
		\\&\quad
		+\tfrac{11 q^{26} P(5) P(2)^2 P(3) }{P(6)}
		+{\scriptstyle 12 q^{26} P(1) P(3)^2}
		+\tfrac{10 q^{39} P(1) P(2)^2 P(3) }{P(5)}
		+{\scriptstyle 4 q^{39} P(1)^2 P(2)}
	,\\
	B'_{13,9}(q^{13})
	&=
		\tfrac{3 P(6)^2 P(2) P(3) P(4) }{q^{13} P(1)^2}
		+\tfrac{3 P(5)^3 P(2) P(3) }{q^{13} P(1)^2}
		+\tfrac{7 P(5) P(6) P(3)^3 }{q^{13} P(1)^2}
		+\tfrac{11 P(5) P(6)^3 P(2) }{P(3) P(4)}
		+\tfrac{10 P(6) P(3)^2 P(4)^2 }{P(5) P(1)}
		\\&\quad		
		+\tfrac{3 P(5) P(6)^2 P(2) P(4) }{P(3)^2}
		+\tfrac{8 P(5) P(6) P(2) P(3) }{P(1)}
		+\tfrac{9 P(5)^2 P(2)^3 }{P(1)^2}
		+\tfrac{5 P(6)^2 P(3)^2 }{P(2)}
		+\tfrac{11 P(5)^2 P(6)^2 P(3) }{P(4)^2}
		\\&\quad
		+\tfrac{10 P(6) P(3) P(4)^2 }{P(2)}
		+\tfrac{9 P(6) P(2)^2 P(3)^2 }{P(1)^2}
		+\tfrac{P(5) P(2) P(4)^2 }{P(1)}
		+\tfrac{5 P(5)^3 P(3) P(4) }{P(6) P(2)}
		+\tfrac{3 P(5) P(6) P(1) P(4)^2 }{P(2)^2}
		\\&\quad
		+\tfrac{5 P(5)^2 P(3)^2 P(4) }{P(6) P(1)}
		+\tfrac{3 P(6)^2 P(2)^2 P(4) }{P(1) P(3)}		
		+{\scriptstyle 6 P(5)^2 P(6)}
		+\tfrac{q^{13} P(5) P(2)^2 P(3)^2 }{P(1) P(4)}
		+\tfrac{q^{13} P(5) P(1) P(4)^3 }{P(6) P(2)}
		\\&\quad
		+\tfrac{7 q^{13} P(2)^3 P(4)^2 }{P(1) P(3)}
		+\tfrac{5 q^{13} P(5) P(6) P(2)^2 }{P(3)}
		+\tfrac{8 q^{13} P(6) P(2) P(4)^2 }{P(5)}
		+\tfrac{11 q^{13} P(5)^2 P(2) P(4) }{P(6)}
		+\tfrac{5 q^{13} P(6)^2 P(2) P(3) }{P(5)}
		\\&\quad
		+\tfrac{10 q^{13} P(5)^2 P(2) P(3) }{P(4)}
		+\tfrac{8 q^{13} P(5) P(2)^2 P(3) P(4) }{P(6) P(1)}
		+\tfrac{8 q^{13} P(6) P(2)^3 }{P(1)}
		+\tfrac{5 q^{13} P(5) P(1) P(3) P(4) }{P(2)}
		+{\scriptstyle 7 q^{13} P(3)^2 P(4)}
		\\&\quad
		+{\scriptstyle 5 q^{13} P(6)^2 P(1)}
		+\tfrac{q^{26} P(2)^2 P(3) P(4) }{P(5)}
		+\tfrac{12 q^{26} P(5) P(1) P(3)^2 }{P(6)}
		+\tfrac{4 q^{26} P(6) P(1) P(2) P(3) }{P(4)}
		+{\scriptstyle q^{26} P(1) P(2) P(4)}
		\\&\quad
		+\tfrac{4 q^{39} P(1) P(2)^2 P(3) }{P(6)}
	,\\
	B'_{13,10}(q^{13})
	&=
		\tfrac{10 P(5)^2 P(6) P(4) }{q^{13} P(1)}
		+\tfrac{8 P(5) P(6) P(2) P(3) P(4) }{q^{13} P(1)^2}
		+\tfrac{8 P(5)^2 P(6)^2 P(3) }{q^{13} P(1) P(4)}
		+\tfrac{3 P(5) P(6) P(2)^2 P(4) }{P(1) P(3)}
		\\&\quad		
		+\tfrac{10 P(6)^2 P(2) P(3) P(4) }{P(5) P(1)}
		+\tfrac{11 P(5)^2 P(6) P(2) P(3)^2 }{P(1) P(4)^2}
		+\tfrac{10 P(5) P(6)^2 P(1) P(4) }{P(2) P(3)}
		+\tfrac{11 P(5) P(6)^2 P(2)^2 }{P(1) P(4)}
		\\&\quad
		+\tfrac{P(5)^2 P(2) P(4)^2 }{P(6) P(1)}
		+\tfrac{8 P(5) P(6)^3 P(1) }{P(2) P(4)}
		+\tfrac{12 P(6) P(2)^3 P(4) }{P(1)^2}
		+\tfrac{4 P(5) P(3) P(4)^2 }{P(2)}
		+\tfrac{10 P(5) P(6) P(3)^2 }{P(2)}
		\\&\quad
		+\tfrac{8 P(3)^2 P(4)^2 }{P(1)}
		+\tfrac{10 P(6) P(3)^3 }{P(1)}
		+{\scriptstyle P(5)^3}
		+{\scriptstyle 10 P(6)^2 P(4)}		
		+\tfrac{3 q^{13} P(5) P(1) P(4)^2 }{P(3)}
		+\tfrac{2 q^{13} P(5) P(3)^2 P(4) }{P(6)}
		\\&\quad
		+\tfrac{q^{13} P(5) P(2)^3 }{P(1)}
		+\tfrac{10 q^{13} P(5)^2 P(2)^2 P(3)^2 }{P(6) P(1) P(4)}
		+\tfrac{8 q^{13} P(5)^2 P(1) P(3) P(4) }{P(6) P(2)}
		+\tfrac{8 q^{13} P(2)^2 P(3) P(4)^2 }{P(5) P(1)}
		\\&\quad
		+{\scriptstyle 3 q^{13} P(6) P(2) P(3)}
		+{\scriptstyle 2 q^{13} P(5) P(6) P(1)}
		+{\scriptstyle 12 q^{13} P(2) P(4)^2}
		+\tfrac{10 q^{26} P(2)^2 P(3) P(4) }{P(6)}
		+\tfrac{3 q^{26} P(6) P(1) P(2)^2 }{P(3)}
		\\&\quad
		+\tfrac{3 q^{26} P(1)^2 P(3) P(4) }{P(2)}
		+\tfrac{4 q^{26} P(5) P(1) P(2) P(4) }{P(6)}
		+\tfrac{8 q^{26} P(5) P(1) P(2) P(3) }{P(4)}
		+\tfrac{3 q^{26} P(2)^2 P(3)^2 }{P(4)}
		\\&\quad
		+\tfrac{6 q^{39} P(6) P(1)^2 P(2) P(3) }{P(5) P(4)}
	,\\	
	B'_{13,11}(q^{13})
	&=
		\tfrac{6 P(6)^2 P(4)^2 }{q^{13} P(1)}
		+\tfrac{7 P(5)^3 P(4) }{q^{13} P(1)}
		+\tfrac{8 P(5) P(6) P(2)^2 P(4)^2 }{q^{13} P(1)^2 P(3)}
		+\tfrac{P(6)^2 P(2) P(3) P(4)^2 }{q^{13} P(5) P(1)^2}
		+\tfrac{8 P(5)^2 P(6) P(2) P(3)^2 }{q^{13} P(1)^2 P(4)}
		\\&\quad		
		+\tfrac{7 P(5)^2 P(6) P(1) P(4) }{P(2) P(3)}
		+\tfrac{6 P(5)^2 P(6) P(2)^2 }{P(1) P(4)}
		+\tfrac{9 P(6) P(2) P(3) P(4) }{P(1)}
		+\tfrac{6 P(5)^3 P(2) P(3)^2 }{P(1) P(4)^2}
		+\tfrac{9 P(5)^2 P(3) P(4)^2 }{P(6) P(2)}
		\\&\quad
		+\tfrac{P(6)^2 P(3)^2 P(4) }{P(5) P(2)}
		+\tfrac{10 P(5) P(2)^3 P(4) }{P(1)^2}
		+\tfrac{11 P(5) P(6)^2 P(3) }{P(4)}
		+\tfrac{4 P(2) P(4)^3 }{P(1)}
		+\tfrac{11 P(6)^3 P(2) }{P(3)}
		+\tfrac{8 P(5) P(4)^3 }{P(3)}
		\\&\quad
		+\tfrac{7 P(5) P(3)^3 }{P(1)}		
		+{\scriptstyle 2 P(5) P(6) P(4)}
		+\tfrac{7 q^{13} P(5) P(2) P(4)^2 }{P(6)}
		+\tfrac{11 q^{13} P(1) P(3) P(4)^2 }{P(2)}
		+\tfrac{10 q^{13} P(2)^2 P(3)^2 }{P(1)}
		\\&\quad
		+\tfrac{6 q^{13} P(3)^2 P(4)^2 }{P(5)}
		+\tfrac{9 q^{13} P(6) P(2)^3 P(4) }{P(5) P(1)}
		+{\scriptstyle 9 q^{13} P(5) P(2) P(3)}
		+{\scriptstyle 10 q^{13} P(5)^2 P(1)}
		+\tfrac{8 q^{26} P(1) P(3)^2 P(4) }{P(6)}
		\\&\quad
		+\tfrac{2 q^{26} P(6) P(1) P(2) P(3) }{P(5)}
		+{\scriptstyle 4 q^{26} P(2)^3}
		+\tfrac{q^{39} P(1)^2 P(2) P(3) }{P(4)}
		+\tfrac{q^{39} P(1) P(2)^2 P(3) P(4) }{P(5) P(6)}
	,\\
	B'_{13,12}(q^{13})
	&=
		\tfrac{P(5) P(6)^2 P(3) }{q^{13} P(1)}
		+\tfrac{4 P(6)^2 P(3)^2 P(4)^2 }{q^{13} P(5) P(1) P(2)}
		+\tfrac{11 P(6) P(2) P(3) P(4)^2 }{q^{13} P(1)^2}
		+\tfrac{2 P(5)^3 P(2) P(3)^2 }{q^{13} P(1)^2 P(4)}
		+\tfrac{8 P(5)^2 P(3)^2 P(4) }{q^{13} P(1) P(2)}
		\\&\quad		
		+\tfrac{8 P(5) P(6) P(2) P(3)^2 }{P(1) P(4)}
		+\tfrac{7 P(6) P(2)^3 P(4)^2 }{P(5) P(1)^2}
		+\tfrac{P(6) P(2)^2 P(4)^2 }{P(1) P(3)}
		+\tfrac{12 P(5) P(6) P(2) P(4)^2 }{P(3)^2}
		+\tfrac{7 P(5) P(2) P(3) P(4) }{P(1)}
		\\&\quad
		+\tfrac{3 P(6) P(3)^3 P(4) }{P(5) P(1)}
		+\tfrac{2 P(5) P(6)^2 P(2) }{P(3)}
		+\tfrac{3 P(5)^2 P(4)^3 }{P(6) P(3)}
		+\tfrac{3 P(5)^2 P(6) P(3) }{P(4)}
		+\tfrac{6 P(6) P(3)^2 P(4) }{P(2)}
		\\&\quad
		+\tfrac{12 P(5) P(1) P(4)^3 }{P(2)^2}
		+\tfrac{5 P(6)^3 P(1) }{P(2)}
		+\tfrac{7 P(6)^2 P(2)^2 }{P(1)}
		+{\scriptstyle 11 P(5)^2 P(4)	}
		+\tfrac{6 q^{13} P(5) P(2)^2 P(4) }{P(3)}
		+\tfrac{5 q^{13} P(5) P(6) P(2)^2 }{P(4)}
		\\&\quad		
		+\tfrac{10 q^{13} P(5)^2 P(2) P(3) }{P(6)}
		+\tfrac{11 q^{13} P(2)^3 P(4) }{P(1)}
		+\tfrac{7 q^{13} P(3)^2 P(4)^2 }{P(6)}
		+\tfrac{2 q^{13} P(5)^2 P(6) P(1) P(2) }{P(3) P(4)}
		\\&\quad
		+\tfrac{9 q^{13} P(5) P(1) P(3) P(4)^2 }{P(6) P(2)}
		+\tfrac{8 q^{13} P(6) P(2) P(3) P(4) }{P(5)}
		+{\scriptstyle 12 q^{13} P(6) P(1) P(4)}
		+\tfrac{3 q^{26} P(1) P(2) P(4)^2 }{P(6)}
		\\&\quad
		+\tfrac{9 q^{26} P(2)^2 P(3) P(4)^2 }{P(5) P(6)}
		+{\scriptstyle 4 q^{26} P(1) P(2) P(3)}
		+{\scriptstyle 7 q^{26} P(5) P(1)^2}
		+\tfrac{3 q^{39} P(1) P(2)^2 P(3)^2 }{P(6) P(4)}
	.	
\end{align*}
To finish the proof of (\ref{EqTheoremVMod13}) we must show that
$B'_{13}\equiv B_{13}\pmod{13}$. We do this by reducing each $B'_{13,i}$ modulo 
$13$ in the same way that we reduced the $A'_{13,i}$ in the proof of 
(\ref{EqTheoremUMod13}). Doing so gives that the $B'_{13,i}$ reduce modulo $13$
to the $B_{13,i}$ defined in Theorem \ref{TheoremMain}.

\end{proof}

\section{Remarks}

These functions satisfy additional congruences. For example, it appears that
\begin{align*}	
	u(9n) \equiv 0 \pmod{9}
	,\\	
	v(9n+1) \equiv 0 \pmod{9}
	,\\
	v(27n+1) \equiv 0 \pmod{27}
	.
\end{align*}
One could use the methods of this article to prove these congruences, however 
the difficulty in using these methods when $\ell$ is not a prime is that we can
not as easily reduce $E(1)^{-3}$ modulo $\ell$. In particular, with $\ell=9$,
it is not the case that $\frac{1}{E(1)^3}\equiv\frac{E(1)^6}{E(9)}$, but rather
we instead have from $\frac{E(1)^9}{E(3)^3}\equiv 1\pmod{9}$ that
\begin{align*}
	\frac{1}{E(1)^3}\equiv \frac{E(1)^6 E(3)^6}{E(9)^3}\pmod{9}.
\end{align*}
One would first have to find a short representation of the $9$-dissection
taken modulo $9$, or simply deal with the bothersome number of terms introduced by
expanding this with (\ref{EqECubedModL}). One could also verify the required 
terms are zero modulo $\ell$ by
viewing the eta quotients and generalized eta quotients as modular forms,
however that is also rather computational proof.

\bibliographystyle{abbrv}
\bibliography{partitionQuadruplesWithMod13Ref}

\end{document}